\documentclass{amsart}
\usepackage{euscript,graphicx,epstopdf,amscd,amsgen,amsfonts,amssymb,latexsym,amsmath,amsthm,graphicx,mathrsfs,times,color,overpic,bbm, array}

\usepackage{fourier}

\usepackage{contour}
\contourlength{1pt}

\usepackage{enumitem}
\usepackage{fourier}
 \usepackage[colorlinks=true,backref=page]{hyperref}
\usepackage{xcolor}
\definecolor{forestgreen}{rgb}{0.0, 0.27, 0.13}

\newtheorem{maintheorem}{Theorem}
\newtheorem{maincorollary}[maintheorem]{Corollary}
\newtheorem*{Mthm}{Main Theorem}
\newtheorem{theorem}{Theorem}[section]
\newtheorem{lemma}[theorem]{Lemma}
\newtheorem{corollary}[theorem]{Corollary}
\newtheorem{proposition}[theorem]{Proposition}
\newtheorem{claim}[theorem]{Claim}
\newtheorem{scholium}[theorem]{Scholium}

\theoremstyle{definition}

\newtheorem{remark}[theorem]{Remark}
\newtheorem{definition}[theorem]{Definition}

\newtheorem{notation}[theorem]{Notation}

\def\DD{\mathcal D}
\def\mcE{\mathcal E}

\def\cE{\EuScript{P}^\ast}
\def\cH{\mathcal{H}}
\def\cL{\mathcal{L}}
\def\cN{\EuScript{N}}
\def\cM{\EuScript{M}}

\def\cP{\EuScript{P}}

\def\cW{\EuScript{W}}
\def\cV{\EuScript{V}}
\def\bN{\mathbb{N}}

\def\bR{\mathbb{R}}

\def\bX{\mathbb{X}}

\def\cU{\mathcal{U}}
\def\cS{\mathcal{S}}

\def\cs{\mathrm{cs}}
\def\cu{\mathrm{cu}}
\def\u{\mathrm{u}}
\def\s{\mathrm{s}}
\def\ss{\mathrm{ss}}
\def\uu{\mathrm{uu}}
\def\c{\mathrm{c}}
\def\b{\text{\rm blend}}

\def\h{\mathfrak{h}}

\def\fL{\ell^\sharp}
\def\ud{\,d}
\def\im{\operatorname{Im}}
\def\e{\varepsilon}
\def\sX{\mathfrak{X}}

\def\blender{\mathfrak{B}}

\DeclareMathOperator{\length}{length}
\DeclareMathOperator{\inrad}{inrad}

\DeclareMathOperator{\Per}{Per}

\DeclareMathOperator{\card}{card}
\DeclareMathOperator{\interior}{int}

\DeclareMathOperator{\Image}{Im}

\DeclareMathOperator{\PH}{PH}

\DeclareMathOperator{\RTPH}{RT}
\DeclareMathOperator{\BM}{BM}
\DeclareMathOperator{\MB}{BM}

\DeclareMathOperator{\Diff}{Diff}
\DeclareMathOperator{\Homeo}{Homeo}

\DeclareMathOperator{\Leb}{Leb}
\title[The amount of nonhyperbolicity]{The amount of nonhyperbolicity  for partially hyperbolic diffeomorphisms}

\author[L.~J.~D\'iaz]{Lorenzo J. D\'\i az}
\address{Departamento de Matem\'atica PUC-Rio, Marqu\^es de S\~ao Vicente 225, G\'avea, Rio de Janeiro 22451-900, Brazil}
\email{lodiaz@mat.puc-rio.br}

\author[K.~Gelfert]{Katrin~Gelfert}
\address{Instituto de Matem\'atica Universidade Federal do Rio de Janeiro, Av. Athos da Silveira Ramos 149, Cidade Universit\'aria - Ilha do Fund\~ao, Rio de Janeiro 21945-909,  Brazil}\email{gelfert@im.ufrj.br}

\author[J.~Zhang]{Jinhua~Zhang}
\address{School of Mathematical Sciences, Beihang University\\ 100191, Beijing, P.R. China}
\email{jinhua$\_$zhang@buaa.edu.cn}
 
\thanks{This research has been supported [in part] by 
CAPES -- Finance Code 001, by 
CNPq-grants  310069/2020-3 and 
305327/2022-4,  
CNPq Projeto Universal 430154/2018-6  and 
404943/2023-3, 
E-16/2014 INCT/FAPERJ and  
E-26/200.371/2023 CNE/FAPERJ (all Brazil). 
J. Zhang was partially supported by National Key R\&D Program of China (2022YFA1005801), National Key R\&D Program of China (2021YFA1001900), NSFC 12001027 and the Fundamental Research Funds for the Central Universities (all China).
The authors thank their home institutions for their hospitality while preparing this paper. They also thank A. Tahzibi and J. Yang for helpful comments. Especial thanks to B. Santiago for discussions during early stages of this paper.}


\newcommand{\eqdef}{\stackrel{\scriptscriptstyle\rm def}{=}}

\numberwithin{equation}{section}

\begin{document}

 \begin{abstract}
 We study the amount of nonhyperbolicity within a broad class of (nonhyperbolic) partially hyperbolic diffeomorphisms with a one-dimensional center. For that, we focus on the center Lyapunov exponent and the entropy of its level sets. We show that these entropies vary continuously and can be expressed in terms of restricted variational principles. In this study, no dynamical coherence is required. 
  
Of particular interest is the case where the exponent is zero. To study this level set, we construct a compact set  foliated by curves tangent to the central direction.  Within this set, the entropy attains the maximal possible (and positive) value. Moreover, finite-time Lyapunov exponents converge uniformly to zero. In this construction, we introduce a mechanism to concatenate center curves.

The class studied consists of those robustly transitive diffeomorphisms that have a pair of blender-horseshoes with different types of hyperbolicity and possess minimal strong stable and unstable foliations. This classes includes flow-type and circle-fibered diffeomorphisms as well as some derived from Anosov diffeomorphisms. It also includes the so-called anomalous examples which are dynamically incoherent. 
 \end{abstract}

\maketitle

\tableofcontents

\section{Introduction}

We describe the ``amount of nonhyperbolicity'' of several classes of robustly transitive and nonhyperbolic diffeomorphisms. This is done in terms of entropy and restricted variational principles. Here we deal with those that are partially hyperbolic with a one-dimensional center bundle. For that, we develop a thermodynamical formalism suited for this setting. 

Central to the thermodynamical formalism are (1) Lyapunov exponents (a special type of Birkhoff averages), (2) level sets of points with equal exponent, and (3) their metric and topological entropies. Frequently, the interrelation between these quantifiers is expressed through pressure functions and variational principles. In hyperbolic settings, one of the striking features of this formalism, one may highlight, is that there are uncountably many highly irregular, intertwined, fractal level sets whose quantifiers change in an even analytic way and can be described by employing Gibbs measures. To show this, one heavily relies on symbolic dynamics (see, for example, \cite{Bow:08,Rue:04,PesWei:97}). Some of the major difficulties to go beyond uniform hyperbolicity come from the non-existence of symbolic codings or lack of specification-like properties up to the occurrence of nonhyperbolic measures (for which ``Pesin theory'' is unavailable). 

To go beyond uniform hyperbolicity is a challenge, which this paper addresses. To do so, a minimum amount of hyperbolicity seems natural to require. Following the motto of Pugh--Shub ``a little hyperbolicity goes a long way'' \cite{PugShu:97}, here we assume \emph{partially hyperbolicity with a one-dimensional center bundle}. According to \cite[Preface]{BonDiaVia:05}, this is a noncritical setting where two types of hyperbolicity are ``intermingled'' and ``interact'' and share a common center bundle. The case of higher-dimensional centers is beyond scope. We will see that there is a huge class of examples that fall into the setting studied here.

To be more precise, assume that $f\in\Diff^1(M)$ is \emph{partially hyperbolic} with a $Df$-in\-variant splitting with three nontrivial directions $TM=E^\ss \oplus E^\c\oplus E^\uu$, where $E^\ss$ and $E^\uu$ are the uniformly contracting and expanding bundles, respectively, and $E^\c$ is the one-dimensional center bundle. Recall that there exist (uniquely defined) foliations tangent to the bundles $E^\ss$ and $E^\uu$, called \emph{strong stable} and \emph{strong unstable foliations} and denoted by $\cW^\ss$ and $\cW^\uu$, respectively, \cite{BriPes:74,HirPugShu:77}. Our hypotheses require that $f$ belongs to the class $\BM^1(M)$ of maps satisfying: the simultaneous minimality of strong foliations and the existence of two blender-horseshoes (for discussion see \cite{BonCroDiaWil:16}) of different types of hyperbolicity. These conditions provide precisely the ``interaction'' between two regions with different types of hyperbolicity mentioned above. We observe that these hypotheses are fairly common among partially hyperbolic diffeomorphisms which are nonhyperbolic and robustly transitive.
We emphasize that our methods do not require \emph{dynamical coherence}, that is, the existence of invariant foliations tangent to $E^\ss \oplus E^\c $ and $E^\c  \oplus E^\uu$ (and hence, the existence of a center foliation tangent to $E^\c$). We neither require the existence of a center foliation tangent to $E^\c$. 
Dynamically coherent examples of diffeomorphisms satisfying our hypotheses include 
flow-type ones (this includes perturbations of the time-one map of transitive Anosov flows), 
circle-fibered ones (partially hyperbolic diffeomorphisms with a one-dimensional compact center foliation), 
and some DA (derived from Anosov) diffeomorphisms. Dynamically incoherent ones include the so-called \emph{anomalous} examples in \cite{BonGogHamPot:20}. See Section \ref{ssec:robusts} for further details.

Under the above partially hyperbolic hypothesis, the center bundle is among those provided by the Oseledets splitting \cite{Ose:68} and the absence of hyperbolicity is reflected by a zero Lyapunov exponent associated with this bundle. Three distinct types of nonhyperbolic behavior can be identified:
\begin{itemize}
	\item[(P1)] occurrence of ergodic measures with zero Lyapunov exponent,
	\item[(P2)] existence of points having zero Lyapunov exponent,
	\item[(P3)] existence of ergodic measures whose (nonzero) Lyapunov exponents converge to 0.
\end{itemize}
Clearly, (P1) implies (P2), but the other implications \emph{a priori} do not hold. For $f\in\BM^1(M)$, all three properties simultaneously hold (see \cite{BocBonDia:16,BonZhan:19,DiaGelSan:20,YanZha:20}) and therefore we are in a genuinely nonhyperbolic setting.

To state our main result, let us introduce some notations. Consider the function
\begin{equation}\label{def:varphi}
	\varphi^\c\colon M\to\bR,\quad
	\varphi^\c(x)\eqdef\log\,\lVert D^\c f(x)\rVert,
	\quad\text{ where }\quad
	D^\c f(x)
	\eqdef Df|_{E^\c(x)},
\end{equation}
which, by our hypotheses, is continuous. A point $x\in M$ is \emph{$\varphi^\c$-regular} if
\[
\chi^\c(x)
\eqdef \lim_{n\to\pm\infty}\frac1n\sum_{k=0}^{n-1}\varphi^\c\circ f^k(x)
= \lim_{n\to\pm\infty}\frac1n\log\,\lVert D^\c f^n(x)\rVert.
\]
Given $\alpha$, consider the \emph{level set} of points with exponent $\alpha$,
\begin{equation}\label{defLevelalpha}
	\cL(\alpha)
	\eqdef \big\{x \in M\colon 
	\chi^\c(x)=\alpha\big\}.
\end{equation}
One key feature of our setting is that these sets are nonempty for a closed interval of values $\alpha$ containing zero in its interior. 

By the Birkhoff ergodic theorem, given $\mu\in\cM_{\rm erg}(f)$ (the set of $f$-ergodic measures), $\mu$-almost every point $x$ is $\varphi^\c$-regular and verifies
\[
\chi^\c(x)
= \chi^\c(\mu)
\eqdef \int\varphi^\c\,d\mu.
\]
We call this number the \emph{(center) Lyapunov exponent} of $\mu$. We can also consider the ``level sets'' of ergodic measures having exponent $\alpha$ and the numbers
\begin{equation}\label{e:Halpha}
\cH(\alpha)
	\eqdef \sup\big\{h_\mu(f)\colon\mu\in\cM_{\rm erg}(f),\chi^\c(\mu)=\alpha\big\}.
\end{equation}	
Let
\begin{equation}\label{eq:defalphamin}
	\alpha_{\rm min}^f
	\eqdef \inf_{\mu\in\cM_{\rm erg}(f)}\chi^\c(\mu),
	\quad
	\alpha_{\rm max}^f
	\eqdef \sup_{\mu\in\cM_{\rm erg}(f)}\chi^\c(\mu).
\end{equation}
If the map $f$ is clear from the context, we drop the super-index ${}^f$ from the notation. It is straightforward to check that the above infimum and supremum are both attained by ergodic measures (taking limit measures and considering  ergodic decompositions).

Our main theorem extends classical results in multifractal analysis for uniformly hyperbolic dynamics. For that, as a quantifier, we use the concept of topological entropy on a set $X\subset M$, $h_{\rm top}(f,X)$, introduced by Bowen \cite{Bow:73}.

\begin{Mthm}
Let $f\in\BM^1(M)$. Then the function $\alpha\in[\alpha_{\rm min},\alpha_{\rm max}]\mapsto h_{\rm top}(f,\cL(\alpha))$ is continuous and 
\begin{equation}\label{newoutofmind}
	0	
	< h_{\rm top}(f,\cL(0))
	= \lim_{\alpha\nearrow0}\cH(\alpha)
	= \lim_{\beta\searrow0}\cH(\beta).
\end{equation}
Moreover, there is a set $\Lambda(0)\subset\cL(0)$ consisting of non-degenerate curves tangent to the center bundle $E^\c$ such that 
\begin{itemize}[ leftmargin=0.7cm ]
\item[(1)] $h_{\rm top}(f,\Lambda(0))=h_{\rm top}(f,\cL(0))$, 
\item[(2)] $\lim_{n\to\pm\infty}\frac{1}{n}\log\|D^{\c}f^n(x)\|=0$, uniformly on $x\in \Lambda(0)$.
\end{itemize}
\end{Mthm}

This theorem is a consequence of Theorems~\ref{theorem1} and \ref{theorem2} stated in Section~\ref{ssec:contextresults}. In Theorem~\ref{theorem1}, we deal with the continuity of the entropy spectrum for the hyperbolic part (i.e. for $\alpha\neq 0$) and in Theorem~~\ref{theorem2}, we deal with the nonhyperbolic part (i.e. at $\alpha=0$).

 In the particular case of step skew-products with symbolic base dynamics (a full shift in finitely many symbols) and circle fibers, a precursor of this theorem is given in  \cite{DiaGelRam:19}. The proofs in \cite{DiaGelRam:19} strongly rely on a symbolic description that is unavailable here. Moreover, we deal with dynamics that do have a one-dimensional center bundle, but this bundle may not give rise to a center foliation (see, for instance, the examples in \cite{BonGogHamPot:20}) nor a foliation by circles (for instance the perturbations of time one-map of geodesic flows on negatively curved surfaces). This is a major difficulty overcome in this paper and will be further discussed in Section~\ref{ssec:contextresults}.  
 
One of the main difficulties is to deal with points with zero Lyapunov exponent. By \cite{LedYou1:85,LedYou2:85}, from measure-theoretical viewpoint, when the center exponent is zero then any entropy is ``created in the strong foliations'' (see also \cite{Tah:21}). Inspired by this fact, we introduce a flexible symbol-free strategy working directly with the strong foliations. Our proofs combine two tools. The first one is the use of pressure functions (and their Legendre-Fenchel transforms) and gradual approximation by (uniformly hyperbolic) exhausting families. This part implements and extends the formalism introduced in \cite{DiaGelRam:17,DiaGelRam:19}. The second one is a method to concatenate center curves, which has intrinsic interest.  It is used to construct the set $\Lambda(0)$ in the Main Theorem. For that, we use the weakly integrable property which holds for partially hyperbolic diffeomorphisms with one-dimensional center (see \cite[Proposition 3.4]{BriBurIva:04}, where this property is called \emph{weakly dynamically coherent}). 

In Section \ref{ssec:contextresults}, we provide more detailed statements towards the Main Theorem as well as a discussion of the context and techniques. We also describe the organization of this paper.  

\section{Statement of results and context}\label{ssec:contextresults}

In Section \ref{ss.context}, we provide details on the partially hyperbolic context and define the set of maps $\BM^1(M)$ that we study. In Section \ref{ss.results}, we state our main results. In Section \ref{ss.concatenation}, we present the concatenation and outbranching principle for center curves. In Section \ref{ssec:robusts}, we discuss the relevance of the set $\BM^1(M)$ and present classes of maps in this set. Finally, in Section \ref{ssec:organization}, we explain the organization of this paper.

\subsection{Context}\label{ss.context}

We endow the space $\Diff^1(M)$ of $C^1$-diffeomorphisms of $M$ with the $C^1$-topology. Recall that $f\in\Diff^1(M)$ is \emph{partially hyperbolic}, if there exist a $Df$-invariant splitting $TM=E^\ss \oplus E^\c\oplus E^\uu $ and an integer $N\in\mathbb{N}$ such that for every point $x\in M$,   
\begin{itemize}
	\item $\|Df^N|_{E^\ss (x)}\|<1$ and $\|Df^{-N}|_{E^\uu (x)}\|<1$;
	\item $\max\big\{\|Df^N|_{E^\ss (x)}\|\cdot \|Df^{-N}|_{E^\c(f^N(x))}\|, \|Df^N|_{E^\c(x)}\|\cdot \|Df^{-N}|_{E^\uu (f^N(x))}\|\big\}<1.$
\end{itemize}
Up to changing the metric, one can assume that $N=1$ (see~\cite{Gou:07}) and we do so throughout this paper. Denote by $\PH^1_{\c =1}(M)$ the set of $C^1$-partially hyperbolic diffeomorphisms on $M$ with a one-dimensional center bundle. By definition, this set is open in $\Diff^1(M)$. In what follows, we assume $f\in\PH^1_{\c=1}(M)$.

Given any $f\in \PH^1_{\c =1}(M)$ its space of ergodic measures naturally splits as follows
\[
	\cM_{\rm erg}(f)
	= \cM_{\rm erg,<0}(f)\cup\cM_{\rm erg,>0}(f)\cup\cM_{\rm erg,0}(f),
\]
into the sets of measures with negative, positive, and zero center Lyapunov exponents, respectively. The former we call \emph{hyperbolic} ergodic measures of \emph{contracting} and \emph{expanding type}, respectively. The ones in $\cM_{\rm erg,0}(f)$ are \emph{nonhyperbolic}. \emph{A priori}, without further hypotheses, any of these components (and also any pair of them) can be empty. The most interesting case, and the focus of this paper, occurs when all three components are nonempty. This happens when $f\in\BM^1(M)$ as we proceed to describe.

We start introducing some terminology. Given a \emph{saddle}, that is, a hyperbolic periodic point, its \emph{$\u$-index} is the dimension of its unstable bundle. Any saddle has either $\u$-index $d^\uu\eqdef\dim(E^\uu)$ or $\u$-index $d^\uu+1$. \emph{Transitivity} of $f$ means that this map has a dense orbit. A property of a diffeomorphism is \emph{$C^1$-robust} if it holds in a neighborhood of it in $\Diff^1(M)$. We define $\RTPH^1(M)$ as the subset of $\PH^1_{\c =1}(M)$ of diffeomorphisms which are $C^1$-robustly transitive and have saddles of different $\u$-indices. By its very definition, 
 $\RTPH^1(M)$ is open in $\Diff^1(M)$ and every  $f\in \RTPH^1(M)$ is not hyperbolic. A key property of $\RTPH^1(M)$ is that there is a $C^1$-open and -dense subset formed by diffeomorphisms which possess \emph{stable} and \emph{unstable blender-horseshoes}, simultaneously (see Section \ref{ssec:minfoli}). Blender-horseshoes are important local plugs for obtaining topological and ergodic properties. However, they alone are insufficient for describing the (nonempty) sets $\cM_{\rm erg,0}(f)$ and $\cL(0)$. To achieve this one needs to ``dynamically connect"  these blender-horseshoes. This connection is obtained under the assumption of (simultaneous) minimality of the strong stable and unstable foliations. Recall that a foliation of $M$ is \emph{minimal} if every one of its leaves is dense in $M$. This approach is reminiscent of \cite{GorIlyKleNal:05}, where the systems have ``center-contracting'' and ``center-expanding'' regions with orbits going from one region to the other one.
This discussion provides the context for our two key hypotheses.

\begin{definition}[The set $\BM^1(M)$]\label{defr.ourhypotheses}
The set $\BM^1(M)$ consists of  diffeomorphisms $f\in  \PH^1_{\c =1}(M)$ satisfying the following two properties:
\begin{enumerate}[ leftmargin=1cm ]
	\item[(M)] the foliations $\cW^\uu$ and $\cW^\ss$ are both minimal,
	\item[(B)] there are $n^\pm\in\bN$ such that $f^{n_-}$ has a stable blender-horseshoe and $f^{n_+}$ has an unstable blender-horseshoe.
	\end{enumerate}
\end{definition}

We postpone the precise definitions of the blender-horseshoes in (B) to Section \ref{ss.bh}. The interrelation between properties (B) and (M), and their relation with robust transitivity are discussed in Sections~\ref{ssec:ocblen} and \ref{ssec:minfoli}. We discuss examples in Section \ref{ssec:robusts}.

\begin{remark}[$\BM^1(M)$ is $C^1$-open]\label{r.defr.ourhypotheses}
	Condition (B) is $C^1$-open, while condition (M) itself \emph{a priori} is not $C^1$-open. However, if (M) and (B) hold for a diffeomorphism, then both hold in any small $C^1$-neighborhood of it. This is essentially a consequence of \cite{BonDiaUre:02}, see also Lemma~\ref{l.robustness-of-BM}.	
	Note that $\BM^1(M)\subset\RTPH^1(M)$.
\end{remark} 

\begin{remark}[Topology of the space of ergodic measures]\label{BM1almpa}
	Given $f\in\BM^1(M)$, the occurrence of un-/stable blender-horseshoes implies that the spectrum $\{\chi^\c(\mu)\colon\cM_{\rm erg}(f)\}$ contains negative and positive values. Moreover, the interaction between saddles of different indices (provided by the minimality of the foliations and the blender-horseshoes) implies that this spectrum also contains the value 0. Moreover, the value $0$ is accumulated by negative and positive values (see \cite[Theorem A]{BonZhan:19} and also \cite{DiaGelSan:20,YanZha:20}). These properties have their counterpart on the space of ergodic measures: every measure in $\cM_{\rm erg,0}(f)$  can be approximated in the weak$\ast$-topology and in entropy by measures  in $\cM_{\rm erg,<0}(f)$ and in $\cM_{\rm erg,>0}(f)$ (see \cite{DiaGelSan:20,YanZha:20}). Moreover, there are measures in $\cM_{\rm erg,0}(f)$ with positive entropy \cite{BocBonDia:16}. 
\end{remark}

\subsection{Results}\label{ss.results}

The following result states a \emph{restricted variational principle} for the level sets in \eqref{defLevelalpha} in terms of their topological and metric entropies. Recall  the number  $\cH(\alpha)$ defined in \eqref{e:Halpha} and recall \eqref{eq:defalphamin}.

\begin{maintheorem}[Entropy spectrum]\label{theorem1}
Let $f\in\BM^1(M)$. The following properties hold:
\begin{enumerate}[label=$\bullet$, leftmargin=0.5cm ]
\item $\alpha_{\rm min}<0<\alpha_{\rm max}$ and 
\[
		\big\{\chi^\c(\mu)\colon \mu\in\cM_{\rm erg}(f)\big\}
		= [\alpha_{\rm min},\alpha_{\rm max}]. 
\]	
\item $\cL(\alpha)\ne\emptyset$ for every $\alpha\in[\alpha_{\rm min},\alpha_{\rm max}]$,  
\item the map $\alpha\mapsto h_{\rm top}(f,\cL(\alpha))$ is continuous and concave on $[\alpha_{\rm min},0) $ and $(0,\alpha_{\rm max}]$, respectively, and satisfies
\[
			h_{\rm top}(f,\cL(\alpha))
			= \cH(\alpha),\quad
			\alpha\in[\alpha_{\rm min},0)\cup(0,\alpha_{\rm max}].
\]
\end{enumerate}
Letting
\begin{equation}\label{resvp0}
		\h^-(f)\eqdef \lim_{\alpha\nearrow0}\cH(\alpha)
		\quad\text{ and }\quad
		\h^+(f)\eqdef	\lim_{\beta\searrow0}\cH(\beta)		,
\end{equation}
it holds	
\begin{equation}\label{outofmind}
	0	
	< \cH(0)
	\le h_{\rm top}(f,\cL(0))
	\le \h^-(f)
	= \h^+(f)\eqdef \h(f).
\end{equation}
\end{maintheorem}

Indeed, we will see in Theorem \ref{theorem2} that in \eqref{outofmind} we have equality $h_{\rm top}(f,\cL(0))=\h(f)$. 

Theorem \ref{theorem1} is proved in Section \ref{sec:proofThm1}. The proof involves the construction of so-called skeletons (see Section \ref{sepreskeleton}) and the Legendre-Fenchel transform of a restricted pressure function (see  Section \ref{sec:press}). 

\begin{figure}[h]
		\begin{overpic}[scale=.5]{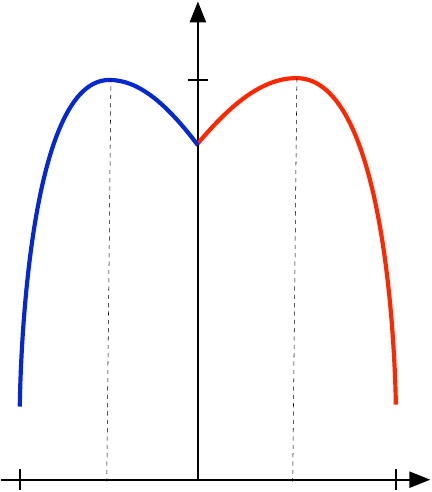}
			\put(43,93){\small{$h_{\rm top}(f,\cL(\alpha))$}}
			\put(0,90){\small{(a)}}
			\put(84,-5){\small{$\alpha$}}
			\put(16,-5){\small{$\alpha_{\rm max}^-$}}
			\put(56,-5){\small{$\alpha_{\rm max}^+$}}
		\end{overpic}\hspace{0.5cm}
		\begin{overpic}[scale=.5]{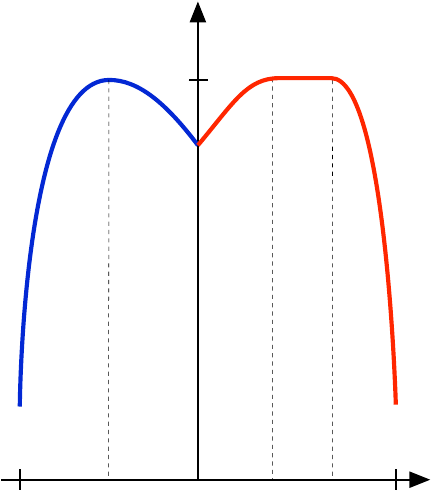}
			\put(43,93){\small{$h_{\rm top}(f,\cL(\alpha))$}}
			\put(0,90){\small{(b)}}
			\put(84,-5){\small{$\alpha$}}
			\put(16,-5){\small{$\alpha_{\rm max}^-$}}
			\put(52,-5){\small{$\alpha_1^+$}}
			\put(64,-5){\small{$\alpha_2^+$}}
		\end{overpic}\hspace{0.5cm}
		\begin{overpic}[scale=.5]{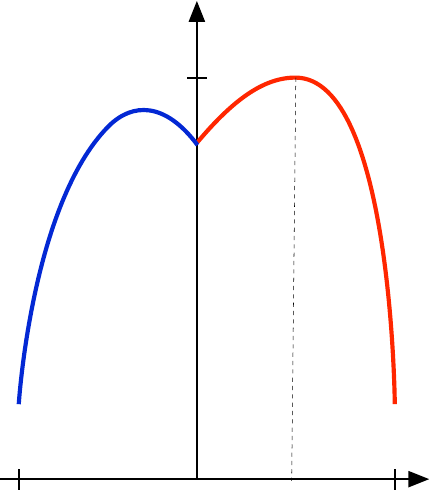}
			\put(43,93){\small{$h_{\rm top}(f,\cL(\alpha))$}}
			\put(0,90){\small{(c)}}
			\put(84,-5){\small{$\alpha$}}
			\put(56,-5){\small{$\alpha_{\rm max}^+$}}
		\end{overpic}\vspace{1cm}
		\begin{overpic}[scale=.5]{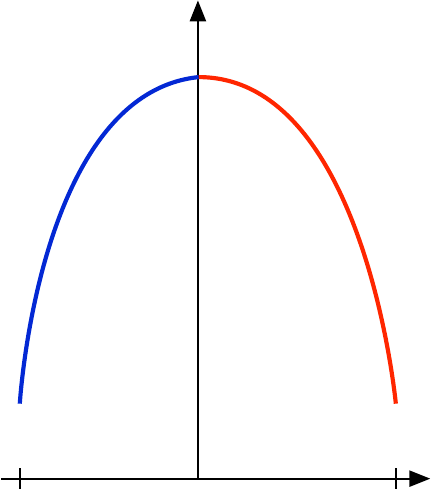}
			\put(43,93){\small{$h_{\rm top}(f,\cL(\alpha))$}}
			\put(0,90){\small{(d)}}
			\put(84,-5){\small{$\alpha$}}
		\end{overpic}\hspace{0.5cm}
		\begin{overpic}[scale=.5]{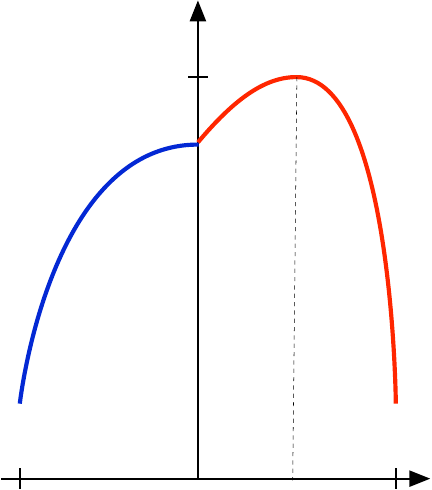}
			\put(43,93){\small{$h_{\rm top}(f,\cL(\alpha))$}}
			\put(0,90){\small{(e)}}
			\put(84,-5){\small{$\alpha$}}
			\put(56,-5){\small{$\alpha_{\rm max}^+$}}
		\end{overpic}
		\caption{Possible shapes of the function $\alpha\mapsto h_{\rm top}(f,\cL(\alpha))$}
		\label{fig.examples}
\end{figure}
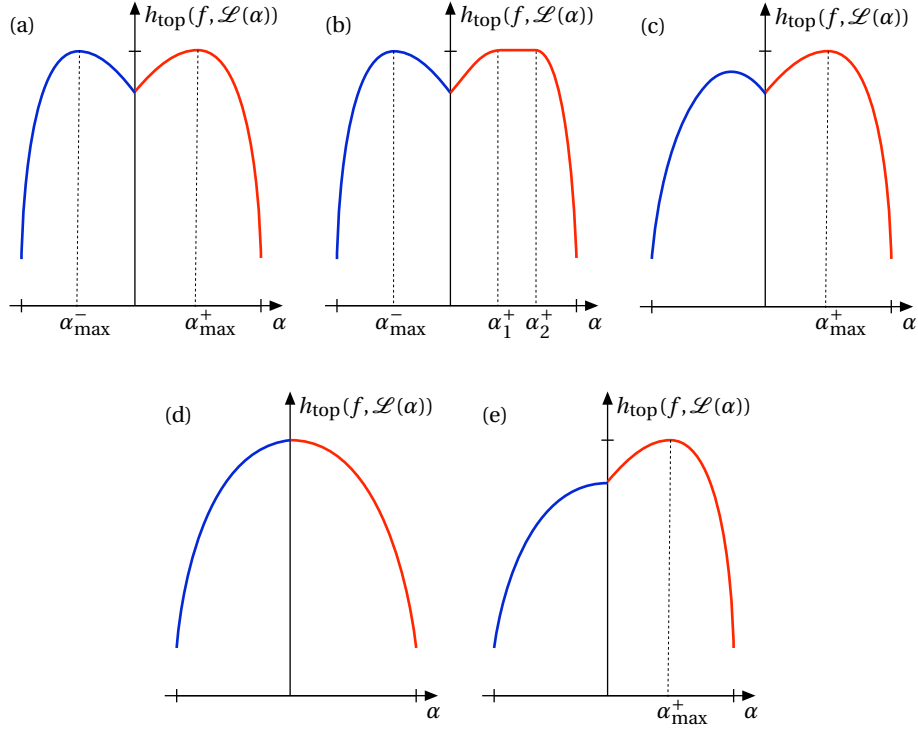

\begin{remark}[Entropy of $\cL(0)$]\label{rem:flipflop}
	The fact $h_{\rm top}(f,\cL(0))>0$ for $f\in \BM^1(M)$ was obtained in \cite{BocBonDia:16} by constructing an $f$-invariant compact  set with positive topological entropy supporting only nonhyperbolic measures.  See Remark~\ref{f.flipblender} for further discussions. In our setting, the property $h_{\rm top}(f,\cL(0))>0$ can also be obtained by using the Legendre-Fenchel transform, see Remark~\ref{remLFTransforFigures}. 
	
		The methods in this paper are insufficient to show $h_{\rm top}(f,\cL(0))<h_{\rm top}(f)$ for $f\in \BM^1(M)$ (compare for example the discussion of general properties of Legendre-Fenchel transforms in \cite[Remark 2.2]{DiaGelRam:19}). Note that if $h_{\rm top}(f,\cL(0))<h_{\rm top}(f)$, then every ergodic measure with entropy larger than $h_{\rm top}(f,\cL(0))$ is hyperbolic. We discuss this further in Remark \ref{remshapesofHbis}. 
Moreover, the methods here are also insufficient to show that $h_{\rm top}(f,\cL(0))=\cH(0)$.
\end{remark}

\begin{remark}[Possible shapes of $\alpha\mapsto h_{\rm top}(f,\cL(\alpha))$ and measures of maximal entropy  for $f\in \BM^1(M)$]\label{remshapesofHbis}
Let us briefly discuss the possible shapes of the graphs in Figure \ref{fig.examples} (further details will be given in Remark \ref{remLFTransforFigures}). The different cases depend essentially on the nature of the measures of maximal entropy\footnote{An ergodic measure is called a \emph{measure of maximal entropy} if its metric entropy coincides with the topological entropy of the dynamics.} (\emph{mme}, for short). First note that such measures always exist in our setting (see \cite{CowYou:05}). More information is available when the diffeomorphism is $C^2$ which is assumed throughout this remark if not stated otherwise.  In general, for $f\in\BM^2(M)$, as all hyperbolic periodic points of the same index are homoclinically related (for details see Section \ref{sechomcla}), there is at most one mme with negative center Lyapunov exponent and at most one mme with positive center Lyapunov exponent. See the discussion in \cite[Section 1.6]{BuzCroSar:22} and  \cite[Theorem 1.5 ]{BenOva:16}. This rules out Figure \ref{fig.examples} (b). Note that we may \emph{a priori} have only nonhyperbolic mme's.

Figure \ref{fig.examples} (a)--(c) depicts the cases when any measure with large enough entropy must be hyperbolic. Such scenarios occur in our setting when the map belongs either to the class of
circle-fibered diffeomorphisms or to the class of flow-type diffeomorphisms discussed in Section~\ref{ssec:robusts}, see \cite{TahYan:19} and \cite{BuzFisTah:22,CroPol:22}, respectively. In both cases, there are two mme's of different types of hyperbolicity; hence we are in Figure \ref{fig.examples} (a). The invariance principle \cite{Led:84,AviVia:10} and its variations play a key role in the proof of those results. 

Under the assumption that there exists a semi-conjugacy with an Anosov diffeomorphism (only $C^1$ differentiability is required), in  \cite{Ure:12,BuzFisSamVas:12} it is shown that there is a unique mme and that it is hyperbolic, the sign of its Lyapunov exponent depends on the Anosov factor. DA diffeomorphisms discussed in Section \ref{ssec:robusts} belong to this category. In these cases, we are either in Figure \ref{fig.examples} (c) or in Figure \ref{fig.examples} (e).

In Figure \ref{fig.examples} (b), there are several mme's with different positive exponents (that, by the above remarks, cannot happen for $f\in\BM^2(M)$). We do not have examples where this occurs for $f\in \BM^1(M)$. 

In Figure \ref{fig.examples} (d), any mme is nonhyperbolic. As shown in \cite{RodRodTahUre:12,TahYan:19,CroPol:22}, this would occur in rigid cases. In particular, in the class of circle-fibered diffeomorphisms they must be of rotation type (see \cite{RodRodTahUre:12,TahYan:19}), though this case would contradict the existence of blenders. See \cite{UreViaYan:21} for further results in the same direction. Another case is the flow-type class where the diffeomorphism must be the time-one map of a topological Anosov flow in \cite{BuzFisTah:22,CroPol:22}, which we cannot rule out \emph{a priori}. 

In the anomalous class in \cite{BonGogHamPot:20}, nothing about the hyperbolicity and uniqueness of mme's is known so far. 

Finally, we do not know if the shape in Figure \ref{fig.examples} (e) can occur.
\end{remark}

Let us describe how the entropies of level sets vary when perturbing the map $f$ (below we use superscripts to indicate the dependence of the quantifiers on $f$). 

\begin{maincorollary}[Lower semi-continuity of the entropy spectrum]\label{cor:semiconti}
	Let $f\in\BM^1(M)$ and consider $\alpha\in(\alpha_{\rm min}^f,0)$. Then there is a neighborhood $\cV\subset \BM^1(M)$ of $f$ such that for every $g\in \cV$ it holds $\alpha\in(\alpha_{\rm min}^g,0)$. Moreover, for every $\delta>0$ (possibly after shrinking $\cV$) it holds that  
\[
	h_{\rm top}(g,\cL^g(\alpha))\ge h_{\rm top}(f,\cL^f(\alpha))-\delta
	\quad\text{ for every $g\in \cV$. }
\]	
	
	The analogous result holds for $\alpha\in(0,\alpha_{\rm max}^f)$.
\end{maincorollary}

We prove Corollary \ref{cor:semiconti} in Section \ref{app2}.
For the next result, recall $\h(f)$ in \eqref{outofmind}.

\begin{maintheorem}\label{theorem2}
For every  $f\in \BM^1(M)$, one has that 
\[
	0 < \cH(0)\le h_{\rm top}(f,\cL(0))=\h(f).
\]
Moreover, there is a set $\Lambda(0)\subset\cL(0)$ formed by non-degenerate pairwise disjoint center curves such that 
\begin{itemize}[ leftmargin=0.7cm ]
\item[(1)] 
	$h_{\rm top}(f,\Lambda(0))
	=h_{\rm top}(f,\cL(0))
	=\h(f)$, 
\item[(2)] $\lim_{n\to\pm\infty}\frac{1}{n}\log\|D^{\c}f^n(x)\|=0$, uniformly on $x\in \Lambda(0)$.
\end{itemize}\end{maintheorem}

In Section \ref{ss.concatenation}, we explain the main heuristic underlying the proof of Theorem \ref{theorem2}: a concatenation of center curves. 

\subsection{Concatenation and outbranching of center curves}\label{ss.concatenation}

	 First, Theorem \ref{theorem1} guarantees that there are ergodic measures $\mu_k$ in $\cM_{\rm erg,<0}(f)$ whose center Lyapunov exponents converge to $0$. Each ergodic measure has a (plenty of) generic point $y_k$ with the corresponding exponent (the cardinality growing exponentially fast, according to its entropy). We fix a sequence $(y_k)_k$ of such points whose exponents tend to zero. We also choose for every $k$ some center curve $\zeta_k$ centered at $y_k$ and sufficiently small.
	 
	 Let us now start by concatenating two center curves, say $\zeta_1$ and $\zeta_2$.  We consider what we call a $\u$-strip $\Delta^\cu$ associated to $\zeta_1$, that is the union of local strong unstable leaves through $\zeta_1$. Analogously, we consider an $\s$-strip $\Delta^\cs$ associated to $\zeta_2$. Taking forward iterates of the former and backward iterates of the latter, we obtain images of them that enter the region of a stable blender-horseshoe. The number of iterates needed to enter only depends on the size of the local strong un-/stable leaves. Using the blender, after some further iterations, we can make both strips intersect through some center curve $\check\zeta$. Pulling back $\check\zeta$ to $\Delta^\cu$, we get a center curve $\widehat\zeta$ that stretches completely over the $\u$-strip of the initial curve $\zeta_1$. Moreover, $\widehat\zeta$ accompanies it during the initial iterations and eventually ends close to $\zeta_2$. Compare Figure \ref{fig.2} in Section \ref{sec:concat}.
	 
	 This concatenation can now be applied to the sequence of center curves  $(\zeta_k)_k$ from above. In the limit, we get a center curve $\zeta_\infty^+$ with forward exponent $0$.
	 
	 To produce \emph{many} center curves $\zeta_\infty^+$ as above (which form a set of positive topological entropy), at each step we choose between exponentially many curves from the next step. This technique is somewhat reminiscent of the one of ``outbranching sequences'' in \cite[Section 8]{DiaGelRam:19}. However, here we perform this in a much more efficient and transparent way. In this process, we get a fractal set produced as a limsup set, which is fibered by center curves of uniform length and entirely contained in some initial $\u$-strips. Its entropy is estimated by employing an entropy distribution principle. The underlying idea is that ``entropy is seen in the strong unstable leaves only''.  	
	 
	 The concatenation produces only a set with \emph{forward} exponent zero. Using once more a concatenation for $f^{-1}$ (this involves a sequence of curves with positive exponents tending to zero and an unstable blender-horseshoe), one gets the set $\Lambda(0)$.

\subsection{The set  $\BM^1(M)$}\label{ssec:robusts}
In this section, we discuss the definition of $\BM^1(M)$ and provide examples.
\subsubsection{Robustly transitive nonhyperbolic diffeomorphisms}\label{sss.rtnh}

The first examples of robustly transitive diffeomorphisms were obtained in \cite{Shu:71} on $\mathbb{T}^4$. Such examples are partially hyperbolic with a two-dimensional center. Thereafter, \cite{Man:78} provided the first examples in $\RTPH^1(M)$  (see Section \ref{ss.context}) falling in the category of so-called DA diffeomorphisms on  $\mathbb{T}^3$. New methods for constructing maps in $\RTPH^1(M)$ were provided in \cite{BonDia:96} using blenders (a precursor of the blender-horseshoes with less structure). Key ingredients in \cite{BonDia:96} are the coexistence of some saddles of different $\u$-indices and the density of their invariant manifolds (the latter is a correlative of the minimality of the strong foliations). These methods were used to prove that appropriate perturbations of product maps $F\times\mathrm{Id}$ on $M_1\times \mathbb{T}^1$, where $F$ is a $C^1$-Anosov diffeomorphism on $M_1$, belong to $\RTPH^1(M_1\times\mathbb{T}^1)$. Further classes are obtained in \cite{BonDia:96} considering manifolds $M$ supporting transitive Anosov flows $\varphi^t$. It is proved that suitable small $C^1$-perturbations of the time-$T$ map $\varphi^T$ belong to $\RTPH^1(M)$. All these examples are dynamically coherent. A common feature of the latter two sorts of examples is the existence of circles tangent to the center bundle. In the perturbations of product maps, there is a center foliation formed by circles. In the case of time-$T$ maps, the circles arise from the periodic orbits of the flow. We refer to these classes of examples as \emph{circle-fibered diffeomorphisms} and \emph{diffeomorphisms of flow-type} (following the terminology in \cite{BuzFisTah:22}), respectively. 

Examples of a rather different nature are the \emph{anomalous diffeomorphisms} obtained in \cite{BonGogHamPot:20} by considering maps in the unitary tangent bundle $M=T^1S$ of a surface $S$ of genus $g \ge 2$. For these maps the union of the circles that are tangent to the center direction is dense in $M$, however, none of those circles is periodic. These examples are not dynamically coherent. Another sort of examples can be found in \cite{Yan:21} perturbing accessible volume-preserving diffeomorphisms whose volume measure is hyperbolic. 

\subsubsection{Occurrence of  blender-horseshoes}\label{ssec:ocblen}

By the \emph{connecting lemmas} \cite{Hay:97,WX:00,BonCro:04}, robust transitivity implies 
that for any pair of saddles of different $\u$-indices of $f \in \RTPH^1(M)$, there is $g$ arbitrarily $C^1$-close to $f$ with a \emph{heterodimensional cycle} associated to them (i.e., the invariant manifolds of the saddles intersect cyclically). Therefore, the set of diffeomorphisms with heterodimensional cycles is $C^1$-dense in $\RTPH^1(M)$. These cycles are associated with saddles of $\u$-indices $d^\uu$ and $d^\uu+1$ (they have co-index $1$). Moreover, partial hyperbolicity implies that these saddles have real center eigenvalues (i.e., the eigenvalues associated with $E^\c$). After arbitrarily small $C^1$-perturbations, this sort of heterodimensional cycle yields blender-horseshoes (for some iterate of the map). They can be chosen as either stable or unstable, according to the type of perturbation, see  \cite[Theorems 3.3 and 4.1]{BonDia:08}. Even though the terminology blender-horseshoe was not used in \cite{BonDia:08}, the construction corresponds exactly to the blender-horseshoes in \cite[Section 5.1]{BonDia:12}. By \cite[Lemma 3.9]{BonDia:12}, the existence of (stable or unstable) blender-horseshoes is a $C^1$-open property. Therefore, property (B) in Definition \ref{defr.ourhypotheses} is $C^1$-open and -dense in  $\RTPH^1(M)$. 

\subsubsection{Minimality of the strong foliations}\label{ssec:minfoli}

The results in \cite{BonDiaUre:02,RodRodUre:07}
provide a criterion for a map $f\in\RTPH^1(M)$ having a minimal strong unstable foliation (respectively, minimal strong stable foliation), involving the existence of so-called complete $\u$-sections (respectively, complete $\s$-sections). This criterion can be reformulated using blender-horseshoes and also provides robust minimality of the strong foliations. A hyperbolic-like condition guaranteeing the minimality of a strong foliation can be found in \cite{PujSam:06}.
 
We are interested in settings where both foliations are simultaneously minimal. For that consider the subset $\RTPH^1_0(M)$ of diffeomorphisms in $\RTPH^1(M)$ which have some compact center leaf. By normal hyperbolicity, this set is open. It is shown in \cite{BonDiaUre:02,RodRodUre:07} that $\BM^1_0(M)$ is open and dense in $\RTPH^1_0(M)$. Note that the circle-fibered, the flow-type, and the anomalous diffeomorphisms belong to $\RTPH^1_0(M)$.
 
In the case of a three-dimensional DA diffeomorphism, by \cite{BonDiaUre:02,RodRodUre:07}, $C^1$-open and -densely at least one of the strong foliations is minimal. \cite[Theorem B]{RodUreYan:22} provides examples of DA diffeomorphism, where both strong foliations are minimal and therefore are in $\BM^1(\mathbb{T}^3)$.

Finally, observe that the list of examples  in $\BM^1(M)$ given above are mainly in dimension three, but this is not a requisite in our results. In Table \ref{defaulttable} we collect the properties of the known examples in $\BM^1(M)$ in dimension three.

\begin{tiny}
\begin{table}[h]\caption{Properties of known examples in $\BM^1(M)$}
\begin{center}
\begin{tabular}{|m{1.5cm}|c|c|c|m{2.5cm}|}
\hline
 & compact center leaves & periodic compact center leaves & dynamically coherent & some references \\
\hline
some DA 	& no & no & yes & \cite{RodUreYan:22} \\
\hline
 fibered by circles & yes & yes & yes & \cite{BonDiaUre:02,RodRodUre:07}\\
\hline
flow-type & yes & yes & yes & \cite{BonDiaUre:02,RodRodUre:07,BuzFisTah:22}\\
\hline
anomalous & yes & no & no & \cite{BonGogHamPot:20}\\
\hline
\end{tabular}
\end{center}
\label{defaulttable}
\end{table}%
\end{tiny}

\subsection{Organization of the paper}\label{ssec:organization}

This paper is organized as follows. 

In Section \ref{sepreskeleton}, we study pre-skeletons that collections of orbit segments which are useful to approximate by basic sets in terms of entropy and Lyapunov exponents. In Section \ref{sec:press}, we study pressure functions and their Legendre-Fenchel transforms. In particular, in Section \ref{sec:proofThm1}, we prove Theorem~\ref{theorem:1} concerning restricted variational principles for hyperbolic values in the interior of the spectrum for maps in $\BM^1(M)$. In Section~\ref{sec:proofThm1}, we obtain an upper bound for the entropy of $\cL(0)$ and prove Theorem \ref{theorem1}. Section \ref{s:centerdistortion} is dedicated to discussing integrability and weak integrability; we also state some distortion estimates of center derivatives.
In Section \ref{ss.bh}, we introduce blender-horseshoes and discuss their properties.
In Section \ref{sec:concat}, we prove a concatenation of center curves, which has an intrinsic interest (Proposition \ref{proLem:key}). Section \ref{secapproximating} collects some auxiliary results about basic sets. 
In Section \ref{secfractalset}, we construct a fractal set with large entropy on which the \emph{forward} Lyapunov exponent is equal to zero and which consists of non-degenerate compact center curves (see Theorem \ref{THETHEObisbis}). The proof of Theorems \ref{THETHEObisbis} is done in Sections \ref{secfractalset}--\ref{secproofTHEO}. In Section \ref{s.final-fractal}, we prove Theorem \ref{theorem2}. Appendix \ref{App:B} collects some auxiliary results on topological entropy. 

\section{Pre-skeletons}\label{sepreskeleton}

In this section, we construct basic sets based on so-called skeletons. Skeletons are a collection of orbit segments whose cardinality mimics entropy and whose finite-time-exponents ``approximate'' a given exponent. This idea is reminiscent of a common strategy (see, for example, \cite[Supplement S.5]{KatHas:95}) to ``ergodically approximate hyperbolic ergodic measures by horseshoes'' and plays a key role in this paper.
This discussion will be continued in Section \ref{ssec:ergappro}.

Here we introduce the \emph{pre-skeleton property} that is a reformulation of analogous concepts in skew-products and in a partially hyperbolic context (see \cite[Section 4]{DiaGelRam:17} and \cite[Section 5.2]{DiaGelSan:20}, respectively). The latter coincides with our setting. Note that the skeleton property in \cite[Section 5.2]{DiaGelSan:20} is associated with an ergodic measure and guaranteed entirely by ergodic properties (see \cite[Section 4.2]{DiaGelRam:17}, where it is called skeleton$\ast$ property). Here, we present a purely topological version, which requires the existence of sufficiently many ``approximating'' orbit segments of arbitrary length (we call it \emph{pre-skeleton}).

Given $m\in\bN$ and $\e>0$, two points $x,y$ are \emph{$(m,\e)$-separated} (with respect to $f\in\Homeo(M)$) if 
\begin{equation}\label{defseparated}
	d_n(x,y)>\varepsilon, \quad\text{ where }\quad
	d_n(x,y)\eqdef \max_{j=0,\ldots,n-1}d(f^j(x),f^j(y)).
\end{equation}
A subset $X\subset M$ is \emph{$(m,\e)$-separated} if it consists of points which are mutually $(m,\e)$-separated.
By compactness, any such set is finite. 

\begin{definition}[Pre-skeleton]
	Given $h>0$ and $\alpha\in\bR$, we say that $f\in\PH^1_{\c =1}(M)$ \emph{has the pre-skeleton property} relative to $h$ and $\alpha$, if for every $\varepsilon_H>0$ and $\varepsilon_E>0$ there exists  $\varepsilon_0>0$ such that for every $\varepsilon\in(0,\varepsilon_0)$ there are  constants $K_0\ge1$ and $n_0\in\bN$ such that for every $m\ge n_0$ there is a $(m,\varepsilon)$-separated set  $X=X(h,\alpha,\varepsilon_H,\varepsilon_E,m)=\{x_i\}\subset M$, called  \emph{pre-skeleton},   satisfying:
	\begin{itemize}
		\item[(1)] the set $X$ has cardinality $\card X\geq e^{m (h-\varepsilon_H)}$;
		\item[(2)] for every $\ell=0,\ldots,m$ and every $i$ one has 
		\[
			K_0^{-1}e^{\ell(\alpha-\varepsilon_E)}\leq
			\lVert D^\c f^\ell (x_i) \rVert 
			\leq K_0e^{\ell(\alpha+\varepsilon_E)}.
		\]
	\end{itemize}  
\end{definition}

\begin{proposition}\label{pro:skel}
	Let $f\in\PH^1_{\c =1}(M)$. Consider $\alpha\in\bR$ with $h_{\rm top}(f,\cL(\alpha))>0$. Then for every $h\in(0,h_{\rm top}(f,\cL(\alpha)))$,  the map $f$ has the pre-skeleton property relative to $h$ and $\alpha$.
\end{proposition}

\begin{proof}
Fix $\varepsilon_H\in(0,h)$ and $\varepsilon_E\in(0,\alpha)$. Given $N\in\bN$, let
\[
	\cL_N(\alpha)
	\eqdef\big\{x\in\cL(\alpha)\colon e^{n(\alpha-\varepsilon_E)}\leq \lVert D^\c f^n(x)\rVert\leq e^{n(\alpha+\varepsilon_E)} \text{ for every } n\ge N\big\}.
\] 
Note that $\cL_N(\alpha)\subset\cL_{N+1}(\alpha)$ and 
\[
	\cL(\alpha)
	= \bigcup_{N\in\bN}\cL_N(\alpha).
\]
As topological entropy is countably stable and monotone (see Remark \ref{rem:monent}), one has 
\[
	h_{\rm top}\big(f,\cL(\alpha))
	= h_{\rm top}\big(f,\bigcup_{N\in\bN}\cL_N(\alpha)\big)
	= \sup_{N\in\bN}h_{\rm top}(f,\cL_N(\alpha))
	= \lim_{N\to\infty}h_{\rm top}(f,\cL_N(\alpha)).
\]
Hence, there is $N_1\in\bN$ such that  
\[
	h_{\rm top}(f,\cL_{N_1}(\alpha))
	> h-\frac13\varepsilon_H.
\]	
Note that, by item (2) of Lemma \ref{lem:entropy}, $\underline{Ch}_{\rm top}(f,\cL_{N_1}(\alpha)) \ge h_{\rm top}(f,\cL_{N_1}(\alpha))$. It follows from the definition in \eqref{eqjin} that for every $\varepsilon_0>0$ sufficiently small and $\varepsilon\in(0,\varepsilon_0)$, there exists $n_0\in\bN$ such that for every $m\ge n_0$, there exists a $(m,\varepsilon)$-separated set $\{x_i\}_{i\in I}$ in $\cL_{N_1}(\alpha)$ of cardinality
\[
	\card I
	\ge e^{m(h-\varepsilon_H)}.
\]
This implies property (1) of a pre-skeleton. 

To get property (2) of a pre-skeleton, it is enough to take 
\[
	K_0
	\eqdef \max_{\ell=0,\ldots,N_1-1}
	\left\{\lVert D^{\c}f^\ell\rVert^{-1} e^{\ell(\alpha-\varepsilon_E)},
	\lVert D^{\c}f^\ell\rVert  e^{-\ell(\alpha+\varepsilon_E)}, 1
	\right\}.
\]
This proves the proposition.
\end{proof}

For the next result, recall that a set is \emph{basic} if it is compact, $f$-invariant, locally maximal, topologically transitive, and hyperbolic. The \emph{$\u$-index} of a basic set is the $\u$-index of any of its saddles. Note that for $f\in\PH^1_{\c =1}(M)$, the $\u$-index of any basic set is either $d^\uu$ or $d^\uu+1$. Considering the action in the central bundle, we call such sets \emph{of contracting} or \emph{of expanding type}, respectively. The $\u$-index of a hyperbolic ergodic measure is, by definition, the number of its positive Lyapunov exponents counted with multiplicities. When $\mu$ is equidistributed on a hyperbolic periodic orbit, then its $\u$-index   coincides with the $\u$-index of the orbit. Given an $f$-invariant subset $\Lambda\subset M$, we denote by $\cM_{\rm erg}(f|_\Lambda)$ the corresponding subset of ergodic measures.

Proposition \ref{pro:skelhors} below is similar to \cite[Theorem 5.1]{DiaGelSan:20}, where $h=h_\mu(f)$ is the entropy of a nonhyperbolic ergodic measure $\mu$ and $\alpha=\chi^\c(\mu)=0$ is its center Lyapunov exponent. The proof of \cite[Theorem 5.1]{DiaGelSan:20} is based on essentially two ingredients: the pre-skeleton property relative to $h_\mu(f)$ and $\chi^\c(\mu)$ and the minimality of strong foliations. The proof of Proposition \ref{pro:skelhors} follows verbatim as in \cite[Section 5]{DiaGelSan:20}. We refrain from providing the details.

\begin{proposition}\label{pro:skelhors}
	Assume that $f\in\MB^1(M)$ has the pre-skeleton property relative to $h>0$ and $\alpha\ge0$. Then for every $\gamma\in(0,h)$ and $\varepsilon>0$, there is a basic set $\Lambda\subset M$ of expanding type such that
\begin{itemize}
	\item  $h_{\rm top}(f,\Lambda)\in(h-\gamma,h+\gamma)$;
	\item  $\chi^\c(\mu)\in(\alpha-\varepsilon,\alpha+\varepsilon)\cap\bR_+$,   for every $\mu\in\cM_{\rm erg}(f|_\Lambda)$.
\end{itemize}
	
	The analogous result holds when $\alpha\le0$.
\end{proposition}

Finally, we recall the following result (whose proof also uses skeletons).

\begin{theorem}[{\cite[Theorem 6.1]{DiaGelSan:20}}]\label{the:goingotherside}
	Let $f\in\MB^1(M)$. There exists $K(f)>0$ such that  for every positive numbers $\delta$ and $\gamma$, every $\alpha\in(\alpha_{\rm min},0)$, and $\mu\in\cM_{\rm erg,>0}(f)$, there is a basic set $\Lambda$ of contracting type such that 
	\[
	h_{\rm top}(f,\Lambda)
	\ge \frac{h_\mu(f)}{1+K(f)(\beta+|\alpha|)}-\gamma,
	\quad\text{ where }\quad
	\beta\eqdef\chi^\c(\mu)>0,
	\]	
	and every $\nu\in\cM_{\rm erg}(f|_\Lambda)\subset \cM_{\rm erg,<0}(f)$ satisfies
	\[
	\frac{\alpha}{1+K(f)(\beta+|\alpha|)}-\delta
	<\chi^\c(\nu)
	<\frac{\alpha}{1+\frac{1}{\log\,\lVert Df\rVert}(\beta+|\alpha|)}+\delta.
	\]
	
	There is an analogous result for measures in $\cM_{\rm erg,<0}(f)$.
\end{theorem}

\section{Pressure functions and their Legendre-Fenchel transforms}\label{sec:press}

In this section, we study (restricted) pressure functions that are one of the main technical tools to prove the restricted variational principles in Theorem \ref{theorem1}. In that direction, our main result is Theorem \ref{theorem:1}, which almost implies Theorem \ref{theorem1}. Throughout this section, we assume that $f\in\PH^1_{\c =1}(M)$. The essential ingredients here are homoclinic relations, which provide an underlying framework.

Let us define the main tools. Let $\cN\subset\cM_{\rm erg}(f)$. Given $q\in\bR$ and $\varphi^\c$ as in \eqref{def:varphi}, consider the \emph{restricted variational pressure} of the potential $q\varphi^\c$  (with respect to $\cN$),
\[
	\cP_\cN(q)
	\eqdef \sup_{\mu\in\cN}\big(h_\mu(f)+q\int\varphi^\c\,d\mu\big).
\]
Consider its \emph{Legendre-Fenchel transform} defined by
\begin{equation}\label{def:EN}
	\cE_\cN(\alpha)
	\eqdef\inf_{q\in\bR}\big(\cP_\cN(q)-q\alpha\big)
\end{equation}
on its domain
\begin{equation}\label{defdomain}
D(\cE_\cN)
\eqdef \Big\{\alpha\in\bR\colon\inf_{q\in\bR}\big(\cP_\cN(q)-q\alpha\big)>-\infty\Big\}.
\end{equation}
Let us also consider
\begin{equation}\label{def:HN}
\cH_\cN(\alpha)
\eqdef \sup\big\{h_\mu(f)\colon \mu\in\cN,\chi^\c(\mu)=\alpha\big\}.
\end{equation}

\begin{notation}\label{notation}
Our main interest is when $\cN=\cM_{\rm erg,<0}(f)$ or $\cN=\cM_{\rm erg,>0}(f)$, in which case we simply write $\cP_{<0}$ and $\cP_{>0}$ and analogously $\cE_{<0}$ and $\cE_{>0}$, respectively. Also we simply write $\cH$ (the sign of $\alpha$ determines if $\cN=\cM_{\rm erg,<0}$ or $\cN=\cM_{\rm erg,>0}$).
\end{notation}

Recall the numbers $\alpha_{\rm min}$ and $\alpha_{\rm max}$ defined in \eqref{eq:defalphamin}.

\begin{figure}[h]
		\begin{overpic}[scale=.55]{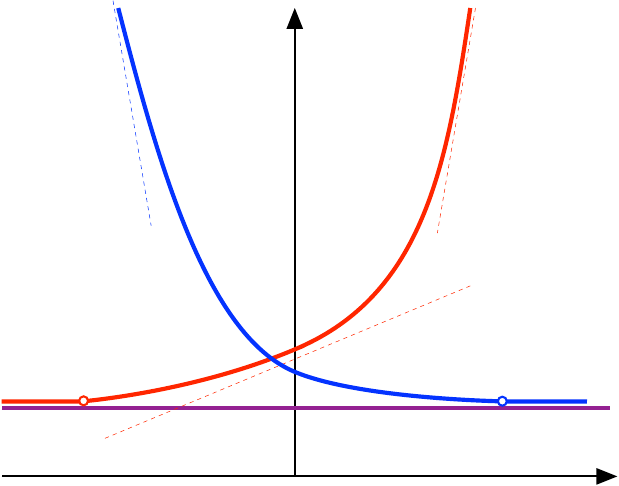}
			\put(96,-5){\small{$q$}}
			\put(101,11){\small{\textcolor{magenta}{$\cH(0)$}}}
			\put(58,69){\small{\textcolor{red}{$\cP_{>0}(q)$}}}
			\put(77,64){\small{\textcolor{red}{$\sim q\alpha_{\rm max}$}}}
			\put(-2,64){\small{\textcolor{blue}{$\sim q\alpha_{\rm min}$}}}
			\put(77,34){\small{\textcolor{red}{$h_\mu(f)+q\alpha^+_{\rm max}$}}}
			\put(0,-5){\small{\textcolor{red}{$-K(f)\h(f)$}}}
			\put(74,-5){\small{\textcolor{blue}{$K(f)\h(f)$}}}
			\put(22,69){\small{\textcolor{blue}{$\cP_{<0}(q)$}}}
			\put(46,-4){\small{$0$}}
		\end{overpic}\hspace{1.3cm}
		\begin{overpic}[scale=.6]{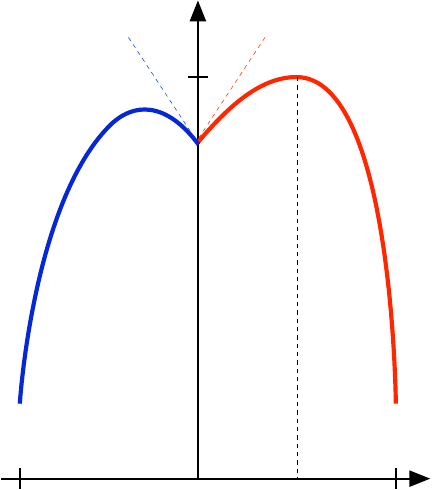} 
			\put(2,95){\small{$h_{\rm top}(f,\cL(\alpha))$}}
			\put(75,-5){\small{$\alpha_{\rm max}$}}
			\put(-4,-5){\small{$\alpha_{\rm min}$}}
			\put(56,-5){\small{$\alpha_{\rm max}^+$}}
			\put(17,86){\small{$h_{\rm top}(f)$}}
			\put(28,68){\small{$\h(f)$}}
			\put(54,95){\small{\textcolor{red}{$\h(f)+K(f)\alpha\h(f)$}}}
			\put(73,75){\small{\textcolor{red}{$\cP^\ast_{>0}(\alpha)$}}}
			\put(2,75){\small{\textcolor{blue}{$\cP^\ast_{<0}(\alpha)$}}}
		\end{overpic}\hspace{1cm}
		\caption{Possible shape of the restricted variational pressures (left figure) and their Legendre-Fenchel transforms (right figure) under the hypotheses of Theorem \ref{theorem:1}. See also Remark \ref{remLFTransforFigures}.}
		\label{figPstar1}
\end{figure}
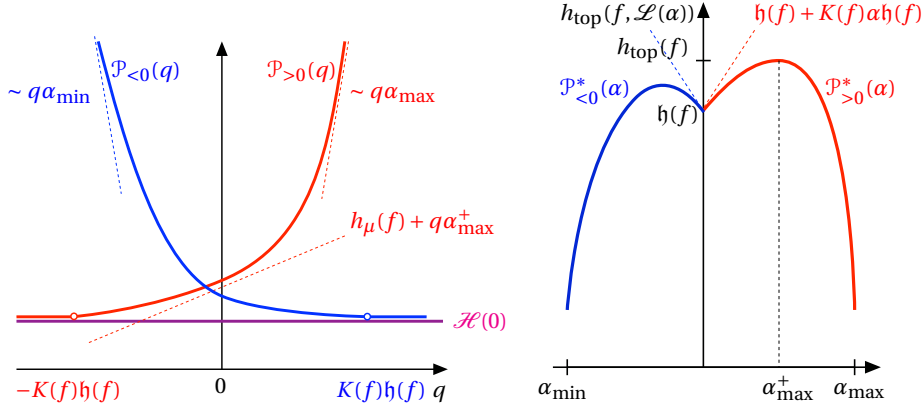

\begin{theorem}[Restricted variational principle -- Hyperbolic parts]\label{theorem:1}
Let $f\in\BM^1(M)$. Then
\begin{equation}\label{nummer}
	D(\cE_{<0})=[\alpha_{\rm min},0].
\end{equation}
Moreover, for every $\alpha\in [\alpha_{\rm min},0)$, it holds $\cL(\alpha)\ne\emptyset$ and
\begin{equation}\label{resvarprinovo}
	h_{\rm top}(f,\cL(\alpha))
	= \cE_{<0}(\alpha)
		= \cH(\alpha).
\end{equation}
	
The analogous result holds for $\cE_{>0}$ and $\alpha\in(0,\alpha_{\rm max}]$.
\end{theorem}

To prove for appropriate $\alpha$ that
\[
	\cH(\alpha)
	= \cE_{<0}(\alpha)
	\le h_{\rm top}(f,\cL(\alpha)),
\]
one only requires $f\in\PH^1_{\c =1}(M)$ and certain ergodic approximation properties; see Proposition \ref{prop:approx1} and Section \ref{ssec:ergappro}. To prove the full equality in \eqref{resvarprinovo}, we require $f\in\BM^1(M)$. Note that Theorem \ref{theorem:1} already implies  Theorem \ref{theorem1} except for the part corresponding to the exponent zero in \eqref{outofmind}.

\begin{remark}[The general shape of the graph of $\alpha\mapsto h_{\rm top}(f,\cL(\alpha))$] \label{remLFTransforFigures}
	Figure \ref{figPstar1} illustrates the pressure functions and their Legendre-Fenchel transforms. 	
	Note that Theorem \ref{theorem:1} implies that $\cH$ is a piecewise concave (hence continuous) function. Further regularity properties (smoothness, analyticity) are unknown so far. In the following, we only discuss some properties of $\cP_{>0}$, the properties of $\cP_{<0}$ are analogous.
	
	The pressure $\cP_{>0}$ has asymptotic slope $\alpha_{\rm max}$ as $q\to\infty$. By definition, $\cP_{>0}$ has supporting lines $\{q\mapsto h_\mu(f)+q\alpha\}$, where $\alpha>0$ ranges over the Lyapunov exponents of ergodic measures $\mu\in\cM_{\rm erg,>0}(f)$. In particular, if there is a unique measure $\mu_{\rm max}$ of maximal entropy and if this measure has positive exponent $\alpha_{\rm max}^+$, then $\cP_{>0}$ is differentiable at $0$ with derivative $\alpha_{\rm max}^+$. 
	
	Moreover, $\cP_{>0}$ has slope zero on the interval $(-\infty,-K(f)\h(f)]$ for $K(f)>0$ as in Theorem \ref{the:goingotherside}. To see this, note that it follows from Theorem \ref{the:goingotherside}, that for every $\mu$ with $\beta\eqdef\chi^\c(\mu)>0$, there is a sequence of ergodic measures $(\mu_i)_i\subset\cM_{\rm erg,<0}(f)$ such that $\chi^\c(\mu_i)\nearrow0$ and
\begin{equation}\label{eqLF}\begin{split}
	\frac{h_\mu(f)}{1+K(f)\cdot\beta}
	&\le \limsup_{i\to\infty}h_{\mu_i}(f)\\
	\text{\small{(by \eqref{def:HN})}}\quad
	&\le \limsup_{i\to\infty}\cH(\chi^\c(\mu_i))\\
	\text{\small{(by Theorem \ref{theorem1})}}\quad
	&\le \lim_{\alpha\nearrow0} \cH(\alpha)=\lim_{\alpha\nearrow0}\cP^\ast_{>0}(\alpha)=\h(f).
\end{split}\end{equation}
Together with the fact that the measure $\mu$ with exponent $\beta>0$ was arbitrary, we get
\[
	\frac{\cH(\beta)}{1+K(f)\cdot\beta}
	= \frac{\cP_{>0}^\ast(\beta)}{1+K(f)\cdot\beta}
	\le \h(f)
	\quad\Longleftrightarrow\quad
	\frac{\cP_{>0}^\ast(\beta)-\h(f)}{\beta}
	\le K(f)\h(f).
\]
Taking limits $\beta\searrow0$, implies that the right hand-side derivative satisfies $D_R\cP_{>0}^\ast(0)\le K(f)\h(f)$. Hence, the assertion about the slope zero follows from classical properties of the Legendre-Fenchel transform.

	Note that taking $\mu$ in \eqref{eqLF} satisfying $h_\mu(f)>0$, one  immediately gets that $\h(f)\ge\cH(0)>0$.

	At the current state of the art, it is unknown if $\cH(\alpha_{\rm min})$ and $\cH(\alpha_{\rm max})$ are positive or zero or even equal to $h_{\rm top}(f)$; this type of question is studied in the context of \emph{ergodic optimization} (see, for example, \cite{Jen:19}). 
\end{remark}

This section is organized as follows. In Section \ref{ssec:genpro}, we collect some general properties of the thermodynamical quantifiers defined above. In Section \ref{sechomcla}, we recall the concepts of intersection and homoclinic classes. In Section \ref{ssec:ergappro} we recall some ergodic approximation properties. Section \ref{ssec:exhaus} explains how to approximate restricted pressure functions and their convex conjugates by means of so-called exhausting families.  In Sections \ref{ssec:final} and \ref{sec:proofThm1}, we prove Theorems \ref{theorem:1} and \ref{theorem1}, respectively. Section \ref{app2} contains the proof of the semi-continuity of the spectrum in  Corollary~\ref{cor:semiconti}.

\subsection{General properties of the restricted pressure}\label{ssec:genpro}

We recall some general properties of the restricted variational pressure and its Legendre-Fenchel transform (see \cite[Section 5]{DiaGelRam:19} for proofs and further properties, here we keep the numeration in \cite{DiaGelRam:19}).

Given $\cN\subset\cM_{\rm erg}(f)$, let
\begin{equation}\label{eqdefvarphicN}
	\varphi^\c_\cN
	\eqdef \Big\{\int\varphi^\c\,d\mu\colon\mu\in\cN\Big\},\quad
	\underline\varphi^\c_\cN
	\eqdef\inf\varphi^\c_\cN,\quad
	\overline\varphi^\c_\cN
	\eqdef\sup\varphi^\c_\cN.
\end{equation}
We have the following properties:
\begin{itemize}
	\item[(P1)] For every $\mu\in\cN$, the function $\cP_{\{\mu\}}$ is affine and satisfies $\cP_{\{\mu\}}\le\cP_\cN$.
	\item[(P2)]  $\cP_{\cN'}\le\cP_\cN$ for every $\cN'\subset\cN$.
	\item[(P4)] The function $\cP_\cN$ is uniformly Lipschitz continuous.
	\item[(P5)] The function $\cP_\cN$ is convex. 
	\item[(E1)] The function $\cE_\cN$ is concave (and hence continuous). Hence, it is differentiable at all but at most countably many $\alpha$ and the left and right derivatives are defined for all $\alpha\in D(\cE_\cN)$.
	\item[(E2)] $D(\cE_\cN)$ is an interval and $D(\cE_\cN)\supset(\underline\varphi^\c_\cN,\overline\varphi^\c_\cN)$.
	\item[(E4)] It holds 
	\[
	\max_{\alpha\in D(\cE_\cN)}\cE_\cN(\alpha)=\cP_\cN(0).
	\]	
	\item[(E5)] For every $\alpha\in D(\cE_\cN)$,
	\[
	\cH_\cN(\alpha)
	\le 	\cE_\cN(\alpha).
	\]
\end{itemize}

When in (E5) we have equality, then this is also called a \emph{variational principle restricted to $\cN$} or simply \emph{restricted variational principle}. For that, in general, some further properties are required, see Section \ref{ssec:exhaus}.

\subsection{Intersection and homoclinic classes}\label{sechomcla}
Homoclinic relations between saddles play a pivotal role in this paper. Let us recall some key facts. Denote by $\Per_{<0}(f)$ and $\Per_{>0}(f)$ the sets of all saddles of $\u$-index $d^\uu$ and $d^\uu+1$, respectively. 
Two saddles of the same $\u$-index are \emph{homoclinically related} if the unstable manifold of one saddle has non-empty transverse intersections with the stable manifold of the other saddle, and vice versa. Note that this defines an equivalence relation on the sets $\Per_{<0}(f)$ and $\Per_{>0}(f)$. We call \emph{intersection classes} the equivalence classes for this relation. 
The \emph{homoclinic class} of a saddle $P$ is the closure of its intersection class, denoted by $H(P,f)$. Note that if $P,Q$ are homoclinically related saddles, then $H(P,f)=H(Q,f)$. Note that $H(P,f)$ may contain saddles whose $\u$-index is different from the one of $P$. Indeed, we see below that this feature occurs for $f\in\BM^1(M)$, see Remark \ref{r.homorelated}.

\begin{remark}[Homoclinically related saddles]\label{r.homorelated}
	Given $f\in\BM^1(M)$, any pair of saddles with the same $\u$-index are homoclinically related. Furthermore, $\Per_{<0}(f)$ and $\Per_{>0}(f)$ are intersection classes and satisfy
	\[\overline{\Per_{<0}(f)}=\overline{\Per_{>0}(f)}= M.	\]
\end{remark}

Remark~\ref{r.homorelated} follows from the fact that 	for $f\in\BM^1(M)$, the strong stable manifold of a saddle of $\u$-index $d^\uu+1$   is its stable manifold, the strong unstable manifold of a  saddle of $\u$-index $d^\uu$   is its unstable manifold,  and the strong foliations  $\cW^\ss$ and $\cW^\uu$ both are, by hypothesis, minimal.

\subsection{Ergodic approximation property}\label{ssec:ergappro}

We now resume the ideas in Section \ref{sepreskeleton}. The goal is to ``gradually approximate'' the restricted pressure functions (and the corresponding Legendre-Fenchel transforms). 

\begin{definition}\label{defef}
	Given $f\in\PH^1_{\c =1}(M)$, a compact $f$-invariant set $\Lambda\subset M$ has the \emph{$\u$-index $d$-ergodic approximation property} or simply \emph{$d$-ergodic approximation property} if the following properties hold: 
	\begin{enumerate}[label=$\bullet$, leftmargin=0.5cm ]
		\item[(1)] Every pair of saddles of $\u$-index $d$ in $\Lambda$ are homoclinically related.
		\item[(2)] For every hyperbolic measure $\mu\in\cM_{\rm erg}(f|_\Lambda)$ of $\u$-index $d$ and every $\varepsilon>0$ there is a basic set $\Xi$ of $\u$-index $d$ such that 
		\begin{itemize}
			\item[(i)] $\Xi\subset\Lambda$,
			\item[(ii)] $\Xi$ is contained in the $\varepsilon$-neighborhood of the support of $\mu$,
			\item[(iii)] $h_{\rm top}(f,\Xi)
			\ge h_\mu(f) -\varepsilon$,
			\item[(iv)] $\lvert\chi^\c(\mu')-\chi^\c(\mu)\rvert
			\le\varepsilon\text{ for every }\mu'\in\cM_{\rm erg}(f|_\Xi)$.
		\end{itemize}
	\end{enumerate}
\end{definition}

For maps in $\PH^1_{\c =1}(M)$, the existence of ``ergodically approximating basic sets'', in the sense of satisfying properties (ii)--(iv), for hyperbolic ergodic measures was shown in \cite{Gel:16} (see also \cite[Proposition 1.4]{Cro:11} and compare \cite[Lemma 6.6]{DiaGelRam:19}). Below we explore contexts where the ergodic approximation property holds.

The following result is the main step towards the proof of Theorem \ref{theorem:1}. 

\begin{proposition}\label{prop:approx1}
	Let $f\in\PH^1_{\c =1}(M)$ and assume that $\Lambda\subset M$ is a compact $f$-invariant set having the $d^\uu$-ergodic approximation property. Letting 
	\[
	\cN
	\eqdef \cM_{\rm erg}(f|_\Lambda)\cap\cM_{\rm erg,<0}(f),
	\]
for every $\alpha\in (\underline\varphi^\c_\cN,\overline\varphi^\c_\cN)$, it holds $\cL(\alpha)\cap\Lambda\ne\emptyset$ and 
	\begin{equation}\label{resvarpribis}
		\cE_\cN(\alpha)
		= \cH_\cN(\alpha)
		\le h_{\rm top}(f,\cL(\alpha)\cap\Lambda).
\end{equation}\end{proposition}

After collecting some ingredients, Proposition \ref{prop:approx1} will be proved at the end of Section \ref{ssec:exhaus}. The first step is the following.

\begin{lemma}\label{lem:approx1}
	Let $f\in\PH^1_{\c =1}(M)$ such that $\Per_{<0}(f)$ is a nonempty intersection class. Then 
	\begin{equation}\label{duuuu}
		\emptyset\ne\cM_{\rm erg,<0}(f) \subset \cM_{\rm erg}(f|_{\overline{\Per_{<0}(f)}}).
	\end{equation}
	Moreover, the sets $M$ and $\overline{\Per_{<0}(f)}$ have the $d^\uu$-ergodic approximation property. 	
	The analogous assertion holds when $\Per_{>0}(f)$ is an intersection class.
\end{lemma}

\begin{proof}
	The hypothesis that $\Per_{<0}(f)$ is a nonempty intersection class implies $\cM_{\rm erg,<0}(f)\ne\emptyset$. Moreover, it immediately implies that property (1) in Definition \ref{defef} is satisfied by taking $\Lambda=M$. 	Since $f\in \PH^1_{\c =1}(M)$, for every $\mu\in\cM_{\rm erg,<0}(f)$ and for $\mu$-almost every $x$, the subspace $E^\ss(x)\oplus E^\c(x)$ contains all stable subspaces and $E^\uu(x)$ contains all unstable subspaces of its Oseledets splitting.  Hence, by Theorem 1 in \cite{Gel:16}, every measure in $\cM_{\rm erg,<0}(f)$ is weak$\ast$ approximated by periodic ones and  $M$ satisfies (ii)--(iv) in the property (2); while (i) is immediate. In particular, this implies the inclusion in \eqref{duuuu}.
	
	Let us argue that the $d^\uu$-ergodic approximation property also holds for $\overline{\Per_{<0}(f)}$. By definition, property (1) is immediate. To check property (2), consider any hyperbolic measure $\mu\in\cM_{\rm erg}(f|_{\overline{\Per_{<0}(f)}})$ of $\u$-index $d^\uu$. Invoking again \cite{Gel:16}, one gets basic sets $\Xi\subset M$ of $\u$-index $d^\uu$ ``approximating'' $\mu$ satisfying (ii)--(iv) in property (2). As $\overline{\Per_{<0}(f)}$ is an intersection class, every saddle in $\Xi$ is in $\Per_{<0}(f)$ and hence $\Xi\subset\overline{\Per_{<0}(f)}$. 
\end{proof}

\begin{remark}
	Under the hypotheses of Lemma \ref{lem:approx1}, $\overline{\Per_{<0}(f)}=H(P,f)$ for every $P\in \Per_{<0}(f)$. In general, the inclusion in Equation \eqref{duuuu} is proper. It can happen that the set $\cM_{\rm erg}(f|_{\overline{\Per_{<0}(f)}})$   intersects $\cM_{\rm erg,>0}(f)$. Indeed, this is what happens for $f\in\BM^1(M)$, see Remark \ref{r.homorelated}.
	The topology of the space of ergodic measures was further studied in \cite{GorPes:17} assuming that all saddles with the same index are in the same intersection class. 
\end{remark}

\subsection{Gradual approximation of pressure: Exhausting families}\label{ssec:exhaus}

Consider a compact $f$-in\-variant set $\Lambda\subset M$. Given a family $(\bX_i)_{i\in\bN}$ of compact $f$-invariant subsets of $\Lambda$, introducing a simplified version of the notation in Section \ref{sec:press}, for $q\in\bR$ let  
\begin{equation}\label{acceptcomments}
	\cP_i(q)
	\eqdef \cP_{\cM_i}(q)
	= \sup_{\mu\in\cM_i}\big(h_\mu(f)+q\int\varphi^\c\,d\mu\big), \quad 
	\text{ where }\quad
	\cM_i\eqdef\cM_{\rm erg}(f|_{\bX_i})
\end{equation}
and define
\[
\cE_i(\alpha)
\eqdef \cE_{\cM_i}(\alpha)
= \inf_{q\in\bR}\big(\cP_i(q)-q\alpha\big)
\] 
for every 
\[
	\alpha\in
	D(\cE_i)
	\eqdef \big\{\beta\in\bR\colon \inf_{q\in\bR}(\cP_i(q)-q\beta)>-\infty\big\}.
\]

\begin{definition}[Exhausting family]
Given a compact $f$-in\-variant set $\Lambda\subset M$ and $\cN\subset\cM_{\rm erg}(f|_\Lambda)$, we say that a family $(\bX_i)_i$ is \emph{$(\Lambda,\varphi^\c,\cN)$-exhausting} if for every $i$:
\begin{enumerate}
	\item[(Exh1)] $\cM_i\eqdef\cM_{\rm erg}(f|_{\bX_i})\subset\cN$;
	\item[(Exh2)] $f|_{\bX_i}$ has the specification property%
	\footnote{We refrain from giving this technical definition, see, for example, \cite{Sig:74}. Just recall that this property is satisfied for any basic set of $f$.};
	\item[(Exh3)] For every $\alpha\in\interior D(\cE_i)$ the \emph{restricted variational principle} holds:
	\[
	\cE_i(\alpha)
	= \cH_i(\alpha)
	\eqdef \cH_{\cM_i}(\alpha)
	= \sup\big\{h_\mu(f)\colon\mu\in\cM_i,\int\varphi^\c\,d\mu=\alpha\big\}.
	\]
	\item[(Exh4)] For every $q\in\bR$, 
	\[
	\lim_{i\to\infty}\cP_i(q)
	= \cP_\cN(q)
	= \sup_{\mu\in\cN}\big(h_\mu(f)+q\int\varphi^\c\,d\mu\big).
	\]	 
	\item[(Exh5)] Recalling notation in Equation \eqref{eqdefvarphicN}, it holds
	\[
	\underline\varphi^\c_\cN
	=\lim_{i\to\infty}\underline\varphi^\c_{\cM_i},\quad
	\overline\varphi^\c_\cN
	=\lim_{i\to\infty}\overline\varphi^\c_{\cM_i}.
	\]	
\end{enumerate}  
\end{definition}

\begin{remark}\label{rem:basic}
	When $\bX_i$ is a basic set (and in particular has the specification property), then (Exh2) holds. Property (Exh3) would follow immediately when $\varphi^\c$ is H\"older continuous. When $\varphi^\c$ is just continuous, compare the approximation arguments in \cite[Proposition 5.4]{DiaGelRam:19} in the skew product setting.%
	\footnote{Note that  the notation ${\mathcal E}_\cN(\cdot)=\cP_\cN^\ast(\cdot)$ is used in \cite{DiaGelRam:19}.} 
\end{remark}

\begin{remark}\label{rem:continuity-of-LG-transformation}
	By \cite[Theorem 6.2]{Wijsman:66}, (Exh4) implies that $\lim_{i\to\infty}\cE_i(\alpha)=\cE_{\cN}(\alpha)$.
\end{remark}

\begin{remark}\label{remsimpleBowenarg}
	As $\varphi^\c\colon M\to\bR$ is continuous, then 
	\[
	\cH_\cN(\alpha)
	\le 	\cH_{\cM_{\rm erg}(f|_\Lambda)}(\alpha)
	\le 	h_{\rm top}(f,\cL(\alpha)\cap\Lambda)
	\]
	(see, for instance, general results by Bowen \cite{Bow:73} or \cite[Lemma 5.2]{DiaGelRam:19}). 
Moreover, when $\bX_i$ is a basic set, then by \cite[Theorem 6.1]{PfiSul:07} we can apply \cite[Proposition 7.1]{PfiSul:07} to get that for every $\alpha\in\interior D(\cE_i)$ one has 
	\[
	 \cH_i(\alpha)
	=h_{\rm top}(f,\cL(\alpha)\cap \bX_i).
	\]
\end{remark}

A family $(\bX_i)_i$ is called \emph{increasing} if $\bX_i\subset \bX_{i+1}$ for every $i$. The following result gives the existence of exhausting families in our setting (a reformulation of \cite[Proposition 6.3]{DiaGelRam:19}).

\begin{lemma}\label{lempro:exhaust}
	Let $f\in\PH^1_{\c =1}(M)$ and   $\Lambda\subset M$ be a compact $f$-invariant set with the $d^\uu$-ergodic approximation property. Let 
\[
	\cN=\cM_{\rm erg}(f|_\Lambda)\cap\cM_{\rm erg,<0}(f).
\]	 
Then there exists an increasing $(\Lambda,\varphi^\c,\cN)$-exhausting family consisting of basic sets. The analogous assertion holds considering the $d^\uu+1$-ergodic approximation property and the set $\cM_{\rm erg}(f|_\Lambda)\cap\cM_{\rm erg,>0}(f)$. 
\end{lemma}

\begin{proof}
	By the ergodic approximation property, every $\mu\in\cN$ can be arbitrarily well ``ergodically approximated'' by basic sets of $\u$-index $d^\uu$ in $\Lambda$. By property (1) in Definition~\ref{defef}, every pair of saddles in $\Lambda$ are homoclinically related. Therefore, given any two such basic sets $\Gamma_1,\Gamma_2\subset \Lambda$, their periodic points are homoclinically related and hence there is a basic set $\Gamma\supset\Gamma_1\cup\Gamma_2$ of $\u$-index $d^\uu$ in $\Lambda$. This will guarantee that the family can be chosen to be increasing.
		
	The proof then proceeds analogously to \cite[Proof of Proposition 6.3]{DiaGelRam:19} considering an auxiliary sequence of basic sets $\bX_i\subset \Lambda$ eventually containing every saddle of $\u$-index $d^\uu$ in $\Lambda$. Property (Exh1) is immediate. Properties (Exh2) and (Exh3) are satisfied for every $\bX_i$ by Remark \ref{rem:basic}. To show properties (Exh4) and (Exh5), one proceeds as in \cite[Sections 5.2--5.3 and 6]{DiaGelRam:19}. This proves the lemma.
\end{proof}

The following result is an immediate consequence of \cite[Lemma 6.1 and Proposition 6.2]{DiaGelRam:19} in our setting. Recall the notations in Equation \eqref{eqdefvarphicN}.

\begin{proposition}\label{pro:useful}
	Let $f\in\PH^1_{\c =1}(M)$. Assume that $\Lambda\subset M$ is a compact $f$-invariant set and that there is $\cN\subset\cM_{\rm erg}(f|_\Lambda)$ having an increasing  $(\Lambda,\varphi^\c,\cN)$-exhausting family $(\bX_i)_i$. Then $\varphi^\c_\cN$ is an interval and 
	\[
	\interior D(\cE_\cN)
	= (\underline\varphi^\c_\cN,\overline\varphi^\c_\cN)
	\subset\varphi^\c_\cN
	\subset[\underline\varphi^\c_\cN,\overline\varphi^\c_\cN].
	\]
	Moreover, for every $\alpha\in (\underline\varphi^\c_\cN,\overline\varphi^\c_\cN)$, it holds $\cL(\alpha)\cap\Lambda\neq\emptyset$ and 
	\[
		\cE_\cN(\alpha)
		= \lim_{i\to\infty}\cE_i(\alpha)
		=\lim_{i\to\infty}\cH_{\cM_{\rm erg}(f|_{\bX_i})}(\alpha)
		\le h_{\rm top}(f,\cL(\alpha)\cap\Lambda).
	\]
\end{proposition}

\begin{remark}
	Note that under the general hypotheses of Proposition \ref{pro:useful}, the interval $\varphi^\c_\cN$ could be either open, half-open, or closed. For example, for $f\in\BM^1(M)$ with $\Lambda=M$ and $\cN=\cM_{\rm erg,<0}(f)$, it holds $\varphi^\c_\cN=[\alpha_{\rm min}^f,0)$.
\end{remark}

We finally are ready to give the proof of Proposition~\ref{prop:approx1}.

\begin{proof}[Proof of Proposition \ref{prop:approx1}]
Let $f\in\PH^1_{\c =1}(M)$ and assume that  $\Lambda\subset M$ is a compact $f$-invariant set with the $d^\uu$-ergodic approximation property. Let $\cN=\cM_{\rm erg}(f|_\Lambda)\cap\cM_{\rm erg,<0}(f)$. 
By Lemma \ref{lempro:exhaust} and Proposition \ref{pro:useful}, for every $\alpha\in(\underline\varphi^\c_\cN,\overline\varphi^\c_\cN)$, it holds $\cL(\alpha)\cap\Lambda\ne\emptyset$ and
\begin{equation}\label{almosssst}
	\cE_\cN(\alpha)
	\le h_{\rm top}(f,\cL(\alpha)\cap\Lambda).
\end{equation}

\begin{claim}\label{lemalmossss}
	For every $\alpha\in(\underline\varphi^\c_\cN,\overline\varphi^\c_\cN)$, $\cE_\cN(\alpha)= \cH_\cN(\alpha)$.
\end{claim}
\begin{proof}
	By (E5) in Section \ref{ssec:genpro}, it holds
	\[
	\cE_\cN(\alpha)
	\ge \cH_\cN(\alpha)
	\]
Hence, it remains to show $\cE_\cN(\alpha)
	\le \cH_\cN(\alpha)$.
	
		By Lemma \ref{lempro:exhaust}, there is an exhausting family $(\bX_i)_i$ of basic sets. Take $\alpha\in(\underline\varphi^\c_\cN,\overline\varphi^\c_\cN)$. Hence, by  (Exh5), there exists $i_0\in\bN$ such that $\alpha\in(\underline\varphi^\c_{\cM_i},\overline\varphi^\c_{\cM_i})$ for every $i\ge i_0$. 
Recalling the notation in \eqref{acceptcomments}, by (Exh3) for $\bX_i$, one has 
\[
	\cE_i(\alpha)=\sup\big\{h_\mu(f)|\mu\in\cM_i,\int\varphi^c\ud\mu=\alpha \big\}.
\]
As $\cM_i\subset\cN$, one has $\cE_i(\alpha)\leq \cH_\cN(\alpha).$ Combining with  (Exh4) and   Remark~\ref{rem:continuity-of-LG-transformation}, one has $\cE_\cN(\alpha)=\lim_{i\rightarrow+\infty}\cE_i(\alpha)\le \cH_\cN(\alpha)$, proving the claim. 
\end{proof}

Claim \ref{lemalmossss} together with Equation \eqref{almosssst} completes the proof of Proposition \ref{prop:approx1}.
\end{proof}

\subsection{Proof of Theorem \ref{theorem:1}}\label{ssec:final}

Recall Notation \ref{notation}. Note that for $f\in\BM^1(M)$, by Remark~\ref{r.homorelated}, the set $\Per_{<0}(f)$ is an intersection class and  $\overline{\Per_{<0}(f)}=M$. Hence, by Lemma \ref{lem:approx1},  one has 
\[
	\cM_{\rm erg,<0}(f)\cap\cM_{\rm erg}(f|_{\overline{\Per_{<0}(f)}})
	= \cM_{\rm erg,<0}(f)
	\eqdef\cN
\]	
and therefore $\cP_{<0}= \cP_\cN$ and $\cE_{<0}=\cE_\cN$.
Recall that $\alpha_{\rm min}\eqdef\alpha_{\rm min}^f=\inf_{\mu\in\cM_{\rm erg}(f)}\chi^\c (\mu)$. For coherence, we also write $\underline\varphi^\c_{<0}=\underline\varphi^\c_\cN$ and $\overline\varphi^\c_{<0}=\overline\varphi^\c_\cN$.

\begin{claim}\label{lc.lyapunov-interval}
	$\underline\varphi^\c_{<0}= \alpha_{\rm min}<0$ and 
	$\overline\varphi^\c_{<0}= 0$.
\end{claim}

\begin{proof}
	The fact $\overline\varphi^\c_{<0}= 0$ follows from Remark \ref{BM1almpa}.
	As $\Per_{<0}(f)\ne\emptyset$, one has $\alpha_{\rm min}<0$. Recall that in the definition of $\alpha_{\rm min}$ it suffices to take the infimum over ergodic measures in $\cM_{\rm erg,<0}(f)$, and thus over $\cN$. 
\end{proof}

For the next, recall the definition of the domain $D(\cdot)$ in \eqref{defdomain}. As $\Per_{<0}(f)$ is an intersection class, Lemma \ref{lem:approx1} implies that $M$ has the $d^\uu$-ergodic approximation property. Hence, by Lemma~\ref{lempro:exhaust}, there exists an increasing $(M,\varphi^\c ,\cN)$-exhausting family. By Claim \ref{lc.lyapunov-interval}, we have $\underline\varphi^\c_{<0}= \alpha_{\rm min}$. Proposition~\ref{pro:useful} and property (E2) imply that 
\[
	(\alpha_{\rm min}, 0)\subset D(\cE_{<0})\subset[\alpha_{\rm min}, 0].
\]	 

\begin{claim}\label{nextlemm}
	$D(\cE_{<0})=[\alpha_{\rm min}, 0]$.
\end{claim}

\begin{proof}
	It is enough to see that $\alpha_{\rm min}$ and $0$ belong to this domain. We first consider $\alpha_{\rm min}$. Taking limits and ergodic decompositions, it is straightforward to check that $\{\mu\in\cN\colon \int\varphi^\c \,d\mu=\alpha_{\rm min}\}\ne\emptyset$. Thus, 
\[\begin{split}
	\inf_{q\in\bR}\big(\cP_{<0}(q)-q\alpha_{\rm min}\big)
	&=\inf_{q\in\bR}\big(\sup_{\mu\in\cN}(h_\mu(f)+q\int\varphi^\c\,d\mu)-q\alpha_{\rm min}\big)\\
	&\ge\inf_{q\in\bR}\Big(\sup_{\mu\in\cN\colon\int\varphi^\c\,d\mu=\alpha_{\rm min}}(h_\mu(f)+q\alpha_{\rm min})-q\alpha_{\rm min}\Big)\\
	&= \sup_{\mu\in\cN\colon \chi^\c(\mu)=\alpha_{\rm min}}h_\mu(f)
= \cH(\alpha_{\rm min})
\geq 0
>-\infty.
\end{split}\] 
This implies $\alpha_{\rm min}\in D(\cE_{<0})$. 

To consider $0$, note that for every $q\geq  1/|\alpha_{\rm min}|$, 
\[
\cN_q
\eqdef\big\{\mu\in\cN\colon \chi^\c(\mu)=-1/q \big\}
\ne\emptyset.
\]
Thus, for every $q\in (1/|\alpha_{\rm min}|,\infty)$, taking $\mu_q\in\cN_q$, it follows
\[
	\sup_{\mu\in\cN}\big(h_\mu(f)+q\chi^\c(\mu)\big)\geq h_{\mu_q}(f)+q\chi^\c(\mu_q)\geq  -1.
\]
Thus, we get
\[\begin{split}
	\inf_{r\geq  1/|\alpha_{\rm min}|}\cP_{<0}(r)
	&=	\inf_{r\geq  1/|\alpha_{\rm min}|}\,\,\sup_{\mu\in\cN}\big(h_\mu(f)+r\chi^\c(\mu)\big)
	\ge -1.
\end{split}\] 
On the other hand, 
\[	
\inf_{q\leq 0}\cP_{<0}(q)
\geq \sup_{\mu\in\cN }h_\mu(f)\geq 0
\quad\text{ and }\quad 
\inf_{0\leq q\le 1/|\alpha_{\rm min}|}\cP_{<0}(q)\geq \sup_{\mu\in\cN}h_\mu(f)-1>-\infty.
\]	 
Hence, one concludes that $0\in D(\cE_{<0})$.
\end{proof}

Claims \ref{lc.lyapunov-interval} and \ref{nextlemm} prove the assertion \eqref{nummer} in the theorem. Let us now prove the remaining assertion \eqref{resvarprinovo}.

\begin{claim}\label{lclstrange}
	For every $\alpha\in[\alpha_{\rm min},0)$, $\cE_{<0}(\alpha)= \cH(\alpha)$.
\end{claim}

\begin{proof}
We first consider $\alpha\in (\alpha_{\rm min},0)$, recall again that $M$ has the $d^\uu$-ergodic approximation property. Applying Proposition \ref{prop:approx1} to $\Lambda=M$, we conclude $\cH(\alpha)=\cE_{<0}(\alpha)$.

We now consider the case $\alpha_{\rm min}$. Note that, using property (E5) and definition \eqref{def:EN},
\begin{equation}\label{lazy}\begin{split}
	\cH(\alpha_{\rm min})	
	\le \cE_{<0}(\alpha_{\rm min})
	&= \inf_{q\in\bR}\big(\cP_{<0}(q)-q\alpha_{\rm min}\big)
	\leq  \inf_{q<0}\big(\cP_{<0}(q)-q\alpha_{\rm min}\big)\\
	&=\inf_{q<0}\sup_{\mu\in\cN}\big(h_\mu(f)+q(\chi^\c(\mu)-\alpha_{\rm min})\big).
\end{split}\end{equation} 
For each $q<0$, take $\mu_q\in\cN$ such that 
\begin{eqnarray}
	\sup_{\mu\in\cN}\big(h_\mu(f)+q(\chi^\c(\mu)-\alpha_{\rm min})\big)
	&\leq&  h_{\mu_q}(f)+q\big(\chi^\c(\mu_q)-\alpha_{\rm min}\big)-q^{-1}
		\label{firsts}\\
	&\leq& h_{\mu_q}(f)-q^{-1}.
	\label{eq:variantional-principle-alphamin}
\end{eqnarray}
Thus, \eqref{firsts} implies that  
\[ 
	h_{\mu_q}(f)+q\big(\chi^\c(\mu_q)-\alpha_{\rm min}\big)-q^{-1}
	\geq \cH(\alpha_{\rm min}),
\]
which gives (recall that $q<0$)
\[
	\chi^\c(\mu_q)-\alpha_{\rm min}
	\leq \frac{1}{q}\big(\cH(\alpha_{\rm min})-h_{\mu_q}(f)\big)+q^{-2}. 
\]
Up to taking a subsequence, one can assume that $\mu_q$ converges to an invariant measure $\mu_\infty$ for $q$ tending to $-\infty$. As $|\cH(\alpha_{\rm min})-h_{\mu_q}(f)|\leq 2 h_{\rm top}(f)<\infty$, then
 \[\int\varphi^\c\ud\mu_\infty-\alpha_{\rm min}\leq 0.\] 
 By the definition of $\alpha_{\rm min}$, one has $\int\varphi^\c\ud\mu_\infty=\alpha_{\rm min}$. Note that almost every ergodic component of $\mu_{\infty}$ also has center Lyapunov exponent $\alpha_{\rm min}$. By the affine property of the metric entropy,  one has $h_{\mu_\infty}(f)\leq  \cH(\alpha_{\rm min})$. As $f$ has one-dimensional center bundle, by \cite{LiViYa:13} (see also \cite{DiFiPaVi:12}), the metric entropy varies upper semi-continuously, and thus $\limsup_{q\to-\infty}h_{\mu_q}(f)\leq h_{\mu_\infty}(f)$. By inequalities~\eqref{eq:variantional-principle-alphamin} and \eqref{lazy}, one has 
\[\begin{split}
	\cH(\alpha_{\rm min})
	&\le \cE_{<0}(\alpha_{\rm min})
	\le \inf_{q<0}	\sup_{\mu\in\cN}\big(h_\mu(f)+q(\chi^\c(\mu)-\alpha_{\rm min})\big)\\
	&\leq\limsup_{q\to-\infty}h_{\mu_q}(f)-q^{-1}\leq  h_{\mu_\infty}(f)
	\leq  \cH(\alpha_{\rm min}) . 
\end{split}\]
Therefore, one has $\cE_{<0}(\alpha_{\rm min})=\cH(\alpha_{\rm min})$, proving the assertion for $\alpha=\alpha_{\rm min}$.
\end{proof}

It remains to show the first equality in \eqref{resvarprinovo}. 

\begin{claim}
	For every $\alpha\in[\alpha_{\rm min},0)$, $h_{\rm top}(f,\cL(\alpha))=\cE_{<0}(\alpha)$ and $\cL(\alpha)\ne\emptyset$.
\end{claim}

\begin{proof}
	Let us first consider $\alpha=\alpha_{\rm min}$. Here, the restricted variational principle
\[
h_{\rm top}(f,\cL(\alpha_{\rm min}))
=  \sup\big\{h_\mu(f)\colon\mu\in\cM_{\rm erg}(f),\chi^\c(\mu)=\alpha_{\rm min}\big\}
= \cH(\alpha_{\rm min})
\]  
follows from \cite[Lemma 5.2]{DiaGelRam:19} (which holds in a general context of a continuous map and level sets for a continuous potential). Together with Claim \ref{lclstrange}, we get
\[
h_{\rm top}(f,\cL(\alpha_{\rm min}))
=\cE_{<0}(\alpha).
\]

Let us now assume $\alpha\in(\alpha_{\rm min},0)$. Here we use again that $M$ has the $d^\uu$-ergodic approximation property and apply Proposition \ref{prop:approx1} to $\Lambda=M$ to get $\cL(\alpha)\ne\emptyset$ and 
\begin{equation}\label{eqchuva}
	\sup\big\{h_\mu(f)\colon\mu\in\cN,\chi^\c(\mu)=\alpha\big\}
	= \cE_{<0}(\alpha)
	\le h_{\rm top}(f,\cL(\alpha)).
\end{equation}
It remains to prove equality in \eqref{eqchuva}. By concavity-continuity property (E1), for every $\gamma>0$,  there is $\varepsilon>0$ such that 
 for every $\beta\in(\alpha-\varepsilon,\alpha+\varepsilon)$, one has
$|\cE_{<0}(\beta)-\cE_{<0}(\alpha)|<\gamma.$
By Propositions \ref{pro:skel} and \ref{pro:skelhors}, there exists a basic set $\Lambda\subset M$ such that
\[
	h_{\rm top}(f,\cL(\alpha))-\gamma
	< h_{\rm top}(f,\Lambda)
	~\text{ and }~
	\chi^\c(\mu)\in(\alpha-\varepsilon,\alpha+\varepsilon)\text{ for all }\mu\in\cM_{\rm erg}(f|_\Lambda).
\]
Take $\nu\in \cM_{\rm erg}(f|_\Lambda)$ the measure of maximal entropy of $f|_\Lambda$, then one has 
\begin{align*}
	h_{\rm top}(f,\cL(\alpha))
	< h_{\nu}(f)+\gamma
	&\leq \sup\big\{h_\mu(f)\colon\mu\in\cN,\chi^\c(\mu)=\chi^c(\nu)\big\}+\gamma
	= \cE_{<0}(\chi^c(\nu))+\gamma\\
	&\leq \cE_{<0}(\alpha)+2\gamma,
\end{align*}
proving the claim.
\end{proof}

The proof of  Theorem \ref{theorem:1} is now complete.
\qed

\subsection{Proof of Theorem \ref{theorem1}}\label{sec:proofThm1}

First note that Theorem \ref{theorem:1} already provides the following facts about the entropy spectrum of  non-zero Lyapunov exponents:
\begin{itemize}[leftmargin=0.8cm ]
\item $\{\chi^\c(\mu)\colon\mu\in\cM_{\rm erg,<0}(f)\cup \cM_{\rm erg,>0}(f)\}\subset [\alpha_{\rm min},0)\cup(0,\alpha_{\rm max}]$;
\item for every $\alpha\in[\alpha_{\rm min},0)\cup(0,\alpha_{\rm max}]$, one has $\cL(\alpha)\ne\emptyset$;
\item for every $\alpha\in[\alpha_{\rm min},0)$ and $\beta\in(0,\alpha_{\rm max}]$, one has 
	\begin{equation}\label{newresvpproof}
		\cE_{<0}(\alpha)
		= \cH(\alpha)
		\quad\text{ and }\quad
		\cE_{>0}(\beta)
		= \cH(\beta)
	\end{equation}
	and
	\begin{equation}\label{resvpproof-minus}
		h_{\rm top}(f,\cL(\alpha))
		=\cH(\alpha) 	\quad\text{ and }\quad
		h_{\rm top}(f,\cL(\beta))
		= \cH(\beta).
	\end{equation}	
\end{itemize}
Moreover, it follows from property (E1) that $\cE_{<0}$ and $\cE_{>0}$ are both continuous functions. 

So what remains to study is the exponent zero. First recall that by Remark \ref{BM1almpa}, $0\in \{\chi^\c(\mu)\colon\mu\in\cM_{\rm erg}(f)\}$ and hence $\cL(0)\ne\emptyset$. Moreover, by Remark \ref{rem:flipflop}, $h_{\rm top}(f,\cL(0))>0$.
By Equations \eqref{newresvpproof}, \eqref{resvpproof-minus}, and property (E1), the function $\beta\mapsto h_{\rm top}(f,\cL(\beta))$ is  concave on $(0,\alpha_{\rm max})$, thus the following limit, as in \eqref{resvp0},
\[
\h^+(f)=\lim_{\beta\searrow0}h_{\rm top}(f,\cL(\beta))
\]	 
is well-defined. Analogously, $\h^-(f)$ is well-defined. 

To complete the proof, it remains to show the following two claims.

\begin{claim}\label{clem:imits}
	$\h(f)\eqdef\h^+(f)=\h^-(f)$.
\end{claim}

\begin{proof}
	By Equations \eqref{newresvpproof} and \eqref{resvpproof-minus}, for every sequence $\alpha_k\nearrow0$ there is a sequence $\mu_k$ of ergodic measures satisfying $\chi^\c(\mu_k)=\alpha_k$ and $h_{\mu_k}(f)\to \h^-(f)$ as $k\to\infty$. By Theorem \ref{the:goingotherside},  there is a sequence $\mu_k'$ of ergodic measures satisfying $\chi^\c(\mu_k')\searrow0$ and $h_{\mu_k'}(f)\to \h^-(f)$ as $k\to\infty$. This implies $\h^-(f)\le \h^+(f)$. An analogous argument gives $\h^+(f)\le \h^-(f)$. 
\end{proof}

\begin{claim}
	$h_{\rm top}(f,\cL(0))\le \h(f)$.
\end{claim}

\begin{proof}
	By Remark \ref{rem:flipflop}, one has $h_\ast\eqdef h_{\rm top}(f,\cL(0))>0.$ By Proposition \ref{pro:skel},  $f$ has the pre-skeleton property relative to $h$ and $0$, for every $h\in(0,h_\ast)$. 
	Then  by Proposition \ref{pro:skelhors}, for every $h\in(0,h_\ast)$, $ \gamma\in(0,h)$, and $\varepsilon>0$, there exist $\alpha\in(0,\varepsilon)$ and an ergodic measure $\mu$ such that  $\chi^\c (\mu)=\alpha$  and $h_{\mu}(f)\ge h-\gamma$, thus $h_{\rm top}(f,\cL(\alpha))\ge h-\gamma$ due to the restricted variational principle  \eqref{newresvpproof}. By the arbitrariness of $\varepsilon$ and Claim \ref{clem:imits}, one has $\h(f) \ge h-\gamma$. By the arbitrariness of $h\in(0,h_\ast)$  and $\gamma\in(0,h)$, one has $\h(f) \ge h_\ast= h_{\rm top}(f,\cL(0))$, proving the claim 
\end{proof}

The proof of Theorem \ref{theorem1} is now complete.
\qed

\subsection{Lower semi-continuity of entropy spectrum: Proof of Corollary~\ref{cor:semiconti}}\label{app2} 

Consider $f\in\BM^1(M)$ and  $\alpha\in(0,\alpha_{\rm max})$. 
Take $\varepsilon\in(0,h_{\rm top}(f,\cL(\alpha))/3)$. As, by Theorem \ref{theorem1}, the function $\beta\mapsto h_{\rm top}(f,\cL(\beta))$ is continuous on $(0,\alpha_{\rm max})$, for every $\delta\in(0,\alpha/4)$ sufficiently small, one has
\begin{equation}\label{pff1}
\min_{\beta\in[\alpha-4\delta,\alpha+4\delta]}h_{\rm top}(f,\cL(\beta))
\ge h_{\rm top}(f,\cL(\alpha))-\varepsilon.
\end{equation}
Moreover, by Theorem \ref{theorem1}, it holds $\cL(\beta)\ne \emptyset$ for all $\beta\in[\alpha-4\delta,\alpha+4\delta]$, and there are $\mu^\pm\in\cM_{\rm erg}(f)$ such that  
\[	
	\chi^\c(\mu^\pm)=\alpha\pm3\delta
	\quad\text{ and }\quad
	h_{\mu^\pm}(f)\ge h_{\rm top}(f,\cL(\alpha\pm3\delta))-\varepsilon,
\]
respectively. 

By Remark \ref{r.homorelated}, $\Per_{>0}(f)$ is an intersection class and hence we can invoke Lemma \ref{lem:approx1}. In particular, $M$ has the $(d^\uu+1)$-ergodic approximation property and for every $\widetilde\delta\in(0,\delta)$ there exist $f$-basic sets $\Xi^\pm\subset M$ in the $\widetilde\delta$-neighborhood of the support of $\mu^\pm$ with 
\begin{equation}\label{pff2}
h_{\rm top}(f,\Xi^\pm)\ge h_{\mu^\pm}(f)-\varepsilon
	~\text{ and } ~	|\chi^\c(\widetilde\mu^\pm)-\chi^\c(\mu^\pm)|<\widetilde\delta,\textrm{ for every $\widetilde\mu^\pm\in\cM_{\rm erg}(f|_{\Xi^\pm})$.}
\end{equation}
Denote by  $\mu^-_{\rm max}$  and  $\mu^+_{\rm max}$ the measures  of maximal entropy of  $f|_{\Xi^-}$ and  $f|_{\Xi^+}$ respectively. Then one has 
 $\chi^c(\mu^-_{\rm max})\in [\alpha-4\delta,\alpha-2\delta]$ and $\chi^c(  \mu^+_{\rm max}) \in[\alpha+2\delta,\alpha+4\delta].$  
 
 Recall that $\BM^1(M)$ is a $C^1$-open set (recall Remark \ref{r.defr.ourhypotheses}). 
Consider now a sufficiently small neighborhood $\cV$ of $f$ in $\BM^1(M)$ such that for every $g\in \cV$ the continuations   $\Xi^{g,\pm}$ of $\Xi^\pm$ are well defined. Note that   $h_{\rm top}(g,\Xi^{g,\pm})=h_{\rm top}(f,\Xi^\pm)$ and the measures of maximal entropy $\mu^{g,\pm}_{\rm max}$ of $g|_{\Xi^{g,\pm}}$ are the continuations of    $\mu^\pm_{\rm max}$. Moreover, as the exponents are given by integral of continuous functions and the corresponding bundle $E^\c$ varies continuously with $g$, their Lyapunov exponents also vary continuously. Thus for $\cV$ small, one has $\chi^c(\mu^{g,-}_{\rm max})\le \alpha-\delta$ and $\alpha+\delta\leq \chi^c(\mu^{g,+}_{\rm max}).$
Applying Theorem~\ref{theorem1} to $g\in\cV$, one gets the concavity and continuity of the function $\beta\in(0,\alpha_{\rm max}^g]\mapsto h_{\rm top}(g,\cL(\beta))$. As $\alpha\in \big(\chi^c(\mu^{g,-}_{\rm max}), \chi^c(\mu^{g,+}_{\rm max})\big)$, by concavity, one has 
\begin{align*}
	h_{\rm top}(g,\cL(\alpha))
	&\geq \min\big\{h_{\rm top}(g,\cL(\chi^c(\mu^{g,-}_{\rm max}))),h_{\rm top}(g,\cL(\chi^c(\mu^{g,+}_{\rm max})))\big\}
	\\
	&\geq \min\{h_{\mu^{g,-}_{\rm max}}(g), h_{\mu^{g,+}_{\rm max}}(g)\}
	\\
	&=\min\big\{h_{\rm top}(f,\Xi^-), h_{\rm top}(f,\Xi^+)\big\}\geq h_{\rm top}(f,\cL(\alpha))-3\e
	\quad\text{\small (by \eqref{pff1} and \eqref{pff2})}.
\end{align*} 

The proof for values in $(\alpha_{\rm min},0)$ is analogous and hence omitted. \qed

\section{Weak integrability and center distortion}\label{s:centerdistortion} 

In this section, we introduce two objects (disks and strips) that will be used when introducing blender-horseshoes in Section \ref{ss.bh}. We also prove a distortion result along the center direction.
Throughout this section, we assume that $f\in \PH_{\c =1}^1(M)$ and consider its partially hyperbolic splitting
$E^\ss\oplus E^\c \oplus E^\uu$. Let $E^\cs\eqdef E^\ss\oplus E^\c$ and $E^\cu\eqdef E^\c \oplus E^\uu$. Recall that the strong stable and strong unstable bundles give rise to the strong stable and strong unstable foliations $\cW^\ss$ and $\cW^\uu$, respectively. Given $r>0$ and $\star\in\{\ss, \uu\}$, denote by $\cW^\star(x,r)$ the ball centered at $x$ and of radius $r$ (relative to the induced distance on $\cW^\star(x)$). 

The case for the  center bundle is delicate.  As $E^\c$ is one-dimensional, and hence defines a continuous vector field, by Peano's theorem, at each point, there are $C^1$-curves tangent to $E^\c$.
Note that these center curves \emph{a priori} do not form a foliation, see for instance \cite{BonGogHamPot:20}. The cases of the (higher-dimensional) bundles $E^\cu$ and $E^\cs$ are even more delicate. However, we have the following. 

\begin{remark}[Weak integrability]\label{remweaint}
Assume $f\in\PH^1_{\c=1}(M)$. By \cite{BriBurIva:04}, the bundles $E^\cu$ and $E^\cs$ are \emph{weakly integrable} in the sense that at each point $x\in M$ there exist two immersed complete $C^1$ sub-manifolds $\cW_{\rm loc}^\cu(x)$ and $\cW_{\rm loc}^\cs(x)$ that contain $x$ and are everywhere tangent to $E^\cu$ and $E^\cs$, respectively. Moreover, there exists $\varepsilon_0>0$ such that for every $C^1$-curve $\gamma$ tangent to $E^\c$ and with length at most $\varepsilon_0$, the sets
\begin{equation}\label{defWgamma}
	\cW^\uu(\gamma,\varepsilon_0)\eqdef\bigcup_{z\in\gamma}\cW^\uu(z,\varepsilon_0) \quad\textrm{~and~}  \quad
	\cW^\ss(\gamma,\varepsilon_0)\eqdef\bigcup_{z\in\gamma}\cW^\ss(z,\varepsilon_0) 
\end{equation}
are $C^1$-submanifolds tangent to $E^\cu$ and $E^\cs$, respectively (\cite[proof of Proposition 3.4]{BriBurIva:04}). In particular, for every basic set of contracting type $\Lambda$, there is $r>0$ such that for every $x\in\Lambda$ every center curve centered at $x$ and of radius $r$ is contained in the stable manifold of $x$ of inner radius $r$,  denoted as $\cW^\s(x,r)$. Analogously for basic sets of expanding type.
\end{remark}

The objects in \eqref{defWgamma} are particular cases of the ones in the following definition.

\begin{definition}[Disks]\label{defstrips}
Given $\star\in\{\ss,\cs,\c,\cu, \uu\}$, a \emph{$\star$-disk} $\Delta$ is a (connected) $C^1$-sub\-mani\-fold tangent to $E^{\star}$, which is diffeomorphic to a disk of dimension $\dim(E^\star)$. The {\em{inner radius}} of $\Delta$, denoted by $\inrad (\Delta)$, is greater than $\delta$ if $\Delta$ contains a ball of radius $\delta$ (in the corresponding metric of $\Delta$ as a submanifold). When $\star=\c$, we simply speak of \emph{center curves}. 
\end{definition}

\begin{definition}[Strips and complete center curves]\label{defstripsbis}
A \emph{$\u$-strip} $S$ is a $\cu$-disk  that is foliated by $\uu$-disks. More precisely, there is a center curve $\gamma$ such that 
\[
	S
	= \bigcup_{x\in \gamma}\Delta^\uu(x),
\]
where $\Delta^\uu(x)$ is a $\uu$-disk containing $x$ in its interior. The \emph{$\uu$-boundary} of $S$ is the set 
\[
	\partial^\uu S
	\eqdef  \bigcup_{x\in \partial\gamma}\Delta^\uu(x).
\]
A center curve $\beta$ contained in $S$ is \emph{$(\c,S)$-complete} if its end-points belong to $\partial^\uu S$. Analogously, we define an \emph{$\s$-strip $S'$}, its \emph{$\ss$-boundary}, and being \emph{$(\c,S')$-complete}. The \emph{inner center width} of a ($\s$- or $\u$-)strip $S$ is the infimum of the lengths of the $(\c,S)$-complete curves of $S$, denoted by $\mathrm{w}^\c(S)$. 
\end{definition}

The next lemma is an immediate consequence of the uniform expansion/contraction along un-/stable foliations (see, for example, \cite[Lemma 2.4]{BocBonDia:16}).

\begin{lemma}[Center distortion control-time]\label{lemsomecount}
	For every $\delta>0$ and $\varepsilon>0$ there exists $\ell^\flat=\ell^\flat(\delta,\varepsilon)\in\bN$ such that for every $i\geq \ell^\flat$, every $x\in M$ and every $y\in f^{-i}(\cW^\uu(f^i(x),\delta))$, one has 
\[
	\big|\log\|D^\c f^{i}(x)\|-\log\|D^\c f^{i}(y)\|\big|<i\,\varepsilon.
\]
\end{lemma}

\begin{proof}
Given $\varepsilon>0$, by the uniform continuity of $\log\|D^\c f\|$, there exists $\eta>0$ such that 
\begin{equation}\label{more1}
	\big|\log\|D^{\c}f(x)\|-\log\|D^{\c}f(y)\|\big|<\frac\varepsilon2,\quad
\text{for every $x,y\in M$ with $d(x,y)<\eta$.}
\end{equation}
As $f^{-1}$ is uniformly contracting along the foliation $\cW^\uu$, there exists $n_1\in\bN$ such that
\[
	d(f^{-j}(x), f^{-j}(y))<\eta,
	\quad\textrm{for every $x\in M$, $y\in\cW^\uu(x,\delta)$, and  $j> n_1$.}  
\] 
Take $\ell^\flat\in\mathbb{N}$ such that 
\begin{equation}\label{more2}
	\frac{2n_1}{\ell^\flat}\max|\log\|D^\c f\||<\frac\e2.
\end{equation}
 Then for every $i\geq \ell^\flat$ and $x\in M$, $y\in f^{-i}(\cW^\uu(f^i(x),\delta))$, one has $f^j(y)\in\cW^\uu(f^j(x),\eta)$ for every $0\leq j<i-n_1$, and thus 
\[\begin{split}
	\Big|\log\frac{\|D^\c f^{i}(x)\|}{\|D^\c f^{i}(y)\|}\Big|
	&=\Big|\sum_{j=0}^{i-n_1-1}\log\frac{\|D^\c f(f^{j}(x))\|}{\|D^\c f(f^{j}(y))\|}
		+\sum_{j=i-n_1}^{i-1}\log\frac{\|D^\c f(f^{j}(x))\|}{\|D^\c f(f^{j}(y))\|}\Big|\\
	\text{\small{(by \eqref{more1})}}\quad	
	&\le (i-n_1)\frac\varepsilon2+2n_1\max|\log\|D^\c f\||\\	
	\text{\small{(by \eqref{more2})}}\quad	
	&< i\frac\varepsilon2+i\frac{2n_1}{\ell^\flat}\max|\log\|D^\c f\||
	<i\varepsilon,
\end{split}\]
proving the lemma.
\end{proof}

\section{Blender-horseshoes and minimality}\label{ss.bh}

In this section, we collect results about blender-horseshoes. We refrain from providing the complete definition and prefer to focus on their properties relevant here. The assumption $f\in \PH_{\c=1}^1(M)$ allows a less technical approach omitting cone fields. We will follow closely \cite[Section 2]{DiaGelSan:20} where blender-horseshoes are introduced in a context similar to the one here (partial hyperbolicity with one-dimensional center bundle). In Section \ref{ssdefblendhor}, we sketch their definition and properties. In Section \ref{secInteraction}, we discuss the interaction of blender-horseshoes of different types assuming minimality. Finally, in Section \ref{secauxstrip}, we investigate heteroclinic connections between strips.

\subsection{Unstable blender-horseshoes}\label{ssdefblendhor}

An \emph{unstable blender-horseshoe} $\blender^+=(\Lambda, \mathbf{C}, f)$  is a triple of a diffeomorphism $f$, a compact set $\mathbf{C}$, and a set $\Lambda\subset \mathrm{int} (\mathbf{C})$ which is the maximal invariant set of $f$ in $\mathbf{C}$ and which is hyperbolic, transitive, and has $\u$-index $\dim (E^\uu)+1$ (i.e., the center bundle $E^\c$ is an unstable direction of $\Lambda$).
The complete definition of an \emph{unstable blender-horseshoe} involves furthermore conditions (BH1)--(BH6) in \cite[Section 3.2]{BonDia:12} that we sketch below (in \cite{DiaGelSan:20} additional properties (BH7)--(BH9) are stated which follow from the former ones and from the constructions in \cite{BonDia:12}). 

\begin{remark}[Conditions (BH1)--(BH6)]
Conditions (BH1) and (BH3) state the existence of a Markov partition 
consisting of two disjoint ``sub-rectangles'' $\mathbf{C}_{\mathbb{A}}$  and  $\mathbf{C}_{\mathbb{B}}$  of $\mathbf{C}$ such that 
\[
\Lambda
= \bigcap_{i\in \mathbb{Z}} 
f^i(\mathbf{C}_{\mathbb{A}} \cup  \mathbf{C}_{\mathbb{B}} )
\]	  
and $f|_\Lambda$ is topologically conjugate to the full shift in two symbols (the symbols $\mathbb{A}$ and $\mathbb{B}$). In particular, $\Lambda$ contains two fixed points $P\in \mathbf{C}_{\mathbb{A}}$ and $Q\in \mathbf{C}_{\mathbb{B}}$, called the \emph{distinguished saddles} of the blender. 
Given $x\in \Lambda$, we denote by $\cW^\ss_{\mathbf{C}} (x)$ the connected component of $\cW^\ss(x) \cap \mathbf{C}$ containing $x$, analogously for $\cW^\uu_{\mathbf{C}}(x)$.  

The general definition of an unstable blender-horseshoe also involves cone fields. As $f\in  \PH_{\c =1}^1(M)$, condition (BH2) just means that $Df$ uniformly expands the vectors tangent to $E^\cu(x)$ for
$x\in\mathbf{C}_{\mathbb{A}}\cup\mathbf{C}_{\mathbb{B}}$.

The remaining conditions (BH4)--(BH6) involve the definitions of complete disks and strips and their iterations. A \emph{complete $\uu$-disk} $\Delta^\uu$ is a sufficiently large $\uu$-disk ``stretching over the entire set  $\mathbf{C}$''. Complete $\uu$-disks disjoint from $\cW^\ss_{\mathbf{C}}(P)$ are of two types:  they are either \emph{at the right} or \emph{at the left} of $\cW^\ss_{\mathbf{C}}(P)$. Analogously, for complete $\uu$-disks disjoint from  $\cW^\ss_{\mathbf{C}}(Q)$.
A complete $\uu$-disk is \emph{in-between} if it is disjoint from $\cW^\ss_{\mathbf{C}}(P)\cup \cW^\ss_{\mathbf{C}}(Q)$ and is at the right of $\cW^\ss_{\mathbf{C}}(P)$ and at the left of $\cW^\ss_{\mathbf{C}}(Q)$. This completes the description of the condition (BH4).
The key property of an unstable blender-horseshoe is that the image by $f$ of any complete $\uu$-disk in-between contains a complete $\uu$-disk in-between, see conditions (BH5)--(BH6). 
\end{remark}
 
A triple $\blender^-=(\Lambda, \mathbf{C}, f)$ is a \emph{stable blender-horseshoe} if $(\Lambda, \mathbf{C}, f^{-1})$ is an unstable blender-horseshoe.

The following is a consequence of properties (BH1)--(BH6).

\begin{lemma}[Iterations of complete $\uu$-disks in-between]\label{l.iterations}
Consider an unstable blender-horseshoe $\blender^+=(\Lambda, \mathbf{C},f)$ and a complete $\uu$-disk in-between $\Delta^\uu$. For every $n\ge 0$ there is $\Delta^{\uu}_n\subset \Delta^\uu$ such that
\begin{itemize}[leftmargin=0.8cm ]
\item $f^i (\Delta^\uu_n)$ is contained in some complete $\uu$-disk in-between for each $i=0, \dots, n$,
\item $f^n (\Delta^\uu_n)$ is a complete $\uu$-disk in-between.
\end{itemize}
\end{lemma}

Given a blender $\blender^+=(\Lambda, \mathbf{C},f)$, we now introduce complete $\u$-strips. A $\u$-strip $S$ is \emph{complete} if it is foliated by complete $\uu$-disks. A complete $\u$-strip is \emph{in-between} if it is foliated by complete $\uu$-disks in-between. The next lemma states the main property of iterations of $\u$-strips. We say that a complete $\u$-strip is \emph{$(\c,\blender^+)$-complete} if it intersects $\cW^\ss_{\mathbf{C}}(P)$ and $\cW^\ss_{\mathbf{C}}(Q)$, where $P$ and $Q$ are the distinguished saddles of the blender $\blender^+$. Note that there is $\mathrm{w}_{\blender^+}>0$ such that every $(\c,\blender^+)$-complete $\u$-strip $S$ has inner center width at least $\mathrm{w}_{\blender^+}$, that is, satisfies $\mathrm{w}^\c(S)\ge \mathrm{w}_{\blender^+}$.
 
 The following lemma is essentially \cite[Proposition 2.3]{DiaGelSan:20} reformulating \cite[Lemma 1.7]{BonDia:96}.

\begin{lemma}[Iterations of $\u$-strips in-between]\label{l.iterationsstrip}
Consider an unstable blender-horseshoe $\blender^+$. There are $r>0$ and a function $M_{\b} \colon (0,r)\to \mathbb{N}$ such that for every $\delta \in (0,r)$ and every complete $\u$-strip $S$ in-between with inner center width $\mathrm{w}^\c(S) \ge \delta$ the following holds:  for every $m \ge M_\b (\delta)$ there is $S_m \subset S$ such that
\begin{itemize}[leftmargin=0.8cm ]
\item $f^i(S_m)\subset \mathbf{C}$ for every $i=0,\dots, m$,\\[-0.3cm]
\item $f^m(S_m)$ is a $(\c,\blender^+)$-complete $\u$-strip.
\end{itemize}
\end{lemma}

\subsection{Interaction between minimality and blender-horseshoes}\label{secInteraction}

 In this section, we additionally require that $f\in\BM^1(M)$, that is, $f\in\PH^1_{\c=1}(M)$, both strong foliations are minimal, and there are unstable and stable blender-horseshoes. The following is an immediate consequence of the minimality of the strong stable and the strong unstable foliations, the compactness of $M$, and the fact that $f$ uniformly expands (contracts) the vectors in $E^\uu$ (in $E^\ss$). 

\begin{lemma}\label{lem:minfol}
	Let $f\in \BM^1(M)$. For every $\delta>0$,  there is $\ell_0= \ell_0(\delta)\in\bN$ such that 	for every $\ell\geq \ell_0$  and for every $\uu$-disk $\Delta^{\uu}$, $\ss$-disk $\Delta^{\ss}$, $\cs$-disk $\Delta^{\cs}$, and $\cu$-disk $\Delta^{\cu}$, with
	\[
	\min\big\{\inrad (\Delta^\uu),
	\inrad (\Delta^\ss),
	\inrad (\Delta^\cs),
	\inrad (\Delta^\cu)\big\}>\delta,
	\]
it holds
	\[
	f^\ell(\Delta^\uu)\pitchfork \Delta^\cs\ne\emptyset,
	\quad
	f^{-\ell}(\Delta^\ss)\pitchfork \Delta^\cu \ne\emptyset,
	\quad\mbox{and}\quad
	f^{-\ell}(\Delta^\cu)\pitchfork \Delta^\cs \ne\emptyset,
	\]	 
	where the latter contains center curves.
\end{lemma}

Although it is not required in the definition of an unstable blender-horseshoe, for further use we introduce $\ss$-disks and $\s$-strips which are complete and in-between (recall Definition \ref{defstrips}). Although topologically their definitions are analogous to their $\u$-counterparts, we point out that \emph{a priori} Lemmas \ref{l.iterations} and \ref{l.iterationsstrip} do not extend to these disks and strips.

\begin{remark}[Complete and in-between $\ss$-disks and $\s$-strips]\label{r.strips}
	Let  $\blender^+=(\Lambda, \mathbf{C},f)$ be an unstable blender-horseshoe.
	A $\ss$-disk $\Delta^\ss$ is \emph{complete} if it is a sufficiently large $\ss$-disk containing a point of $\mathbf{C}$.  There are two types of complete $\ss$-disks disjoint from $\cW^\uu_{\mathbf{C}}(P)$:  \emph{at the right} and \emph{at the left} of $\cW^\uu_{\mathbf{C}}(P)$. Analogously, for complete $\ss$-disks disjoint from  $\cW^\uu_{\mathbf{C}}(Q)$. Thus, we can speak of \emph{complete $\ss$-disks in-between}.
	Similarly as above, we can introduce \emph{complete $\s$-strips} and \emph{complete $\s$-strips in-between}.
	
	Note that if $S=S^\cu$ is a $(\c,\blender^+)$-complete $\u$-strip then for every complete $\s$-strip $S^\cs$  in-between the sets $S^\cu$ and $S^\cs$ intersect transversally along a $(\c, S^\cs)$-complete curve.
\end{remark}

\begin{lemma}\label{lem:insideblenders}
	Let $f\in \BM^1(M)$ and $\blender^+=(\Lambda,\mathbf C,g=f^k)$ be  an unstable blender-horseshoe. Then for every $r>0$, there is $N(r)\in\bN$, satisfying $N(r)\to\infty$ as $r\to0$, such that 
\begin{itemize}[leftmargin=0.8cm ]
\item[(1)] for every $\ss$-disk $\Delta^\ss$ of inner radius $r$ and $\uu$-disk $\Delta^\uu$ of inner radius $r$  and  for every $N\ge N(r)$,
\[
	f^{-N}(\Delta^\ss)	
	\quad\text{ and }\quad 
	f^N(\Delta^\uu)
\]	
contain a complete in-between $\ss$-disk and $\uu$-disk, respectively;
\item[(2)]  for every $N\ge N(r)$ there is $R=R(N)>0$ such that for every center curve $\gamma$ of length less than $R$,  the sets
\[
	f^{-N}\big(\cW^\ss(\gamma,r)\big)
	\quad\text{ and }\quad
	f^N\big(\cW^\uu(\gamma,r)\big)
\] 
contain a complete $\s$-strip $S^\cs$ in-between  and a complete $\u$-strip  $S^\cu$ in-between respectively, such that 
\[
	\partial^\ss(S^\cs)\subset \cW^\ss(\partial(f^{-N}(\gamma)))
		\quad\text{ and }\quad
	\partial^\uu(S^\cu)\subset \cW^\uu(\partial(f^N(\gamma)));
\]
\item[(3)] for every  complete $\u$-strip  $S^\cu$ in-between and every $m\ge M_\b(\mathrm{w}^\c(S^\cu))$ there is a compact subset $S^\cu_m\subset S^\cu$ such that
	\begin{itemize}
		\item[--] $g^i(S^\cu_m)\subset\mathbf C$ for every $ i=0,\ldots,m$,
		\item[--] $g^m(S^\cu_m)$ is a $(\c,\blender^+)$-complete $\u$-strip and therefore has $\mathrm{w}^\c\big(g^m(S^\cu_m)\big)\ge\mathrm{w}^\c(\blender^+)$.
	\end{itemize}
	Moreover, $g^m(S^\cu_m)$ transversely intersects $S^\cs$ along a center curve which is $(\c,S^\cs)$-complete.
	\end{itemize}
	
The analogous statement holds for a stable blender-horseshoe. 
\end{lemma}

\begin{proof}
	The assertion (1) and (2) follow from uniform contraction and expansion of the strong foliations and the minimality of strong foliations. 
	The first part of assertion (3) follows applying Lemma \ref{l.iterationsstrip} to the strip in-between $S^\cu$. The second part about the intersection of $g^m(S^\cu_m)$ with $S^\cs$ follows from the last assertion of Remark~\ref{r.strips}.
\end{proof} 

The following remark is a consequence of the continuity of the strong foliations.

\begin{remark}\label{r.etas}
Consider a center curve $\gamma$ of length $r^\c$ and the complete $\s$-strip $S^\cs(\gamma,r)$ in-between  and the complete $\u$-strip  $S^\cu(\gamma,r)$ in-between,
$$
S^{\cs}(\gamma,r) \subset f^{-N(r)}\big(\cW^\ss(\gamma,r)\big)
\quad
\mbox{and}
\quad 
S^\cu(\gamma,r)\subset f^{N(r)} \big(\cW^\uu(\gamma,r)\big),
$$
provided by Lemma~\ref{lem:insideblenders}. We let
\begin{equation}\label{eq:etas}
	\eta^\s (r^\c,r)
	\eqdef \inf \big\{ \mathrm{w}^\c\big(S^{\cs}(\gamma,r)\big)\colon
		\mbox{$\gamma$ is a center curve of length $r^\c$}\big\},
\end{equation}
analogously, define $\eta^\u(r^\c,r)$. Note that $\eta^\star (r^\c, r)>0$, $\star=\s,\u$.
\end{remark}

Now, for completeness, we show that the set $\BM^1(M)$ is $C^1$-open as stated in Remark~\ref{r.defr.ourhypotheses}. Here we largely follow the arguments in \cite[Theorem 3]{BonDiaUre:02} using an additional ingredient from \cite{BriBurIva:04}.

\begin{lemma}\label{l.robustness-of-BM}
	Let $f\in\BM^1(M)$. There exists a $C^1$-neighborhood $\mathcal{U}$ of $f$ such that $\mathcal{U}\subset\BM^1(M).$
\end{lemma}

\begin{proof}
Let $\blender^+=(\Lambda^{+}, \mathbf C^+,f^n)$ and $\blender^-=(\Lambda^{-}, \mathbf C^-,f^m)$ be unstable and stable blender-horseshoes, respectively. For simplicity, we will assume that $n=m=1$.

By the minimality of the strong stable and unstable foliations of $f$, there  exists $L>0$ such that for every $x\in M$, 
\begin{itemize}
\item[(R1)] $\cW^\ss(x,L)$ intersects every $(\c,\blender^+)$-complete $\u$-strip transversely in its interior;
\item[(R2)] $\cW^\uu(x,L)$ intersects every $(\c,\blender^-)$-complete $\s$-strip transversely in its interior;
\item[(R3)] $\cW^\ss(x,L)$ contains a complete $\ss$-disk in between for $\blender^-$;
\item[(R4)] $\cW^\uu(x,L)$ contains a complete $\uu$-disk in between for $\blender^+$.
\end{itemize}
By \cite[Lemma 3.10]{BonDia:12}, there is a $C^1$-neighborhood $\cU$ of $f$ such that every $g\in\cU$ has an unstable blender-horseshoe $\blender^+_g=(\Lambda_g^{+},\mathbf  C^+,g)$ and a stable blender-horseshoe $\blender^-_g=(\Lambda_g^{-},\mathbf  C^-,g)$, where $\Lambda^{+}_g$ and $\Lambda^{-}_g$ are the continuations of $\Lambda^{+}$ and $\Lambda^{-}$ of $g$ in $\mathbf C^+$ and in $\mathbf C^-$, respectively. By the uniform continuity of strong unstable and stable distributions and the fact that being partially hyperbolic is a $C^1$-open property, up to shrinking $\cU$, one can assume that each $g\in\cU$ is partially hyperbolic with a one-dimensional center bundle, and the corresponding properties (R1)--(R4) also hold for every $g\in\cU$.  We denote by $TM=E^\ss_g\oplus E^\c_g\oplus E^\uu_g$ the partially hyperbolic splitting for $g$ and by $\cW^\ss_g$ and $\cW^\uu_g$ the strong stable and unstable foliations of $g$, respectively. 

Now, we prove the minimality of the strong stable foliation $\cW^\ss_g$ of $g$. The proof for the minimality of  $\cW^\uu_g$ is analogous. Take   $x\in M$ and a nonempty open set $U\subset M$.  Let $y\in U$ and take  a center curve $\gamma\subset U$ centered at $y$. Take $\delta>0$ small such that  $\cW_g^\uu(\gamma,\delta)\eqdef \bigcup_{z\in\gamma}\cW_g^\uu(z,\delta)$ is contained in $U$. By \cite{BriBurIva:04}, $\cW_g^\uu(\gamma,\delta)$ is a $C^1$-submanifold tangent everywhere to $E_g^\c\oplus E_g^\uu$.  By the uniform expansion of $g$ along the strong unstable bundle, there exists $k\in\bN$ such that $\cW_g^\uu(g^k(y),L)\subset g^k(\cW_g^\uu(y,\delta/2))$. By property $({\rm R4})$, $g^{k}(\cW_g^\uu(\gamma,\delta))$ contains a complete $\u$-strip in between for $\blender_g^{+}$. By Lemma~\ref{l.iterationsstrip}, there exists $l\in\bN$ such that  $g^{k+l}(\cW_g^\uu(\gamma,\delta))$ contains a $(c,\blender_g^{+})$-complete $\u$-strip.
Applying property $({\rm R1})$ to $g^{k+l}(x)$, one gets that $\cW_g^\ss(g^{k+l}(x),L)$ intersects $g^{k+l}(\cW_g^\uu(\gamma,\delta))$ transversely. This implies that $\cW_g^\ss(x)$ intersects $\cW_g^\uu(\gamma,\delta)$ and hence $U$. By the arbitrariness of $x$ and $U$, one has that the strong stable foliation of $g$ is minimal.  
\end{proof}

We close this section recalling the results in \cite{BocBonDia:16} implying that $h_{\rm top}(f,\cL(0))>0$ for every $f\in \BM^1(M)$.

\begin{remark}[Flip-flop configurations and blender-horseshoes]\label{f.flipblender}
Let us briefly describe what is a {\em{split flip-flop configurations}} and explain why every $f\in \BM^1(M)$ displays such a configuration. By \cite[Theorem 5]{BocBonDia:16}, there exists a compact $f$-invariant set with positive topological entropy contained in the set $\cL(0)$.
	
In very rough terms, and in our partially hyperbolic setting, a \emph{flip-flop configuration} involves an 
unstable blender-horseshoe $(\Lambda,\mathbf C,f^n)$%
\footnote{More precisely, in \cite{BocBonDia:16} \emph{dynamical blenders} are used, see \cite[Definition 3.11]{BocBonDia:16}. Note that blender-horseshoes are dynamical blenders.}  and a saddle $Q$ of $\u$-index $d^\uu$. The required conditions are that the strong unstable manifold of $\Lambda$ intersects the stable manifold of $Q$ transversely and the unstable manifold of $Q$ contains a complete $\uu$-disk in-between. In our context, these conditions are guaranteed by the partial hyperbolicity and by the minimality of strong foliations. Further partial hyperbolicity-like conditions are automatically satisfied in our setting. Indeed, here a stronger form of a so-called \emph{split flip-flop configuration} holds, see \cite[Definition 4.7]{BocBonDia:16}. This means that the blender-horseshoe and the saddle have ``compatible'' partially hyperbolic splittings. 
\end{remark}

\subsection{Heteroclinic connections between strips}\label{secauxstrip}

The following is a key step in the final construction of the fractal set in  $\cL(0)$, see Section \ref{s.final-fractal}. It is a consequence of the minimality of the strong stable foliations.

Note that, for $f\in\BM^1(M)$, there are infinitely many points $q$ satisfying the hypothesis of the following lemma.

\begin{lemma}\label{remfirstLambda}
Let $f\in\BM^1(M)$. Let $q$ be a hyperbolic periodic point of $f$ of expanding type. There are positive constants $r^\c,r^\s,r^\u_1,r^\u_2>0$ such that
\begin{enumerate}[ leftmargin=0.5cm ]
\item[$\bullet$] $q$ has an unstable manifold $\cW^\u(q,r^\u_1)$ of inner radius $r^\u_1$ which is tangent  to $E^\c\oplus E^\uu$.
\item[$\bullet$] for every center curve $\gamma$ of  inner radius
 less than $r^\c$ and every  $(\c,\cW^\uu(\gamma,r^\u_2))$-complete center curve $\gamma'$, there exists a $(\c,\cW^\ss(\gamma',r^\s))$-complete center curve contained in  $\cW^\u(q,r^\u_1)$.
\end{enumerate}
\end{lemma}

\begin{proof}
There exists $r^\u_1>0$ such that the unstable manifold of $q$ contains a disc $\cW^\u(q,r^\u_1)$ centered at $q$ and of inner radius $r^\u_1$. By the partial hyperbolicity of $f$, $\cW^\u(q,r^\u_1)$ is tangent to $E^\c\oplus E^\uu$ everywhere. 
	
 By the minimality of the strong stable foliation, there exists $r^\s>0$ such that $\cW^\ss(x,r^\s)$ has a nonempty transverse intersection with  $\cW^{\u}(q,r^\u_1/4)$ for any $x\in M$. By the uniform continuity of the strong stable bundle, there exists $r^\c>0$ such that for every center curve $\gamma$ of inner radius
  no more than $r^\c$, one has that $\cW^\ss(\gamma,r^\s)$ intersects $\cW^{\u}(q,r^\u_1/2)$ into a center curve which is $(\c,\cW^\ss(\gamma,r^\s))$-complete. By the uniform continuity of the strong stable and strong unstable bundles, there exists $r^\u_2>0$ such that for any $(\c,\cW^\uu(\gamma,r^\u_2))$-complete center curve $\gamma'$,  there exists a $(\c,\cW^\ss(\gamma',r^\s))$-complete center curve contained in $\cW^{\u}(q,r^\u_1).$  
\end{proof}

\begin{remark}[Choice of pivotal saddle point]\label{pivotal}
	In what follows, we fix some ``pivotal'' hyperbolic periodic point $q_0$ of expanding type and let $r^\c,r^\s,r^\u_1,r^\u_2$ be as in Lemma \ref{remfirstLambda}. 
\end{remark}

\section{Concatenation of center curves}\label{sec:concat}

In this section, we present a result concerning the concatenation of center curves, see Proposition~\ref{proLem:key}. Throughout this section, we assume $f\in\BM^1(M)$ and consider the stable blender-horsesho $\blender^-\eqdef(\Lambda^-, \mathbf{C}^-,f^{n_-})$ given by hypothesis (B).
For simplicity, assume $n_-=1$. 

We start with some preliminaries. Given $r>0$, let $N(r)\in\bN$  and $R(N(r))>0$ be as in Lemma \ref{lem:insideblenders} associated to $\blender^-$. Note that $N(r) \to \infty$ as $r\to 0$. Recall the definition of $M_\b$ in Lemma \ref{l.iterationsstrip}. Note that in the proof of Proposition \ref{proLem:key} below, we will use the versions of the results for backward iterations of $\s$-strips in the stable blender-horseshoe $\blender^-$. Recall also the notations $\cW^\ss(\gamma,r)$ and $\cW^\uu(\gamma,r)$  in Remark \ref{remweaint} and the numbers $\eta^{\s,\u}$ associated to these strips in \eqref{eq:etas}.

	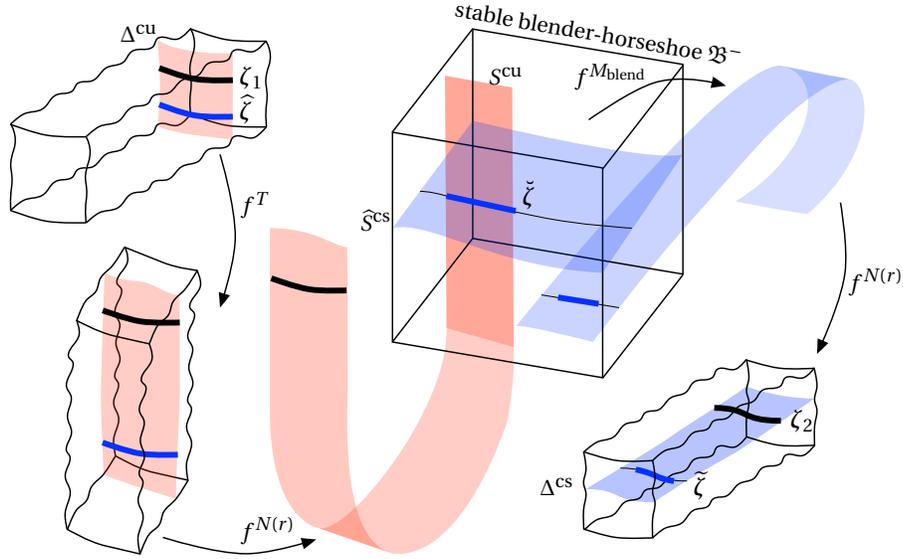
\begin{figure}[h]
		\begin{overpic}[scale=.55]{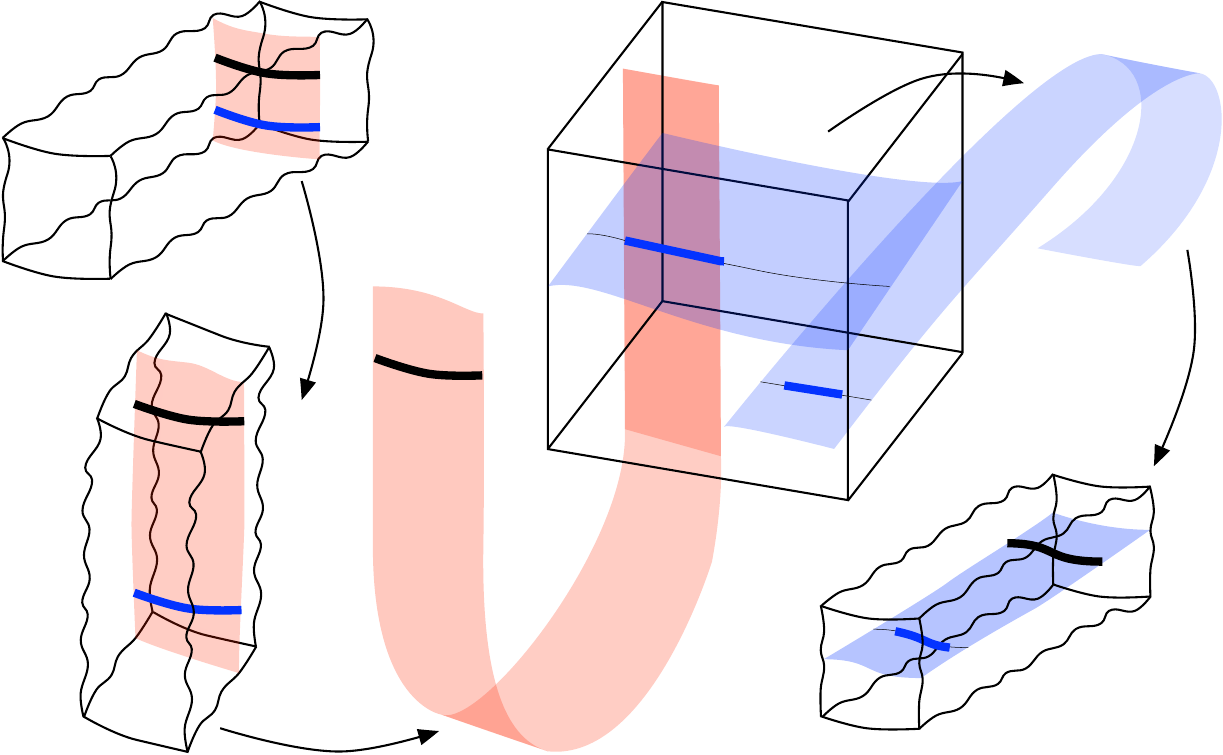}
			\put(52,63){\rotatebox{-10}{\small\text{stable blender-horseshoe }$\blender^-$} }
			\put(27,55){$\zeta_1$}
			\put(91.3,15){$\zeta_2$}
			\put(27,51){$\widehat\zeta$}
			\put(13,60){\small$\Delta^\cu$}
			\put(27,40){\small$f^T$}
			\put(41,38){\small$\widehat S^\cs$}
			\put(60,41){$\breve\zeta$}
			\put(56,55){\small$ S^\cu$}
			\put(27,2){\small$f^{N(r)}$}
			\put(98,31){\small$f^{N(r)}$}
			\put(66,55){\small$f^{M_\b}$}
			\put(62,7){\small$\Delta^\cs$}
			\put(80,7){$\widetilde\zeta$}
	\end{overpic}
		\caption{Proof of Proposition \ref{proLem:key}}
		\label{fig.2}
	\end{figure}
	
We present a ``concatenation procedure'' of contracting curves. 

\begin{proposition}[Concatenation of center curves]\label{proLem:key}\label{changed into proposition, write some hypothesis?}
Let $\zeta_1$ be a center curve of length $r^\c_1>0$ such that there exist $K>1$ and  $\chi<0$ so that for every $n\in\bN$ and $y\in\zeta_1$, 
\begin{equation}\label{eq:crescimento}
	\lVert D^\c f^n(y)\rVert
	\leq K e^{n\chi}.
\end{equation}
Let $r>0$ and $\zeta_2$ be a center curve of length $r^\c_2\in(0,R(N(r)))$. Then there are integers 
\begin{equation}\label{eq:mn}\begin{split}
	m=m(r^\c_2,r)&\eqdef 2N(r)+M_\b(\eta^\s(r^\c_2,r)) 
	\quad \mbox{and}\\
	T^\sharp=
	T^\sharp(r^\c_1,K,\chi,r)&
	\eqdef \left\lceil\frac{1}{\chi}\big(\log R(N(r))-\log K-\log r^\c_1\big)\right\rceil+1,	
\end{split}\end{equation}
such that for every $T\geq T^\sharp$, there exist center curves $\widehat\zeta$ and $\widetilde\zeta\eqdef f^{T+m}(\widehat\zeta)$ satisfying:
\begin{enumerate}[ leftmargin=0.9cm ]
	\item[(i)] $\widehat\zeta$ is $(\c,\cW^\uu(\zeta_1,r))$-complete;\\[-0.2cm]
	\item[(ii)] $f^i(\widehat\zeta)\subset \cW^{\uu}(f^i(\zeta_1),r)$ for every $i=0,\dots,T$;\\[-0.2cm]
	\item[(iii)] $\displaystyle \widetilde\zeta\subset \cW^\ss(\zeta_2,r)$; \\[-0.2cm]
	\item[(iv)]	$f^{-(T+m)}\big(\cW^\uu(\widetilde\zeta,2r))\big)\subset \cW^\uu(\widehat\zeta,r)$.
\end{enumerate}
\end{proposition}

\begin{proof}
Given $r>0$, consider the $\u$-strip and the $\s$-strip given by 
\[
	\Delta^\cu
	\eqdef \cW^\uu(\zeta_1,r)
	\quad\text{ and }\quad
	\Delta^\cs
	\eqdef \cW^\ss(\zeta_2,r),
\]
respectively.
By the results in Section \ref{s:centerdistortion}, $\Delta^\cu$ and $\Delta^\cs$ are tangent to $E^\c \oplus E^\uu$ and $E^\ss \oplus E^\c $ respectively. The proof has two parts. First, considering the backward iterates of $\Delta^\cs$ and the properties of the blender, we get a $(\c, \blender^-)$-complete $\s$-strip $\widehat S^\cs$. Second, considering the forward iterates of $\Delta^\cu$ we get a complete $u$-strip of $\blender^-$ in between. The arguments in this proof are illustrated in Figure \ref{fig.2}.

For the first part, as, by hypothesis, $r^\c_2\in(0,R(N(r)))$, we can apply Lemma \ref{lem:insideblenders} (in its version for the stable blender $\blender^-$) to $\Delta^\cs$, and  get from item (2) in Lemma \ref{lem:insideblenders} that
\begin{enumerate}[ leftmargin=0.9cm ]
\item $f^{-N(r)}(\Delta^\cs)$ contains a complete $\s$-strip $S^\cs$ in-between (relative to $\blender^-$),
\item the  $\ss$-boundary of $S^{\cs}$ is contained in $\cW^\ss(\partial(f^{-N(r)}(\zeta_2)))$.
\end{enumerate} 
By Remark~\ref{r.etas}, the  $\s$-strip $S^\cs$ has inner center width at least $\eta^\s(r^\c_2,r)>0$. Hence,  by item (3) in Lemma \ref{lem:insideblenders} applied to $S^\cs$, the set $f^{-M_\b(\eta^\s(r^\c_2,r))}(S^\cs)$ contains a $(\c,\blender^-)$-complete $\s$-strip $\widehat S^\cs$. 
This completes the first part.
	
For the second part, take $T^\sharp$ as in \eqref{eq:mn}. This choice implies that for every $T\ge T^\sharp$,
\begin{equation}\label{eq:length}
		K r^\c_1 e^{T\chi}
		<R(N(r)).
\end{equation}
It follows from our hypothesis \eqref{eq:crescimento} together with \eqref{eq:length} that for every $T\geq T^\sharp$, the length of the center curve$f^T(\zeta_1)$ is less than $R(N(r))$. Note that $f^T(\Delta^\cu)$ is a $\u$-strip that can be written as 
	\[
	f^T(\Delta^\cu)=\bigcup_{x\in\gamma}\Delta^\uu(x),	
	\quad\text{ where }\quad
	\gamma\eqdef f^T(\zeta_1)
	\]	 
and each of the sets $\Delta^\uu(x)$ is a $\uu$-disk. Note that, due to the expansion in the strong unstable manifold, each disk $\Delta^\uu(x)$ has inner radius bigger than $r$. Thus, we can apply Lemma \ref{lem:insideblenders}  to $f^T(\Delta^\cu)$, and get that
	\[
	f^{N(r)}\big(f^T(\Delta^\cu)\big)
	\] 
	contains a complete $\u$-strip $S^\cu$  in-between (relative to $\blender^-$) whose $\uu$-boundary is contained in 
	\[
	\cW^\uu\big(\partial (f^{N(r)+T}(\zeta_1))\big).
	\]	 
		By Remark \ref{r.strips}, $\widehat S^\cs\pitchfork S^\cu$ contains a center curve ${\breve \zeta}$ which is $(\c,S^\cu)$-complete. Let
	\[
	\widehat\zeta
	= f^{-(T+N(r))}\big(\breve \zeta \big).
	\]
	By construction, the center curve  $\widehat\zeta$ is $(\c,\Delta^\cu)$-complete and
\[
	f^i(\widehat\zeta)\subset \cW^{\uu}(f^i(\zeta_1),r)
	\quad\text{ for }\quad i=0,\dots,T.
\]	
Moreover, the choice of $T^\sharp$ also implies
\[
	\widetilde\zeta\eqdef f^{T+m}(\widehat\zeta)
		\subset \Delta^\cs
\]
where $m$ is as in \eqref{eq:mn}.
	This proves items (i), (ii), and (iii) of the proposition. The item (iv) follows from the uniform expansion of the foliation $\cW^{\uu}$.  	The proof of the proposition is now complete.
\end{proof}

\section{Approximation by basic sets}\label{secapproximating}

In this section, we start by stating some finer properties of basic sets. In particular, we state an approximation result of hyperbolic measures by basic sets (see Lemma \ref{l.horseshoe-with-large-entropy}). We also formalize a concatenation procedure for separated orbit segments which will be a key step to ``produce entropy'' in our constructions (see Proposition \ref{p.separation-in-symbol}). We point out that in this section we only require $f\in\PH^1_{\c =1}(M)$.

To start, with a slight abuse of notation, given two subsets $A,B\subset M$, denote their distance by
\[d(A,B)\eqdef\inf\big\{ d(x,y)|x\in A, y\in B \big\}.\]
For $n\in\bN$ and $\varepsilon>0$, recall the standard definition of an \emph{$(n,\e)$-Bowen ball centered at $x$}
\begin{equation}\label{defBowBal}
	B_n(x,\varepsilon)\eqdef\big\{y\in X\colon d_n(x,y)<\varepsilon\big\},
	\quad\text{ where }\quad
	d_n(x,y)\eqdef\max_{k=0,\ldots,n-1}d(f^k(x),f^k(y)).
\end{equation}

\begin{remark}[Constant of entropy expansiveness]\label{remexpansive}
For a partially hyperbolic diffeomorphism $f$ with one-dimensional center, by \cite{LiViYa:13,DiFiPaVi:12}, $f$ is $\delta_0$-entropy expansive for some $\delta_0>0$, that is, for every $x\in M$,
\[
	h_{\rm top}(f,B_\infty(x,\delta_0))=0, \quad\text{ where }\quad
		B_\infty(x,\delta_0)
	\eqdef\bigcap_{n\in\bN}B_n(x,\delta_0).
\]
 
\end{remark}

The following result we state without proof.

\begin{lemma}\label{l.local-stable-manifold}
	Let $f\in\PH^1_{\c =1}(M)$ and  $\delta_0>0$ be as in Remark \ref{remexpansive}. Let $\Lambda$ be a basic set of contracting type of $f$. There exist $r^\sharp=r^\sharp(\Lambda)\in(0,\delta_0/4)$ and $a^\sharp=a^\sharp(\Lambda)\in\bN$ such that
\begin{itemize}[ leftmargin=0.5cm ]
\item[(1)]  any center curve $\gamma$ centered at a point $x\in\Lambda$ of inner radius $r^\sharp$ is contained in the stable manifold $\cW^\s (x)$ of $x\in\Lambda$ and, in particular, in the stable manifold $\cW^\s (x,r^\sharp)$;
\item[(2)] for every $n\geq a^\sharp$, every $x\in \Lambda$ and every $r\in(0, r^\sharp)$, one has 
\begin{equation}\label{agree}
	f^n(\cW^\s(x,r))\subset \cW^\s(f^n(x), r).
\end{equation}	
\item[(3)] for every $x,x'\in\Lambda$, $r\in(0,r^\sharp)$, and center curves $\xi=\xi(x,r)$ and $\xi'=\xi(x',r)$ of inner radius $r$ and  centered at $x$ and $x'$ respectively, one has 
\begin{equation}\label{defdist}
	\text{$i\ge a^\sharp$}
	\quad\Longrightarrow\quad
	d(f^i(\xi),f^i(\xi'))
	\ge d(f^i(x),f^i(x'))-2r.
\end{equation}
\end{itemize}
\end{lemma} 

The following one shows the importance of basic sets to study hyperbolic measures. 

\begin{lemma}\label{l.horseshoe-with-large-entropy}
Let $f\in\PH^1_{\c =1}(M)$. For every $\nu\in\cM_{\rm erg,<0}(f)$, $\varepsilon>0$ and $\delta>0$, there exist a basic set $\Lambda$ of contracting type and $K>0$ having the following properties: 
\begin{enumerate}[ leftmargin=0.7cm ]
\item\label{l.horseshoe-with-large-entropy-i} $h_{\rm top}(f,\Lambda)>h_\nu(f)-\varepsilon$,
\item\label{l.horseshoe-with-large-entropy-iii} given $r^\sharp=r^\sharp(\Lambda)$ as in Lemma~\ref{l.local-stable-manifold}, every center curve $\gamma$ centered at a point $x\in\Lambda$ of inner radius $r^\sharp$ is contained in the stable manifold of $x$,
\item\label{l.horseshoe-with-large-entropy-iv} for every center curve $\gamma$ centered at a point $x\in\Lambda$ of inner radius $r^\sharp$ and every $y\in \cW^{\ss}\big(\gamma,\delta\big)$, and $n\in\mathbb{N}$, one has 
\[
	K^{-1}\cdot e^{n(\chi^\c(\nu)-\e)}\leq\|D^\c f^n(y)\|
	\leq K\cdot e^{n(\chi^\c(\nu)+\e)}.
\]
\end{enumerate}
\end{lemma}

\begin{proof}
The hyperbolic ergodic measure $\nu$ can be approached by basic sets in the weak$\ast$-topology and in entropy (see for instance \cite{KatHas:95,Gel:16}). Thus, there exists a basic set of contracting type $\Lambda$ such that 
\begin{itemize}[ leftmargin=0.7cm ]
\item[(a)] $|h_{\rm top}(f,\Lambda)-h_{\nu}(f)|<\e$;
\item[(b)] for every $\nu'\in\cM_{\rm erg}(f|_{\Lambda})$, one has
\[
	\Big|\int\log\|D^\c f\|\ud\nu'-\chi^\c(\nu)\Big|<\e/4.
\]
\end{itemize}
Item (a) implies item \eqref{l.horseshoe-with-large-entropy-i} of the assertion. Item (b) implies that there exists $\widetilde K>0$ such that for every $x\in\Lambda$ and every $n\in\bN$, one has 
\[
	\widetilde K^{-1}\cdot e^{n(\chi^\c(\nu)-\e/4)}
	\leq\|D^\c f^n(x)\|\leq \widetilde K\cdot e^{n(\chi^\c(\nu)+\e/4)}.
\]
By Lemma~\ref{l.local-stable-manifold}, every center curve centered at a point  $x\in\Lambda$ of inner radius $r^\sharp$ is contained in $\cW^\s(x,r^\sharp)$, proving item \eqref{l.horseshoe-with-large-entropy-iii}. It follows from the uniform continuity of the center bundle, that there exists $\widehat K\geq 1$ such that for every $x\in\Lambda$, $y\in \cW^\s(x,r^\sharp)$, and $n\in\bN$, one has
\[
	\widehat K^{-1}\cdot e^{n(\chi^\c(\nu)-\e/2)}
	\leq\|D^\c f^n(y)\|
	\leq\widehat K\cdot e^{n(\chi^\c(\nu)+\e/2)}.
\]
As $f$ is uniformly contracting along the strong stable foliation, there is $K\geq 1$ such that for every center curve $\gamma$ centered at a point $x\in\Lambda$ of inner radius $r^\sharp$, $y\in \cW^{\ss}\big(\gamma,\delta\big)$, and $n\in\mathbb{N}$, one has 
 \[ 
 	K^{-1}\cdot e^{n(\chi^\c(\nu)-\e)}
	\leq\|D^\c f^n(y)\|
	\leq K\cdot e^{n(\chi^\c(\nu)+\e)}.
\]
This proves \eqref{l.horseshoe-with-large-entropy-iv} and finishes the proof of the lemma.
\end{proof}

We now state the main result in this section.

\begin{proposition}\label{p.separation-in-symbol}
Let $f\in\PH^1_{\c =1}(M)$ and $\delta_0>0$ be as in Remark \ref{remexpansive}. Let $\Lambda$ be a basic set of contracting type of $f$ and let $r^\sharp=r^\sharp(\Lambda)>0$ and $a^\sharp=a^\sharp(\Lambda)\in\bN$ be as in Lemma~\ref{l.local-stable-manifold}. 

Then there exists $\fL=\fL(\Lambda)\in \bN$ with the following properties. 
For every $\e>0$, there exists $b^\sharp=b^\sharp(\Lambda,\varepsilon)>a^\sharp$ such that  for every $n>b^\sharp$, there exist an $(n,\delta_0)$-separated subset $\cS=\cS(n)\subset\Lambda$ and an integer $\ell\in\{0,\ldots,\fL-1\}$ such that 
\begin{itemize}[ leftmargin=0.5cm ]
\item[(a)] $\card(\cS)\geq e^{n(h_{\rm top}(f,\Lambda)-\e)},$
\item[(b)] 	for every $N\in\bN$, there exists an injective map
\[
	\Psi=\Psi^{n,N}\colon\cS^{N}\to\Lambda
\] 
having the following properties: given two points 
\[
	\sX=(x_1,\ldots,x_N),\quad
	\sX'=(x_1',\ldots,x_N')\in \cS^N,
\] 
let $i_0$ be the smallest integer such that $x_{i_0}\neq x_{i_0}'$, then  
\begin{itemize}
	\item[(b1)] for every $i\leq i_0-1$ and $(i-1)(n+\ell)\leq j<(i-1)(n+\ell)+n$, it holds 
	\[
d\big(f^j\circ\Psi(\sX),
	f^j\circ\Psi(\sX')\big)
	<\delta_0/5;
	\]
	\item[(b2)]  there exist $(i_0-1)(n+\ell)+a^\sharp<j_0<(i_0-1)(n+\ell)+n$ such that 
	\[
		d\big(f^{j_0}\circ\Psi(\sX),
			f^{j_0}\circ\Psi(\sX')\big)>4 \delta_0/5.
	\]
\end{itemize}
\item[(c)] for every pair of different points $x,y\in\Image(\Psi)$ and $r\leq r^\sharp$, one has
\[
		\cW^\s (x,r)\cap\cW^\s (y,r)=\emptyset.
\] 	
\end{itemize}
\end{proposition}

\proof 
We collect two classical facts without proof.

\begin{claim}[Spectral decomposition]\label{clfactergdec}
There exist $p^\sharp\in\bN$ and a decomposition 
\[	\Lambda
	=\widehat\Lambda\cup f(\widehat\Lambda)\cup\ldots\cup f^{p^\sharp-1}(\widehat\Lambda)
\]
such that $\widehat\Lambda, f(\widehat\Lambda),\ldots, f^{p^\sharp-1}(\widehat\Lambda)$ are pairwise disjoint and $f^{p^\sharp}$-invariant compact sets, and for each $0\leq i<p^\sharp$, the map $f^{p^\sharp}\colon f^i(\widehat\Lambda)\to f^i(\widehat\Lambda)$ is topologically mixing.
\end{claim}

The following result, that we state without proof (see, for example, \cite{KatHas:95}), is a consequence of the fact that $f^{p^\sharp}\colon f^i(\widehat\Lambda)\to f^i(\widehat\Lambda)$ has the specification property for each $i$.

\begin{claim}[Specification property]\label{refclaimspec}
	There exists $\ell_1\in\bN$ so that every finite family of orbit segments which start and end in $\widehat\Lambda$ are $\delta_0/10$-shadowed by a true orbit segment in $\Lambda$ with spacing time $\ell_1$.  Analogous result holds  for orbit segments which start and end in $f^i(\widehat\Lambda)$, $i\in\{1,\ldots,p^\sharp-1\}$.
\end{claim}

Let $a^\sharp=a^\sharp(\Lambda)$ as in Lemma~\ref{l.local-stable-manifold}. Let 
\begin{equation}\label{deflsharp}
\fL\eqdef \ell_1+p^\sharp.
\end{equation} 
The following claim provides the set $\cS(n)$ and proves property (a) in the proposition. 

\begin{claim}[The set $\cS(n)$]\label{claimcS}
	For every $\e>0$, there exists $b^\sharp>a^\sharp$ such that for every $n>b^\sharp$, there exists a $(n,\delta_0)$-separated set $\cS(n)$ which satisfies 
\begin{itemize}[ leftmargin=0.5cm ]
	\item  $\card(\cS(n))\geq e^{n(h_{\rm top}(f,\Lambda)-\e)}$;
	\item $\cS(n)\subset f^j(\widehat\Lambda)$ for some $0\leq j<p^\sharp$;
	\item  $\cS(n)$ is contained in  $B_{a^\sharp}(y_0,\delta_0/4)$ for some point $y_0\in\Lambda.$
\end{itemize}
\end{claim}

\begin{proof}
Fix some $(a^\sharp,\delta_0/4)$-separated set of maximal cardinality $\widetilde\cS\subset \Lambda$. The fact that $\widetilde\cS\subset\Lambda$ is an $(a^\sharp,\delta_0/4)$-separated set of maximal cardinality, together with Claim \ref{clfactergdec}, implies that
\begin{equation}\label{maxcardcons}
	\Lambda
	= \bigcup_{i=0}^{p^\sharp-1}f^i(\widehat\Lambda)
	\subset\bigcup_{y\in \widetilde\cS}B_{a^\sharp}(y,\delta_0/4).
\end{equation}
By the compactness of $M$, the set $\widetilde\cS$ is finite. 

As $f$ is $\delta_0$-entropy expansive, by \cite[Corollary 2.5]{Bow:72}, for every $\e>0,$ there exists $n_1\in\bN$ such that for every $n\geq n_1$, every  $(n,\delta_0/4)$-separated set $\widehat\cS(n)\subset\Lambda$ with maximal cardinality satisfies
\begin{equation}\label{maxcarsep}
	 e^{n(h_{\rm top}(f,\Lambda)-\e/2)}
	 \le \card\big( \widehat\cS(n)\big)
	 \le e^{n(h_{\rm top}(f,\Lambda)+\e/2)}.
\end{equation}	
Hence, it follows straightforwardly from \eqref{maxcarsep}, that there exists $n_2\in\bN$ such that for every $n\geq \max\{n_1,n_2\}$,  we have
\[
	e^{n(h_{\rm top}(f,\Lambda)-\e)}
	\le \frac{1}{p^\sharp}\frac{\card\big( \widehat\cS(n)\big)}{\card \big(\widetilde\cS\big)}
	\le e^{n(h_{\rm top}(f,\Lambda)+\e)}.	
\] 
Let 
\[
	b^\sharp\eqdef\max\{a^\sharp,n_1, n_2\}.
\]	 
Consider now $n\ge b^\sharp$ and any $(n,\delta_0)$-separated set with maximal cardinality $\widehat\cS(n)\subset\Lambda$. Hence, by the pigeonhole principle applied to the cover in \eqref{maxcardcons}, there exist $0\leq j<p^\sharp$ and $y_0\in\widetilde\cS$ such that
\[\begin{split}
	\card\big(\widehat\cS(n)\cap f^j(\widehat\Lambda)\cap B_{a^\sharp}(y_0,\delta_0/4)\big)
	&\ge \frac{1}{p^\sharp}\frac{1}{\card\big(\widetilde\cS\big)}\card(\widehat\cS(n))\\
	&\ge e^{n(h_{\rm top}(f,\Lambda)-\varepsilon)}.
\end{split}\]
Taking $\cS(n)\eqdef \widehat\cS(n)\cap f^j(\widehat\Lambda)\cap B_{a^\sharp}(y_0,\delta_0/4)$, the claim follows.
\end{proof}

Without loss of generality, we can assume that in Claim \ref{claimcS} it holds $j=0$ and hence $\cS(n)\subset\widehat\Lambda$. The following is a consequence of Claim \ref{clfactergdec}.

\begin{claim}\label{clgetback}
	For every $n>b^\sharp$, there exists $j_1=j_1(n)$ satisfying $0\leq j_1\leq p^\sharp-1$ and 
\[\cS(n)\cup f^{n+j_1}(\cS(n))\subset \widehat\Lambda.\] 
\end{claim}
 
Let us now prove property (b) in the proposition.

\begin{claim}
	For every $n>b^\sharp$ and $N\in\bN$, there exists an injective map $\Psi\colon (\cS(n))^{N}\to\Lambda$ satisfying properties (b1) and (b2).
\end{claim}

\begin{proof}	
We start with the definition of $\Psi$. Fix $N\in\bN$ and $n>b^\sharp$ and let 
\[
	\ell=\ell(n)\eqdef\ell_1+j_1(n)
	<\fL=\ell_1+p^\sharp.
\]	
Consider a tuple $(x_1,\ldots,x_{N})\in \cS(n)^{N}$. By construction and by Claim \ref{clgetback}, for every $1\le i\le N$, the point $x_i$ and the ``end-point'' $f^{n+j_1(n)}(x_i)$ are in $\widehat\Lambda$. Hence, by Claim \ref{refclaimspec} applied to $\widehat\Lambda$ and $\delta_0/10$ and the orbit segments
\[
	x_1,\ldots,f^{n+j_1(n)}(x_1)\quad
	x_2,\ldots,f^{n+j_2(n)}(x_2)\quad\ldots\quad 
	x_N,\ldots,f^{n+j_1(n)}(x_N),
\] 
there exists $x\in\Lambda$ such that for every $1\leq i\leq N$ and $0\leq j<n$, one has in particular
\begin{equation}\label{consspec}
	d\big(f^{j+(i-1)(n+\ell)}(x),f^j(x_i)\big)<\delta_0/10.
\end{equation}	
We define a map  $\Psi\colon(\cS(n))^N\to\Lambda$ by letting 
\[
	\Psi(x_1,\ldots,x_{N})\eqdef x.
\]

Let us now verify properties (b1) and (b2). Given two $N$-tuples
\[
	\sX=(x_1,\ldots,x_N), \quad
	\sX'=(x_1',\ldots,x_N')\in (\cS(n))^N,
\] 
let $i_0$ be the integer such that $x_{i_0}\neq x_{i_0}'$ and $x_{i}= x_{i}'$ for every $i<i_0$. By the definition of $\Psi$ and  \eqref{consspec}, for every $i\leq i_0-1$ and $0\leq j<n$,  one has that 
\[\begin{split}
	&d\big(f^{j+(i-1)(n+\ell)}\circ\Psi(\sX),f^{j+(i-1)(n+\ell)}\circ\Psi(\sX')\big)
	\\
	&\leq d\big(f^{j+(i-1)(n+\ell)}\circ\Psi(\sX),f^j(x_i)\big)
	+	d\big(f^{j}(x_i),f^{j+(i-1)(n+\ell)}\circ\Psi(\sX')\big)	\\
	&= d\big(f^{j+(i-1)(n+\ell)}\circ\Psi(\sX),f^{j}(x_i)\big)
	+	d\big(f^{j}( x_i'),f^{j+(i-1)(n+\ell)}\circ\Psi(\sX')\big)\\
	&<\delta_0/10+\delta_0/10=\delta_0/5.
\end{split}\]
This is property (b1). 

To prove (b2), by Claim \ref{claimcS}, $\cS(n)$ is an $(n,\delta_0)$-separated set. Hence, as $x_{i_0}$ and $ x_{i_0}'$ are distinct points in $\cS(n)$, there exists $0\le k_0<n$ such that
\begin{equation}\label{eq:separated-at-k0}
	d\big(f^{k_0}(x_{i_0}),f^{k_0}(x_{i_0}')\big)>\delta_0.
\end{equation}
Moreover, as Claim \ref{claimcS} also asserts that $\cS(n)\subset B_{a^\sharp}(y_0,\delta_0/4)$ for some point $y_0\in\Lambda$, for every $k=0,\ldots,a^\sharp-1$, we get
\[\begin{split}
	d\big(f^k(x_{i_0}),f^k( x_{i_0}')\big)
	&\le d(f^k(x_{i_0}),f^k(y_0))+d(f^k(y_0),f^k(x_{i_0}'))
	\le \delta_0/4+\delta_0/4\\
	&= \delta_0/2
	<\delta_0	.
\end{split}\]
Hence, by \eqref{eq:separated-at-k0}, we get $n>k_0\ge a^\sharp$. Then, we get	
\[\begin{split}
	&d\big(f^{(i_0-1)(n+\ell)+k_0}\circ\Psi(\sX),f^{(i_0-1)(n+\ell)+k_0}\circ\Psi(\sX')\big)\\
	&\geq d\big(f^{k_0}(x_{i_0}),f^{k_0}({x}_{i_0}')\big)
		-d\big(f^{(i_0-1)(n+\ell)+k_0}\circ\Psi(\sX),f^{k_0}(x_{i_0})\big)\\
	&\phantom{\ge}
		-d\big(f^{(i_0-1)(n+\ell)+k_0}\circ\Psi(\sX'),f^{k_0}( x_{i_0}')\big)\\
	\text{\small(by \eqref{eq:separated-at-k0} and \eqref{consspec})}\quad	
	&>\delta_0-\delta_0/10-\delta_0/10
	=4 \delta_0/5.
\end{split}\]
This proves property (b2). 
\end{proof}

It remains to prove property (c). Let $r\le r^\sharp$ and recall that $r^\sharp\in(0,\delta_0/4)$. For every two different points $x,y\in\im(\Psi)$, by property (b2), there exists $j_0> a^\sharp$ such that $d(f^{j_0}(x),f^{j_0}(y))>4\delta_0/5$. By contradiction, suppose that there is $z\in\cW^\s(x,r)\cap\cW^\s(y,r)$. Then, by \eqref{agree} for all $j\ge a^\sharp$ one has 
\[d(f^j(x),f^j(y))\leq d(f^j(x),f^j(z))+d(f^j(z),f^j(y))\le2r<\delta_0/2,\] 
which contradicts with $d(f^{j_0}(x),f^{j_0}(y))>4\delta_0/5$. 
This finishes the proof of Proposition~\ref{p.separation-in-symbol}.
\endproof

\section{A fractal set with large entropy and forward center exponent zero}\label{secfractalset}

The heart of this section is the construction of a limsup set with \emph{forward} zero center Lyapunov exponent  and large entropy, whose existence and properties are stated in Theorem \ref{THETHEObisbis}. This theorem contains combinatorial and fractal properties (item \eqref{THETHEObisbis-3}), proved in Section \ref{secproofTHEOBIS}, and ergodic ingredients (items \eqref{THETHEObisbis-1} and \eqref{THETHEObisbis-2}), shown in Sections \ref{proofitem1} and \ref{secproofTHEO}, respectively.  Theorem \ref{THETHEObisbis} is our main step towards the proof of Theorem \ref{theorem2}. 

We first introduce some terminology. In analogy to the classical definition of Bowen balls (recall  \eqref{defBowBal}), given $n\in\bN$, $\varepsilon>0$, and a center curve $\gamma$, let
\[
	B^\cu_n(\gamma,\varepsilon)
	\eqdef \bigcap_{i=0}^{n-1}f^{-i}\Big(\overline{\cW^\uu(f^i(\gamma),\varepsilon)}\Big).	
\]
Given $\delta>0$, sequences of natural numbers $(t_k)_k$ and $(T_k)_k$, and a sequence  $(\mcE_k)_k$  of families of center curves, let 
\begin{equation}\label{defLL}
	L=L\big(\delta,(t_k)_k,(T_k)_k,(\mcE_k)_k\big)
	\eqdef\bigcap_{j\in\bN} L_j,\quad\text{ where }\quad
	L_j
	\eqdef \bigcup_{\gamma_j\in \mcE_j}B^{\cu}_{t_{j-1}+T_j}\big(\gamma_j,\delta/2^{j-1}\big). 
\end{equation}
We will see in the proof of the following theorem that $(L_k)_k$ is a decreasing nested sequence of compact sets and therefore $L$ is compact nonempty.

\begin{theorem}\label{THETHEObisbis}
	Let $f\in\BM^1(M)$. For every $\delta>0$ sufficiently small, there are increasing sequences $(t_k)_k$ and $(T_k)_k$, and a sequence  $(\mcE_k)_k$ of families of center curves such that the set $L=L\big(\delta,(t_k)_k,(T_k)_k,(\mcE_k)_k\big)$ in \eqref{defLL}
has the following properties:
\begin{enumerate}[ leftmargin=0.7cm ]
\item\label{THETHEObisbis-3} $L$ is a compact nonempty set that consists of pairwise disjoint non-degenerate compact center curves $\gamma_\infty$; each $\gamma_\infty$ is $(\c,\cW^\uu(\xi,2\delta))$-complete for some $\xi\in \mcE_1$,
\item\label{THETHEObisbis-1} $h_{\rm top}(f,L)= \h(f)$, where $\h(f)$ is given in \eqref{outofmind},
\item\label{THETHEObisbis-2} for every $x\in L$, $\lim_{n\to+\infty}\frac{1}{n}\log\|D^{\c}f^n(x)\|=0$ and this convergence is uniform on $L$.
\end{enumerate}	
\end{theorem}

Note that Theorem \ref{THETHEObisbis} does not yet provide any estimate of the entropy of $\cL(0)$, as $L$ only contains points whose \emph{forward} center exponent is zero. The estimate of the entropy of $\cL(0)$ will be completed in Section \ref{s.final-fractal}. 

The following result is an immediate consequence of the constructions in the proof of Theorem \ref{THETHEObisbis}, see Section \ref{secSchol} for brief arguments. Compare \cite{Tah:21} for further references on the concept of unstable entropy and related results.

\begin{scholium}\label{schol2}
	Assume the hypotheses of Theorem \ref{THETHEObisbis}. Given $\delta>0$ sufficiently small, let $L=L\big(\delta,(t_k)_k,(T_k)_k,(\mcE_k)_k\big)$ as provided by Theorem \ref{THETHEObisbis}. Then for every $\xi\in\mcE_1$ and $x\in\xi$,
\[
	h_{\rm top}\big(f,\cW^\uu(x,2\delta)\cap L\big)=\h(f).
\]	
\end{scholium}

This section is devoted to proving item \eqref{THETHEObisbis-3} in Theorem \ref{THETHEObisbis}. For that, in Section \ref{sectkTk}, we choose the sequences $(t_k)_k$ and $(T_k)_k$ and the sequence of center curves $(\mcE_k)_k$ having controlled behavior, see Proposition \ref{prolempropos}.  Its proof is given in Section \ref{secproof}.  In Section \ref{ssecprocentercurves} we investigate separation properties of the curves in $(\mcE_k)_k$. In Section \ref{secproofTHEOBIS}, we complete the proof of item \eqref{THETHEObisbis-3} in Theorem \ref{THETHEObisbis}.   Throughout this section, we let $f\in \BM^1(M)$ with a stable blender-horseshoe $\blender^-$.

\subsection{Families of center curves with prescribed behavior}\label{sectkTk}

In this section, we construct families of curves $(\mcE_k)_k$.  The concatenation method in Proposition~\ref{proLem:key} plays a key role. 

Let $\delta_0>0$ be as in Remark \ref{remexpansive}. Recall the choice of the quantifier $r^\u_2$ in Remark \ref{pivotal}.

\begin{proposition}[Families of center curves]\label{prolempropos}
Let $f\in\BM^1(M)$. For every 
\begin{equation}\label{fixdelta}
	\delta\in\Big(0,\min\big\{\frac{1}{20}\delta_0,\frac12r^\u_2\big\}\Big),
\end{equation}	
there are sequences $(t_k)_k$ and $(T_k)_k$ of natural numbers, sequences of families of center curves $(\DD_k)_k$ and $(\mcE_k)_k$, and bijections 
\[
	\phi_k\colon \mcE_k\times \DD_{k+1}\to \mcE_{k+1}
\]	 
such that for every $\gamma_k\in\mcE_k$, $\xi_{k+1}\in \DD_{k+1}$, and $\gamma_{k+1}\eqdef\phi_k(\gamma_k,\xi_{k+1})$,
\begin{enumerate}[ leftmargin=0.9cm ]
\item[(1)] $\gamma_{k+1}$ is $(\c,\cW^\uu(\gamma_k,\delta/2^k))$-complete,\\[-0.3cm]
\item[(2)] $\cW^\uu(f^i(\gamma_{k+1}),\delta/2^k)
	\subset \cW^\uu(f^i(\gamma_k),\delta/2^{k-1})$ for  $ i=0,\ldots, t_{k-1}+T_k$,\\[-0.3cm]
\item[(3)] $f^{t_k}(\gamma_{k+1})\subset \cW^\ss(\xi_{k+1},\delta/2^{k})$,\\[-0.3cm]
\item[(4)] $f^{t_{k-1}}(\gamma_{k+1})$ is $\big(\c,\cW^\uu(f^{t_{k-1}}(\gamma_k),\delta/2^k)\big)$-complete.
\end{enumerate}
Moreover, 
\begin{enumerate}[ leftmargin=0.9cm ]
\item[(5)] $\gamma_k$ is $(\c,\cW^\uu(\xi,\delta))$-complete for some $\xi\in \DD_1$.
\end{enumerate}
\end{proposition}

Before proving this proposition, 
let us first introduce some terminology. 

\begin{definition}
	Given $\gamma_{k+1}\in\phi_k(\{\gamma_k\}\times \DD_{k+1})$, we call $\gamma_k$ a \emph{predecessor} of $\gamma_{k+1}$ and $\gamma_{k+1}$  a \emph{successor} of $\gamma_k$. Given $\ell>k$, we call $\gamma_k\in \mcE_k$ a \emph{$(\ell-k)$-predecessor} of $\gamma_\ell\in \mcE_\ell$ if there exist curves $\{\gamma_{i}\in \mcE_i\colon k+1\leq i\leq \ell-1\}$ such that $\gamma_{i}$ is a predecessor of $\gamma_{i+1}$ for every $k\leq i\leq \ell-1$. In this case, $\gamma_\ell$ is a \emph{$(\ell-k)$-successor} of $\gamma_k$.  
\end{definition}

\subsection{Proof of Proposition \ref{prolempropos}}\label{secproof}

In Section \ref{ssecquanti}, we fix some quantifiers. In Section \ref{secdefDk}, we define the families $\DD_k$. The proof is completed in Section \ref{secchuvv}.

\subsubsection{Choice of quantifiers}\label{ssecquanti}

Let $\delta>0$ be as in \eqref{fixdelta}. Fix a sequence $(\varepsilon_k)_k$ monotonically decreasing to zero. 

Recall that, by Theorem~\ref{theorem1}, one has 
\[\lim_{\beta\nearrow0}\sup\big\{h_\mu(f)\colon\mu\in\cM_{\rm erg}(f),\chi^\c (\mu)=\beta\big\}=\h(f).\] 
Thus, there exist ergodic measures $\nu\in\cM_{\rm erg}(f)$ such that $\chi^\c (\nu)$ is negative and arbitrarily close to zero and $h_{\nu}(f)$ is arbitrarily close to $\h(f)$. Hence, together with Lemma \ref{l.horseshoe-with-large-entropy}, this implies that there exist a sequence $(\Lambda_k)_k$ of basic sets of contracting type and sequences of real numbers $\chi_k<0$, $r_k\in(0, r^\sharp(\Lambda_k))$, $\varepsilon_k>0$, and $K_k\ge1$ such that 
\begin{enumerate}[ leftmargin=0.7cm ]
\item[(i)] $h_k\eqdef h_{\rm top}(f,\Lambda_k)\ge\h(f)-\varepsilon_k$,
\item[(ii)] $\chi_k$, $\varepsilon_k/\chi_k$, and $r_k$ are strictly monotonic and tend to zero as $k\to\infty$,
\item[(iii)] every center curve centered at  $x\in\Lambda_k$ of inner radius  less than $r_k$ is contained in the stable manifold of $x$, 
\item[(iv)] for every center curve $\gamma$ centered at $x\in\Lambda_k$ of inner radius less than $ r_k$, for every $y\in \cW^{\ss}\big(\gamma,\delta\big)$ and  $n\in\mathbb{N}$,  one has 
\[
	K_k^{-1}\cdot e^{-n(\chi_k-\varepsilon_k)}\leq\|D^\c f^n(y)\|\leq K_k\cdot e^{n(\chi_k+\varepsilon_k)}.
\]
\end{enumerate}

Let 
\[\begin{split}	
	&a_k\eqdef a^\sharp(\Lambda_k)\quad\text{ as in Lemma \ref{l.local-stable-manifold}},\\
	&\fL_k \eqdef\fL(\Lambda_k)\text{ and }b^\sharp_k=b^\sharp(\Lambda_k,\varepsilon_k)
	\quad\text{ as in Proposition \ref{p.separation-in-symbol},}\\
	&\ell^\flat_k\eqdef \ell^\flat(\delta,\varepsilon_k)\quad\text{ the distortion-control time in Lemma \ref{lemsomecount}}.
\end{split}\]	 
Recall the constants $N(\cdot)$ and $R(\cdot)$ in Lemma~\ref{lem:insideblenders}, and $r^\sharp(\cdot)$ in Lemma \ref{l.local-stable-manifold}. 
Up to shrinking $r_k$, we can assume that for every $k$
\begin{equation}\label{eq5V}
	r_k\le r^\sharp(\Lambda_k),\quad
	2r_k<R(N(\delta/2^{k-1})),
	\quad \text{ and }\quad
	\sum_k r_k<\delta_0/20.
\end{equation}	
Recall the constant $\eta^\s(\cdot)$ in Remark \ref{r.etas}. For every $k\in\bN$, consider the corresponding numbers as in \eqref{eq:mn},
\begin{equation}\label{EQrem:the-time-iterate-first-curve}\begin{split}
	&m_k
	\eqdef m(2r_{k+1},\delta/2^{k})
	= 2N(\delta/2^{k})+M_\b\big(\eta^\s(2r_{k+1},\delta/2^{k})\big),\\
	&T^\sharp_k
	\eqdef T^\sharp(2r_k,K_k,\varepsilon_k,\delta/2^k)
	= \left\lceil\frac{1}{\chi_{k}+\varepsilon_{k}}\big(\log R(N(\delta/2^{k}))-\log K_{k}-\log{(2r_{k})}\big)\right\rceil+1.
\end{split}\end{equation}	
Let
\begin{equation}\label{defCmax}
	C_{\rm max}
	\eqdef 2\cdot \sup_{x\in M}\big|\log\|D^{\c}f(x)\|\big|.
\end{equation}

We choose a sequence of natural numbers $(n_k)_k$ such that
\begin{eqnarray}
	& \max\big\{\ell^\flat_k,b^\sharp_k,T^\sharp_k\big\}<n_k, \label{eq:condnk-1}\\
	&\displaystyle\frac{m_k}{n_k}C_{\rm max}<\varepsilon_k\label{onemore},\\
	&\displaystyle \frac{n_k}{n_k+\ell_k^\sharp+m_k} \geq 1-\varepsilon_k,\label{cucucu}	\\
	&\displaystyle\frac{\log K_{k+1}}{n_k}<\varepsilon_k,\label{eqexp1bis}\\
	&\displaystyle\frac{1}{n_k}\big(\log  K_{k+1}+(m_{k+1}+\ell^\flat_{k+1})\cdot C_{\rm max}\big)<\varepsilon_k.\label{eq:condnk-5}
\end{eqnarray} 
Inequalities \eqref{onemore} and \eqref{cucucu} are used to control the entropy in Section \ref{seccentr}. Inequalities \eqref{eqexp1bis} and \eqref{eq:condnk-5} are used to control the exponents in Section \ref{secproofTHEO}.

Let $(N_k)_k$ be a sufficiently fast growing sequence of natural numbers (to be specified in \eqref{extra-cond-tsbis}--\eqref{eqexp1bisbis}).
By \eqref{eq:condnk-1}, we can apply Proposition~\ref{p.separation-in-symbol} to $\Lambda_k$ and $\varepsilon_k$ and obtains an integer $\ell_k\in\{0,\ldots,\ell_k^\sharp\}$, an $(n_k,\delta_0)$-separated set 
\begin{equation}\label{eqcSkcardinality}
	\cS_k\eqdef\cS(n_k)\subset \Lambda_k,
	\quad\text{ with }\quad
	\card(\cS_k)\geq e^{n_k(h_k-\varepsilon_k)},
\end{equation}
 and an injective map 
\begin{equation}\label{defPsik}
	\Psi_k\eqdef\Psi^{n_k,N_k}\colon\cS_k^{N_k}\to\Lambda_k.
\end{equation}

\begin{notation}\label{deftkTk}
For every $k\in\bN$, let
\begin{equation}\label{deftk}
	\begin{split}
	t_0&\eqdef0\\
	t_k&\eqdef t_{k-1}+(T_k+m_k) = \sum_{i=1}^k\big(T_i+m_i\big),
	\quad\text{ where }\quad 
	T_i\eqdef N_i(n_i+\ell_i).
\end{split}\end{equation}
\end{notation}

Choosing $(N_k)_k$ in the way that it is growing fast enough, not affecting the choices in \eqref{eq:condnk-1}--\eqref{eq:condnk-5}, we can assume that 
\begin{eqnarray}
	&\displaystyle\frac{n_{k+1}+\ell_{k+1}^\sharp}{t_k}<\varepsilon_k,\label{extra-cond-tsbis}\\ 
	&\displaystyle\frac{t_k}{t_{k+1}}<\frac{1}{C_{\rm max}}\varepsilon_{k+1}.	\label{extra-cond-ts}\\
	&\displaystyle\frac{\log K_k}{N_k(n_k+\ell_k)}<\varepsilon_k,\label{eqexp1bisbis}	
\end{eqnarray}

For further reference, note that \eqref{deftk} and \eqref{eq:condnk-1} imply
\begin{equation}\label{implythat}
	T_k> T_k^\sharp.
\end{equation}

We provide an informal description of the above-defined numbers.

\begin{remark}
The number $n_k$, chosen large enough, enables to compensate distortion and enables uniform separation. 
The number $\ell_k$ is a ``bridging time'' for $\Lambda_k$ associated to the ``spacing time'' in the specification property of $f$ on $\Lambda_k$ and to ``return to a mixing piece of the spectral decomposition of $\Lambda$'' (see Proposition \ref{p.separation-in-symbol}). 
During the time interval $\{0,\ldots,T_k\}$ the orbit of any point in the image of $\Psi_k$ stays close to the hyperbolic set $\Lambda_k$ ``catching entropy and Lyapunov exponent from the hyperbolic set''. The time $m_k$ is the one to ``get into the blender $\blender^-$, stay there, and get out again''.  
\end{remark}

\subsubsection{Definition of the families $\DD_k$}\label{secdefDk}

\begin{lemma}\label{lemclaDk}
For every $k\in\bN$ and $x\in \im(\Psi_k)\subset\Lambda_k$, there is a center curve $\xi(x,r_k)$ centered at $x$ of inner radius $r_k$ and contained in the stable manifold of $x$ of size $r_k$. Moreover,   the family
\[
	\DD_k
	\eqdef\big\{\xi(x,r_k)\colon x\in \im(\Psi_k)\big\}
\]	
consists of pairwise disjoint curves. 
\end{lemma}

\begin{proof}
Given $x\in \im(\Psi_k)$, the existence of the center curves follows from Remark \ref{remweaint}. By \eqref{eq5V}, we have $r_k<r^\sharp(\Lambda_k)$. Hence, by  item (c) in Proposition~\ref{p.separation-in-symbol}, the center curves in $\DD_k$ are pairwise disjoint. 
\end{proof}

\begin{remark}\label{remcardDk}
	We have $\card \DD_k = \big(\card \cS_k\big)^{N_k}$.
\end{remark}

\subsubsection{End of the proof}\label{secchuvv}

Given $(\DD_k)_k$ as in Lemma \ref{lemclaDk}, we will construct the families of curves $(\mcE_k)_k$ and the maps $(\phi_k)_k$. For that, we proceed inductively using Proposition~\ref{proLem:key} to concatenate the center curves in $(\DD_k)_k$. 
Set 
\[
	\mcE_1\eqdef \DD_1
	=\big \{\xi(x,r_1)\colon x\in\im(\Psi_1)\big\}.
\]	 
By property (iv) in Section \ref{ssecquanti}, every center curve $\gamma_1=\xi_1\in\mcE_1=\DD_1$ satisfies the hypothesis \eqref{eq:crescimento} of Proposition~\ref{proLem:key}. By definition of $\DD_k$ in Lemma~\ref{lemclaDk},  every $\xi_2\in \DD_2$ has length $2r_2<R(N(\delta/2))$, recall \eqref{eq5V}. By  \eqref{implythat}, we have $T_1>T_1^\sharp$. We now apply Proposition~\ref{proLem:key} to the center curves $\gamma_1$ and $\xi_2$ (the latter has length $\length(\zeta_2)=2r_2$) and to the number $\delta/2$ and get the associated center curves $\widehat\zeta_2$ and $\widetilde\zeta_2$ and the number $m_1$ as in \eqref{EQrem:the-time-iterate-first-curve}. Let 
\[
	\phi_1(\gamma_1,\xi_2)\eqdef \gamma_2,
	\quad\text{ where }\quad
	\gamma_2\eqdef\widehat\zeta_2
	\quad\text{ and }\quad
	f^{T_1+m_1}(\gamma_2)=\widetilde\zeta_2.
\]	  
	\begin{figure}[t]
		\begin{overpic}[scale=.45]{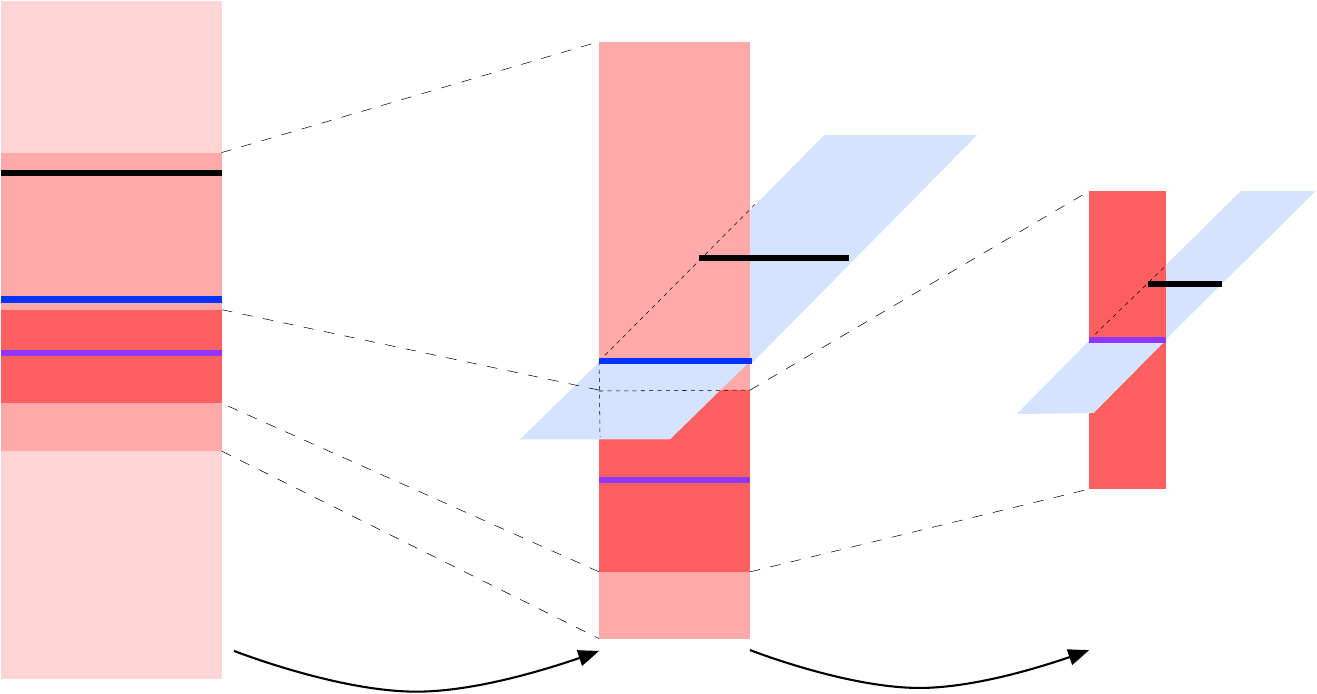}
			\put(64,3){\small$f^{T_2+m_2}$}
			\put(22,3){\small$f^{T_1+m_1}=f^{t_1}$}
			\put(18,25.5){\small$\gamma_3$}
			\put(18,29.5){\small$\gamma_2=\widehat\zeta_2$}
			\put(18,39){\small$\gamma_1=\xi_1$}
			\put(58,25){\small$f^{T_1+m_1}(\gamma_2)=\widetilde\zeta_2$}
			\put(58,16){\small$\widehat\zeta_3$}
			\put(66,32.3){\small$\xi_2$}
			\put(18,50){\small$\cW^\uu(\xi_1,\delta)$}
			\put(62,44){\small$\cW^\ss(\xi_2,\delta/2)$}
			\put(92,39.5){\small$\cW^\ss(\xi_3,\delta/4)$}
			\put(94,30){\small$\xi_3$}
			\put(90,26){\small$\widetilde\zeta_3$}
		\end{overpic}
		\caption{Proof of Proposition \ref{prolempropos}}
		\label{fig.cascade}
	\end{figure}	
Moreover,
\begin{itemize}[ leftmargin=0.7cm ]
\item[(i)] $\gamma_2$ is $\big(\c,\cW^{\uu}(\gamma_1,\delta/2)\big)$-complete,\\[-0.2cm]
\item[(ii)] $f^i(\gamma_2)\subset \cW^{\uu}\big(f^i(\gamma_1),\delta/2\big)$ for $i=0,\dots,T_1$,\\[-0.2cm]
\item[(iii)] $f^{T_1+m_1}(\gamma_2)\subset\cW^\ss(\xi_2,\delta/2)$,\\[-0.2cm]
\item[(iv)] $f^{-(T_1+m_1)}(\cW^\uu(f^{T_1+m_1}(\gamma_2),\delta))\subset\cW^\uu(\gamma_2,\delta/2)$. 
\end{itemize} 
Denote by $\mcE_2$ the set of all curves $\gamma_2$ constructed in this way,
\[
	\mcE_2
	\eqdef \big\{\phi_1(\gamma_1,\xi_2)\colon\gamma_1\in\mcE_1=\DD_1,\xi_2\in \DD_2\big\}.
\]
Recalling \eqref{deftk}, properties (i)--(iv) above prove properties (1), (2), (3), and (4) of the assertion for $k=1$ of Proposition \ref{prolempropos}. Moreover, property (i) is property (5). 

Let us now proceed inductively. Assume that $k\ge2$ and that the families $\mcE_1,\ldots,\mcE_k$ of center curves and the maps $\phi_1,\ldots,\phi_{k-1}$ are already defined with the numbers in \eqref{deftk} and have the properties (1)--(4) of the proposition. Let us define $\mcE_{k+1}$ and $\phi_k$. Given any $\gamma_k\in \mcE_k$, by property (3) there is $\xi_{k}\in \DD_{k}$ such that  
\[
f^{t_{k-1}}(\gamma_k)\subset \cW^{\ss}(\xi_{k}, \delta/2^{k-1}).
\]
Hence, by item (iv) in Section \ref{ssecquanti}, the center curve $f^{t_{k-1}}(\gamma_k)$ satisfies the hypotheses \eqref{eq:crescimento} of  Proposition~\ref{proLem:key}. By \eqref{implythat}, we have $T_{k+1}>T_{k+1}^\sharp$. Given any $\xi_{k+1}\in \DD_{k+1}$ and applying Proposition \ref{proLem:key} to the center curves $\xi_{k+1}$ and $\xi_{k+1}$ and to the number $\delta/2^{k}$, we get associated center curves $\widehat\zeta_{k+1}$ and $\widetilde\zeta_{k+1}$ and the number $m_{k+1}$ as in \eqref{EQrem:the-time-iterate-first-curve} having the posited properties.
 Let
\[
	\phi_k(\gamma_k,\xi_{k+1})
	\eqdef \gamma_{k+1}
	\eqdef f^{-t_{k-1}}(\widehat\zeta_{k+1}).
\]
Finally, let $\mcE_{k+1}$ be the set of all center curves constructed in this way,
\[
\mcE_{k+1}
\eqdef \big\{\phi_k(\gamma_k,\xi_{k+1})\colon \gamma_k\in\mcE_k,\xi_{k+1}\in \DD_{k+1}\big\}.
\]	 
As for the case $k=1$, Proposition~\ref{proLem:key} implies that the properties (1)--(4) of the assertion hold for $k+1$. This proves the first part of the proposition.

It remains to show property (5). Recall that $\DD_1=\mcE_1$.  By construction, there exist center curves $\gamma_{i}\in\mcE_i$ for $1\leq i\leq k-1$  such that $\gamma_{i+1}$ is $(\c,\cW^\uu(\gamma_{i},\delta/2^{i}))$-complete for $1\leq i\leq k-1.$ Therefore, for any point $x_k\in\gamma_k$, there exist $x_i\in\gamma_i$, for $1\leq i\leq k-1 $, such that $x_{i+1}\in\cW^\uu(x_i,\delta/2^i)$ for $1\leq i\leq k-1$. Furthermore, if $x_{k}$ is an endpoint of $\gamma_k$, then $x_i$ is an endpoint of $\gamma_i$ for $1\leq i\leq k-1$, due to the definition of being $(\c,\cW^\uu(\gamma_{i},\delta/2^{i}))$-complete. This implies that $x_k\in\cW^\uu(x_1,\delta)$. By the arbitrariness of $x_k$, one has $\gamma_k\subset \cW^\uu(\gamma_1,\delta)$. Furthermore, if $x_k$ is an endpoint $\gamma_k$, then $x_1$ is an endpoint of $\gamma_1$. Thus, $\gamma_k$ is $(\c,\cW^\uu(\gamma_1,\delta))$-complete. This proves the proposition.
\qed

\begin{remark}\label{rembox}
	With $\Psi_k$ in \eqref{defPsik} and $\phi_k$ in Proposition \ref{prolempropos}, let
	\[
	\Phi_k\colon \mcE_k\times \cS_{k+1}^{N_{k+1}}\to \mcE_{k+1},\quad
	\Phi_k(\gamma_k,\sX)
	\eqdef\phi_k\big(\gamma_k,\xi(\Psi_{k+1}(\sX),r_{k+1})\big).
	\]
	By definition, $\Phi_k$ and $\phi_k$ are both onto.
\end{remark}

\medskip

The following result follows from the proof of Proposition \ref{prolempropos} replacing each $\DD_k$ by just one curve $\{\xi_k\}$. It will be used in Section \ref{s.final-fractal}.

\begin{scholium}\label{schol}
	Consider the sequence $(\Lambda_k)_k$ in Section~\ref{ssecquanti}. Let $(\xi_k)_k$ be sequence of center curves  $\xi_k=\xi(x_k,r_k)$ centered at $x_k\in\Lambda_k$ and of inner radius  $r_k$. There is a sequence $(\gamma_k)_k$ satisfying the assertions of Proposition \ref{prolempropos}.
\end{scholium}

\subsection{Separated properties of the center curves $(\mcE_k)_k$.}\label{ssecprocentercurves}

We start by studying some properties of the families $\mcE_{k}$ and the maps $\Phi_k$. In particular, we study the ``separation'' of the center curves. In analogy to the usual definition for points (recall  \eqref{defseparated}), two curves $\gamma,\gamma'$ are \emph{$(n,\varepsilon)$-separated} if for every $x\in\gamma$ and $x'\in\gamma'$, one has $d_n(x,x')>\e$.

The main aim of this subsection is to prove the following result.

\begin{proposition}\label{prooooop}
	 For every $k$,  one has 
\begin{itemize}[ leftmargin=0.7cm ]
\item $\mcE_k$ is a family of $\big(t_{k-1}+T_k,\delta_0/2\big)$-separated center curves,
\item the maps $\Phi_k$ and $\phi_k$ (see Remark \ref{rembox} and Proposition \ref{prolempropos}) are bijective.
\end{itemize}
\end{proposition} 

We postpone the proof of the above result to the end of this section.

The following result is now immediate.

\begin{corollary}\label{garb}
For every $\ell>k$,
\[
	\card \mcE_{\ell}
	= \card \mcE_k \cdot \prod_{j=k+1}^{\ell}\card \DD_{j}.
	\quad\text{ In particular, }\quad
		\card \mcE_{\ell}
	= \prod_{j=1}^\ell\DD_j.
\]
\end{corollary}

The following refines Corollary \ref{garb}, we prove it at the end of this subsection.

\begin{corollary}\label{corlem:unique-predecessor}
	For every $\ell>k$, each  $\gamma_\ell\in\mcE_\ell$ has a unique $(\ell-k)$-predecessor in $\gamma_k\in\mcE_k$, and $\gamma_\ell$ is $(\c,\cW^\uu(\gamma_k,\delta/2^{k-1}))$-complete.
\end{corollary}

Recall our notation \eqref{defdist} for distance between curves.

\begin{lemma}[Separation of curves in $\DD_k$]\label{lemDsep}
	For every $k\in\bN$ and distinct curves $\xi,\xi'\in \DD_k$ there are numbers $N$ and $j$, 
\begin{equation}\label{asin}
	1\leq N\leq N_k
	\quad\text{ and }\quad
	(N-1)(n_k+\ell_k)+a^\sharp_k<j<(N-1)(n_k+\ell_k)+n_k,
\end{equation}
such that for every $x\in\xi$ and $x'\in\xi'$, one has 
\[
	d(f^j(x),f^j(x'))>4\delta_0/5-2r_k.
\]	 
Moreover, the curves $\xi,\xi'$ are $(T_k,4\delta_0/5-2r_k)$-separated. 
\end{lemma}

\begin{proof}
Consider the curves $\xi=\xi(x_k,r_k)$ and $\xi'=\xi(x_k',r_k)$ in $\DD_k$, where $x_k,x_k'\in\im(\Psi_k)$. By item (b2) in Proposition~\ref{p.separation-in-symbol}, there exist $N$ and $j$ as in \eqref{asin}  such that 
\[
	d(f^j(x_k),f^j(x_k'))>4\delta_0/5.
\]
Let $x\in\xi$ and $x'\in\xi'$. Note that \eqref{asin} implies $j>a^\sharp_k$ and hence, as we assume $r_k\in(0,r^\sharp_k)$, by Lemma~\ref{l.local-stable-manifold} item (2), one has 
\[
	d(f^j(x),f^jx_k))<r_k \quad\text{~and~}\quad
	d(f^j(x'),f^j(x_k'))<r_k.
\]
Therefore, one has 
\begin{equation}\label{eq111}\begin{split}
	d(f^j(x),f^j(x'))
	&\geq d(f^j(x_k),f^j(x_k'))-d(f^j(x),f^j(x_k))-d(f^j(x'),f^j(x_k'))\\
	&>4\delta_0/5-2r_k, 
\end{split}\end{equation}
that is, $\xi,\xi'$ are $(j,4\delta_0/5-2r_k)$-separated.

Finally, recall that $T_k=N_k(n_k+\ell_k)$ and hence $N\le N_k$ implies $T_k\ge j$. This together with   \eqref{eq111} implies that $\xi,\xi'$ are $(T_k,4\delta_0/5-2r_k)$-separated. This proves the lemma.
\end{proof}

Now, we turn to the separated properties of curves in $\mcE_k.$
\begin{lemma}[Separation of curves in $\mcE_k$]\label{induzsepar}
For every $k\ge1$, the curves in $\mcE_k$ are pairwise
\[\big(t_{k-1}+T_k,4\delta_0/5-\sum_{j=0}^{k-2}\delta/2^j-2\sum_{j=1}^{k}r_{j}\big)\text{-separated},\]
where for $k=1$ the first sum is disregarded. In particular, the curves in  $\mcE_k$ are pairwise
$(t_{k-1}+T_k,3\delta_0/5)$-separated.
\end{lemma}

\begin{proof}

By definition, $\mcE_1=\DD_1$. Then the case for $k=1$ follows directly from Lemma~\ref{lemDsep}.

Assume that the assertion holds for $\mcE_1,\ldots,\mcE_k$.  In order to check the assertion for $k+1$, consider curves $\gamma_{k+1}= \phi_k(\gamma_k,\xi_{k+1})$ and $\gamma_{k+1}'=\phi_k(\gamma_k',\xi_{k+1}')$ in  $\mcE_{k+1}$, where
\[
	(\gamma_k,\xi_{k+1}) ,(\gamma_k',\xi_{k+1}' )\in \mcE_k\times \DD_{k+1}, 
\]	 
There the following two cases to be considered.\medskip

\noindent\textbf{Case $\xi_{k+1}\neq\xi_{k+1}'$.}
By Lemma \ref{lemDsep}, there exist $N,j$ with
\[
	1\leq N\leq N_{k+1}
	\quad\text{ and }\quad
	(N-1)(n_{k+1}+\ell_{k+1})+a^\sharp_{k+1} <j<(N-1)(n_{k+1}+\ell_{k+1})+n_{k+1}
\]	 
so that for every $y\in\xi_{k+1}$ and $y'\in\xi_{k+1}'$ and
\begin{equation}\label{iolat}
	d(f^j(y),f^j(y'))
	>4\delta_0/5-2r_{k+1}.
\end{equation}
Item (3) in Proposition~\ref{prolempropos} gives that 
\[
	f^{t_{k}}(\gamma_{k+1})\subset \cW^\ss(\xi_{k+1},\delta/2^k) 
	\quad\textrm{~and~}\quad 
	f^{t_{k}}(\gamma_{k+1}')\subset \cW^\ss(\xi_{k+1}',\delta/2^k).
\]
For any $z\in\gamma_{k+1}$ and $z'\in\gamma_{k+1}'$, take $y\in\xi_{k+1}$ and $y'\in\xi_{k+1}'$ such that 
$f^{t_{k}}(z)\subset \cW^\ss(y,\delta/2^k)$  and  $f^{t_{k}}(z')\subset \cW^\ss(y',\delta/2^k).$
By the uniform contraction along the strong stable foliation $\cW^\ss$, one has 
\begin{align*}
	d(f^{t_k+j}(z),f^{t_k+j}(z'))
	&\geq d(f^j(y),f^j(y'))-d(f^{t_k+j}(z),f^{j}(y))-d(f^{t_k+j}(z'),f^{j}(y'))\\
	\text{\small{(by \eqref{iolat})}}\quad
	&\geq 4\delta_0/5-2r_{k+1}-2\delta/2^{k}
	= 4\delta_0/5-2r_{k+1}-\delta/2^{k-1}.
\end{align*}
Note that
\[
	4\delta_0/5-2r_{k+1}-\delta/2^{k-1}
	>4\delta_0/5-\sum_{j=0}^{k-1}\delta/2^j-2\sum_{j=1}^{k+1}r_{j}.
\]
This proves the lemma in this case. 
\medskip

\noindent\textbf{Case $\gamma_k\neq\gamma_k'$.} 
By the induction hypothesis, the curves $\gamma_k$ and $\gamma_k'$ belong to $\mcE_k$ and are $\big(t_{k-1}+T_k,4\delta_0/5-\sum_{j=0}^{k-2}\delta/2^j-2\sum_{j=1}^{k}r_{j}\big)$-separated. 
By item (2) in Proposition \ref{prolempropos}, for any $0\leq i\leq t_{k-1}+T_{k}-1$, one has 
\[
	f^i(\gamma_{k+1})\subset \cW^\uu(f^i(\gamma_k),\delta/2^k)
	\quad\text{ and }\quad
	f^i(\gamma_{k+1}')\subset \cW^\uu(f^i(\gamma_k'),\delta/2^k) .
\]
Hence, $\gamma_{k+1}$ and $\gamma_{k+1}'$ are $\big(t_{k-1}+T_{k},4\delta_0/5-\sum_{j=0}^{k-1}\delta/2^j-2\sum_{j=1}^{k}r_{j}\big)$-separated. This proves the lemma in this case. 

Note that
\[\begin{split}
	4\delta_0/5-\sum_{j=0}^{k-1}\delta/2^j-2\sum_{j=1}^{k}r_{j}
	&\geq 4\delta_0/5-\sum_{j=0}^{\infty}\delta/2^j-2\sum_{j=1}^{\infty}r_{j}\\
	\text{\small{(by \eqref{eq5V} and \eqref{fixdelta})}}\quad
	&>4\delta_0/5-2\delta-\delta_0/10
	>3\delta_0/5.
\end{split}\]	
This finishes the proof of Lemma~\ref{induzsepar}.
\end{proof}

\begin{lemma}\label{lemcorcl.cardinality}
	Given two points $\sX=(x_1,\ldots,x_{N_k}), \sX'=(x_1',\ldots,x_{N_k}')\in \cS_k^{N_k}$, let $N$ be the smallest integer such that $x_N\neq x_N'$. Given $\gamma_{k-1}\in \mcE_{k-1}$, let 
\[
	\gamma_k
	\eqdef\Phi_{k-1}(\gamma_{k-1},\sX)
	,\quad
	\gamma_k'
	\eqdef\Phi_{k-1}(\gamma_{k-1},\sX').
\]
Then there exists $t_{k-1}+(N-1)(n_k+\ell_k)+a^\sharp_k<j<t_{k-1}+(N-1)(n_k+\ell_k)+n_k$ such that 
\[
	d(f^j(\gamma_k),f^j(\gamma_k'))>3\delta_0/5.
\] 
\end{lemma}

\begin{proof}
By items (b2) and (c) in Proposition \ref{p.separation-in-symbol},   $\Psi_k(\sX),\Psi_k(\sX')\in\DD_k$ are distinct center curves as in the assumption of Lemma~\ref{lemDsep}. 
 Combining with the proof of  Lemma~\ref{lemDsep} and arguing as in the proof of Lemma \ref{induzsepar} for the case where $\xi_{k+1}\neq\xi_{k+1}'$, the assertion follows.  
\end{proof}

\begin{proof}[Proof of Proposition \ref{prooooop}]
	The first assertion follows immediately from Lemma \ref{induzsepar}.
		
	By construction, the maps $\Phi_k$ and $\phi_k$ are onto (recall Remark \ref{rembox}). So it is enough to prove that they are injective. By items (b) and (c) in Proposition \ref{p.separation-in-symbol} and the definition of $\DD_k$, the map $\Psi_k:\cS_k^{N_k}\to\DD_k$ is bijective. Then it follows from Lemma~\ref{lemcorcl.cardinality} that $\Phi_k$ and $\phi_k$ are injective. 
\end{proof}

\begin{proof}[Proof of Corollary \ref{corlem:unique-predecessor}]
The first assertion is an immediate consequence of  Proposition \ref{prooooop}.
By Proposition \ref{prolempropos} item (2),  
 \[
 	\cW^\uu\big(f^i(\gamma_{k+1}),\delta/2^k\big)
	\subset \cW^\uu\big(f^i(\gamma_k),\delta/2^{k-1}\big), \textrm{~for every $0\leq i\leq t_{k-1}+T_k$.}
\]
By item (1) in Proposition~\ref{prolempropos}, $\gamma_{\ell}$ is $(\c, \cW^\uu(\gamma_{\ell-1},\delta/2^{\ell-1})$-complete. Then one can inductively show that  $\gamma_\ell\in\mcE_\ell$  is $(\c, \cW^\uu(\gamma_k,\delta/2^{k-1})$-complete. 
\end{proof} 

\subsection{Proof of item \eqref{THETHEObisbis-3} in Theorem \ref{THETHEObisbis}}\label{secproofTHEOBIS}

Proposition \ref{prolempropos} gives the sequences $(t_k)_k$, $(T_k)_k$, and $(\mcE_k)_k$. This defines the set $L$. Let us show the properties claimed in the theorem. 

Assume that $\gamma_{k+1}$ is a successor of $\gamma_k$.
By Proposition \ref{prolempropos} item (2), one has
\[
	f^{-i}\big(\cW^\uu(f^i(\gamma_{k+1}),\delta/2^k)\big)
	\subset f^{-i}\big(	\cW^\uu(f^i(\gamma_k),\delta/2^{k-1})\big),
	 \textrm{~for every $0\leq i\leq t_{k-1}+T_k$.}
\] 
Thus, one obtains 
\[
	B^{\cu}_{t_{k}+T_{k+1}}\big(\gamma_{k+1},\delta/2^{k}\big) 
	\subset  B^{\cu}_{t_{k-1}+T_k}\big(\gamma_k,\delta/2^{k-1}\big). 
\]
This proves that $L_{k+1}\subset L_k$ for all $k\in\mathbb{N}$. Hence, $L$ is a compact nonempty set.
 
By Proposition~\ref{prooooop}, the curves in $\mcE_k$ are $(t_{k-1}+T_k,\delta_0/2)$-separated. As by  \eqref{fixdelta} we have $\delta<\delta_0/20$, it follows that
\[
	\big\{B^{\cu}_{t_{k-1}+T_k}\big(\gamma_k,\delta/2^{k-1}\big)\colon \gamma_k\in \mcE_k\big\}
\]	 
is a collection of pairwise disjoint compact sets. By property (5) in Proposition~\ref{prolempropos},  each $\gamma_k$ is $(\c,\cW^\uu(\xi,\delta))$-complete for some $\xi\in \DD_1$. By the uniform contraction of $f^{-1}$ along the foliation $\cW^\uu$, one deduces that $L$ consists of non-degenerate compact center curves which are $(\c,\cW^\uu(\xi,2\delta))$-complete. This proves the theorem.
\qed

\section{Proof of Theorem \ref{THETHEObisbis}: Entropy estimates}\label{secentropy}

Throughout this section, we assume the hypotheses of Theorem \ref{THETHEObisbis}. 
Consider $L=L\big(\delta,(t_k)_k,(T_k)_k,(\mcE_k)_k\big)$ as in \eqref{defLL}, where $\delta$, $(t_k)_k$, $(T_k)_k$, and $(\mcE_k)_k$ are as in Proposition \ref{prolempropos}.
In Section \ref{proofitem1}, we complete the proof of item \eqref{THETHEObisbis-1} in Theorem \ref{THETHEObisbis}.  To accomplish this, we employ the entropy distribution principle in Appendix \ref{App:B}. For that, we need the auxiliary Proposition \ref{prol.convergence-of-probability-measure} below.  

For each $k\in\bN$, define the probability measure $\mu_k$ as 
\begin{equation}\label{defmuk}
	\mu_k
	\eqdef \frac{1}{\card \mcE_k} \sum_{\gamma_k\in \mcE_k}\Leb_{\gamma_k},
\end{equation}
where $\Leb_{\gamma_k}$ is the normalized Lebesgue measure on $\gamma_k$.
By definition, $\mu_k(L_k)=1$. 

\begin{proposition}\label{prol.convergence-of-probability-measure}
The sequence $(\mu_k)_k$ defined in \eqref{defmuk} converges in the weak$\ast$-topology to a probability measure $\mu$ such that $\mu(L)=1$. Moreover, for every $\theta>0$ there exists $n_0\in\bN$ such that for every $n> n_0$ and  every $x\in M$ satisfying $B_n(x,\delta/2)\cap L\neq\emptyset$, one has 
\begin{equation}\label{eqlocent}
	\mu(B_n(x,\delta/2))\leq e^{-n(\h(f)-\theta)}.
\end{equation}
\end{proposition}

The convergence of $(\mu_k)_k$ and the ``local entropy estimate'' of the limit measure $\mu$ in Proposition \ref{prol.convergence-of-probability-measure} are proven in Sections \ref{ssecconv} and \ref{seccentr}, respectively.

\subsection{Convergence of the sequence of measures $(\mu_k)_k$}\label{ssecconv}

To show convergence, it suffices to prove that for any continuous function $\varphi\colon M\to\bR$, the limit $\lim_{k\to\infty}\int\varphi \,d\mu_k$ exists. In what follows, we fix $\varphi$ and prove this assertion.
 
By item (5) in Proposition \ref{prolempropos}, each curve $\gamma_k\in \mcE_k$ is $(\c,\cW^\uu(\xi,\delta))$-complete for some $\xi\in \DD_1$. In particular, all the curves in $\bigcup_k\mcE_k$ have lengths uniformly bounded away from zero and also from above. By the uniform continuity of the center and the strong unstable bundles, and the uniform continuity of $\varphi$, for every $\omega>0$, there exists $\eta>0$ such that for every pair of $(\c,\cW^\uu(\xi,\delta))$-complete curves $\gamma,\gamma'$ (for some $\xi\in \DD_1$) with $\cW^\uu(\gamma,\eta)\cap\gamma'\ne\emptyset$, one has 
\begin{equation}\label{refertoothiis}
	\Big|\int\varphi \,d\Leb_\gamma-\int\varphi \,d\Leb_{\gamma'}\Big|<\omega.
\end{equation}

Given integers $\ell>k$, by Corollary~\ref{corlem:unique-predecessor}, for any curve  $\gamma\in \mcE_\ell$,  there exists a $\gamma'\in \mcE_k$ such that $\gamma$ is $(\c,\cW^\uu(\gamma',\delta/2^{k-1}))$-complete. 
On the other hand, inductively applying Lemma~\ref{lemcorcl.cardinality}, it follows that every curve $\gamma'\in \mcE_k$ gives rise to exactly $(\prod_{j=k+1}^\ell\card \DD_j)$ distinct curves in $\mcE_\ell\cap \cW^\uu(\gamma',\delta/2^{k-1})$. 

Consider $k\in\bN$ sufficiently large such that $\delta/2^{k-1}<\eta$. Thus, for every  $\ell>k$, by the choice of $\eta$, one has
\[\begin{split}
	&\Big|\int\varphi \,d\mu_\ell-\int\varphi \,d\mu_k\Big| \\
	&=\Big|\frac{1}{\card \mcE_\ell}
		\sum_{\gamma\in \mcE_\ell}\int\varphi \,d\Leb_{\gamma}
		-\frac{1}{\card \mcE_k}
		\sum_{\gamma'\in \mcE_k}\int\varphi \,d\Leb_{\gamma'}\Big|\\
	&\leq \frac{1}{\card \mcE_k}
		\sum_{\gamma'\in \mcE_k}
			\frac{\card \mcE_k}{\card \mcE_\ell}
	\left(\sum_{\textrm{$\gamma$ is $(\ell-k)$ successor of $\gamma'$}}
		\Big|\int\varphi \,d\Leb_{\gamma}-\int\varphi \,d\Leb_{\gamma'}\Big|\right)\\
	\text{\small{(by \eqref{refertoothiis})}}\quad	
	&\leq \omega.
\end{split}\]
This proves that the limit $\mu\eqdef\lim_k\mu_k$ indeed exists.

Recall that  $(L_k)_k$ is a nested sequence of compact sets and $\mu_k(L_k)=1$, thus for any $\ell>k$, one has $\mu_\ell(L_k)=1$. Taking the limit $\ell\to\infty$, one gets $\mu(L_k)\geq 1$. As $k$ was arbitrary, this gives $\mu(L)=1$. This proves the first part of Proposition~\ref{prol.convergence-of-probability-measure}.
\qed 

\subsection{Estimates of the ``local entropies'' of $\mu$}\label{seccentr}

We now estimate the $\mu$-measure of Bowen balls intersecting the set $L$. We start by some preliminary estimates.

Given $\theta>0$, as $\h(f)>0$, we fix $\e'\in(0,\theta)$ sufficiently small such that
\begin{equation}\label{eq:choice-of-eprime} 
(\h(f)-\e')(1-\e')(1-2\e')>\h(f)-\theta.
\end{equation}
By the choice in \eqref{extra-cond-ts}, \eqref{cucucu}, and \eqref{extra-cond-tsbis}, and by $0\le\ell_i\le\ell_i^\sharp$ for every $i\in\bN$,  there exists $k_0>0$ so that for all $i\geq k_0$
\begin{equation}\label{eq:choice-of-kepsilon}
	\max\Big\{|h_i-\varepsilon_i-\h(f)|, 
	\frac{\ell_i+m_i}{n_i+\ell_i+m_i},
	 \frac{t_{i-1}}{t_i},
	 \frac{n_{i+1}+\ell_{i+1}}{t_i}\Big\}
	 <\e'	.
\end{equation}
Fix now $n> t_{k_0+1}$. Consider the associated index $k\geq k_0+1$ such that $t_k\le n<t_{k+1}$ and let $N\in\{0,\ldots, N_{k+1}\}$ be the largest number such that 
\[
	t_{k}+N(n_{k+1}+\ell_{k+1})\leq n.
\]
Hence we get
\begin{equation}\label{eq:choice-of-j0}\begin{split}
	&t_{k_0}< t_{k}\leq  n< t_{k+1},\\
	&t_{k}+N(n_{k+1}+\ell_{k+1})\leq n< t_{k}+(N+1)(n_{k+1}+\ell_{k+1}).
\end{split}\end{equation}
In what follows, $k$ and $N$ are fixed with these properties.

\begin{lemma}
For every $x\in M$, there is at most one curve in $\mcE_{k}$ intersecting $B_n(x,\delta)$.
\end{lemma}

\proof
Arguing by contradiction, suppose that there are two curves $\gamma,\gamma'\in \mcE_{k}$ both intersecting $B_n(x,\delta)$. Let $z\in\gamma\cap B_n(x,\delta)$ and $z'\in\gamma'\cap B_n(x,\delta)$, then  
\[
	d(f^i(z),f^i(z'))\leq 	d(f^i(z),f^i(x))+	d(f^i(x),f^i(z')) <2\delta
	\quad\text{for $i=0,\ldots,n-1$.}
\] 
Recall $\delta\in(0,\delta_0/20)$. By  Proposition~\ref{prooooop}, the curves in $\mcE_{k}$ are $(t_{k-1}+T_{k},\delta_0/2)$-separated. Noting that $n\geq t_{k}>t_{k-1}+T_{k}$, one gets a contradiction. 
\endproof 

\begin{lemma}\label{lemcl.same-predecessor}
Let $\ell>k$ and consider $\gamma_\ell,\gamma_\ell'\in \mcE_\ell$. Let $\gamma_{k}$ and $ \gamma_{k}'$ be the $(\ell-k)$-predecessors of $\gamma_\ell$ and $\gamma_\ell'$, respectively. If $\gamma_\ell$ and $\gamma_\ell'$ intersect $B_n(x,\delta)$ for some $x\in M$, then  $\gamma_{k}=\gamma_{k}'$.
\end{lemma}

\proof 
By hypothesis, we can choose points $z\in \gamma_\ell\cap B_n(x,\delta)$ and ${z}'\in \gamma_\ell'\cap B_n(x,\delta)$. Inductively applying  item (2) in Proposition \ref{prolempropos}, we get
\[ 
	f^i(\gamma_\ell)\subset\cW^\uu\big(f^i(\gamma_{k}),\delta/2^{k-1}\big) 
	\textrm{~~and~~}f^i({\gamma}_\ell')\subset \cW^\uu\big(f^i({\gamma}_{k}'),\delta/2^{k-1}\big) 
	\,\text{ for $0\leq i\leq t_{k-1}+T_{k}-1$.}
\] 
 Therefore, there exist $w\in\gamma_{k} $ and ${w}'\in\gamma_{k}'$ such that
\[ 
	f^i(z)\in \cW^\uu(f^i(w),\delta) 
	\textrm{~and~}
	f^i({z}')\in\cW^\uu(f^i({w}'),\delta) \quad\textrm{for $0\leq i\leq t_{k-1}+T_{k}-1$}.
\] 
As $n\geq t_{k}>t_{k-1}+T_{k}$, for $0\leq i\leq t_{k-1}+T_{k}-1$, one has
\[\begin{split} 
	d(f^i(w),f^i({w}'))
	&\leq d(f^i(w),f^i(z))+d(f^i(z),f^i(x))+d(f^i(x),f^i({z}'))
		+d(f^i({z}'),f^i({w}'))\\
 \text{{\small (by \eqref{fixdelta})}}	& \leq 4\delta<\delta_0/2.
\end{split}\]
By  Proposition~\ref{prooooop}, two curves in $\mcE_{k}$ are $(t_{k-1}+T_{k},\delta_0/2)$-separated. Hence, $\gamma_{k}={\gamma}_{k}'$.
\endproof

For the next lemma, recall $N$ in \eqref{eq:choice-of-j0} and the $(n_{k},\delta_0)$-separated set $\cS_{k}$ in \eqref{eqcSkcardinality}.

\begin{lemma}\label{lemcl.upper-mass-bound}
For every $\ell>k$ and every $x\in M$, one has 
\[
	\mu_{\ell}\big(B_n(x,\delta/2)\big)\leq \frac{1}{\card \mcE_{k}\cdot(\card \cS_{k+1})^{N}} .
\]
\end{lemma}

\proof 
By the definition  of $\mu_{\ell}$ in \eqref{defmuk},  one has 
\begin{equation}\label{waiwai}\begin{split}
	\mu_\ell\big(B_n(x,\delta/2)\big)
	&=\frac{1}{\card \mcE_\ell}
		\sum_{\gamma_\ell\in \mcE_\ell}\Leb_{\gamma_\ell}\big(\gamma_\ell\cap B_n(x,\delta/2)\big)\\
	&\le\frac{1}{\card \mcE_\ell}
		\card\big\{\gamma_\ell\in\mcE_\ell\colon\gamma_\ell\cap B_n(x,\delta/2)\ne\emptyset\big\}	.
\end{split}\end{equation}
Thus, to estimate this measure from above, it suffices to count the curves in $\mcE_{\ell}$ which intersect $B_n(x,\delta/2)$. 

\begin{claim}\label{claimlunch}
For every $x\in M$, the number of  curves in $\mcE_{k+1}$ intersecting $B_n(x,\delta)$ is bounded from above by $(\card \cS_{k+1})^{N_{k+1}-N}$. 
\end{claim}

\begin{proof}
It follows from Lemma \ref{lemcl.same-predecessor} that all curves in $\mcE_{k+1}$ intersecting $B_n(x,\delta)$ have the same predecessor in $\gamma\in\mcE_{k}$. Given $\gamma_{k+1},\gamma_{k+1}'\in \mcE_{k+1}\cap B_n(x,\delta)$, take 
\[
	\sX=(x_1,\ldots,x_{N_{k+1}})\in \cS_{k+1}^{N_{k+1}}
	~\textrm{~~and~~}~ ~
	\sX'=({x}_1',\ldots,{x}_{N_{k+1}}')\in \cS_{k+1}^{N_{k+1}}
\] 
such that $\Phi_{k}(\gamma,\sX)=\gamma_{k+1}$ and $\Phi_{k}(\gamma,\sX')=\gamma_{k+1}'$.
By Lemma~\ref{lemcorcl.cardinality} and using the arguments in the proof of Lemma~\ref{lemcl.same-predecessor}, one has that $x_i={x}_i'$ for every $i\leq N$. As $0\leq N\leq N_{k+1}$, this implies the claim.
\end{proof}

By Corollary \ref{garb} and Remark \ref{remcardDk}, one has 
\[
	\card \mcE_{k+1}
	= \card \mcE_{k}\cdot \card \cS_{k+1}^{N_{k+1}}.
\]
Hence, by  Claim \ref{claimlunch} and applying \eqref{waiwai} to $\ell=k+1$, one gets that  
\[
	\mu_{k+1}\big(B_n(x,\delta/2)\big)
	\leq \frac{\card \cS_{k+1}^{N_{k+1}-N}}{\card \mcE_{k+1}} 
	=\frac{1}{\card \mcE_{k}\cdot (\card \cS_{k+1})^{N}}.
\]
This finishes the proof for $\ell=k+1$. For $\ell>k+1$, by applying  item (2)  in Proposition \ref{prolempropos} inductively, each curve $\gamma_{\ell}\in \mcE_{\ell}$ and its $(\ell-(k+1))$-predecessor $\gamma_{k+1}\in\mcE_{k+1}$ satisfy
\[
	f^i(\gamma_{\ell})\subset 
	\cW^\uu(f^i(\gamma_{k+1}),\delta/2^{k}), 
	\quad\text{ for every $0\leq i\leq  t_{k}+T_{k+1}-1$}. 
\]
Recall that $	t_{k}+N(n_{k+1}+\ell_{k+1})\leq n< t_{k}+(N+1)(n_{k+1}+\ell_{k+1})$ with $0\leq N\leq N_{k+1}$ and $n<t_{k+1}$ (see \eqref{eq:choice-of-j0}). Now, we have two cases to consider.

\noindent\textbf{Case (1) $0\leq N< N_{k+1}$.} Thus $n<t_k+T_{k+1}$.  If $\gamma_{\ell}\in \mcE_{\ell}$ intersects the Bowen ball $B_n(x,\delta/2)$, then its (unique) $(\ell-(k+1))$-predecessor in $\mcE_{k+1}$   intersects  $B_n(x,\delta)$.

\noindent\textbf{Case (2) $N=N_{k+1}$.} Thus $t_k+T_{k+1}\leq n<t_{k+1}$. If $\gamma_{\ell}\in \mcE_{\ell}$ intersects the Bowen ball $B_n(x,\delta/2)$, then its (unique) $(\ell-(k+1))$-predecessor in $\mcE_{k+1}$   intersects  $B_{t_k+T_{k+1}}(x,\delta)$. By  item (1) in Proposition~\ref{prooooop}, there is at most one center curve intersecting the Bowen ball $B_{t_k+T_{k+1}}(x,\delta)$.

Thus, by \eqref{waiwai}, one has 
\[\begin{split}
	\mu_{\ell}\big(B_n(x,\delta/2)\big)
	&\le\frac{1}{\card \mcE_\ell}
		\card\big\{\gamma_\ell\in\mcE_\ell\colon\gamma_\ell\cap B_n(x,\delta/2)\ne\emptyset\big\}	\\
	\text{\small{(by Claim \ref{claimlunch})}}\quad
	&\leq \frac{1}{\card \mcE_{\ell}}
		(\card \cS_{k+1})^{N_{k+1}-N}
		\cdot \frac{\card\mcE_\ell}{\card\mcE_{k+1}}\\
	\text{\small{(by Corollary \ref{garb})}}\quad
	&= \frac{1}{\card \mcE_{\ell}}
		(\card \cS_{k+1})^{N_{k+1}-N}\cdot\prod_{i=k+2}^{\ell}\card \DD_i\\
	&= \frac{(\card \cS_{k+1})^{N_{k+1}-N}}
		{\card \mcE_{k+1}}
	=\frac{1}{\card \mcE_{k}(\card \cS_{k+1})^{N}}.
\end{split}\]
This proves the lemma.
\endproof 

\begin{claim}
	It holds
\[
	\frac{1}{n}\log\big(\card \mcE_{k}(\card \cS_{k+1})^{N}\big)
	> \h(f)-\theta. 
\]
\end{claim}	

\begin{proof}
It follows from Corollary \ref{garb} and Remark \ref{remcardDk} that
\[
	\card \mcE_{k}
	= \prod_{i=1}^{k} \card \DD_i
	=\prod_{i=1}^{k} \big(\card\cS_i\big)^{N_i}.
\]
By our choice of $k_0$ in \eqref{eq:choice-of-kepsilon}, one has
\begin{equation}\label{eqsomethingnewer}
	h_k-\varepsilon_k
	\ge \h(f)-\varepsilon',\quad
	t_{k_0-1} \le \varepsilon't_{k_0},
	\quad\text{ and }\quad	
	(n_{k+1}+\ell_{k+1})\le\varepsilon' t_k.
\end{equation} 
Hence, one has the following estimation 
\[\begin{split}
	&\frac{1}{n}\log\big(\card \mcE_{k}(\card \cS_{k+1})^{N}\big)=\frac{1}{n}\log\big(\prod_{i=1}^{k} \big(\card\cS_i\big)^{N_i}\cdot (\card \cS_{k+1})^{N}\big)
	\\
\text{\small (by \eqref{eqcSkcardinality})}	&\geq \frac{1}{n}\sum_{i=1}^{k}\big(N_i n_i (h_i-\varepsilon_i)+N n_{k+1} (h_{k+1}-\varepsilon_{k+1}) \big)
	\\
	&\geq \frac{1}{n}\sum_{i=k_0}^{k}\big(N_i n_i (h_i-\varepsilon_i)+N n_{k+1} (h_{k+1}-\varepsilon_{k+1}) \big)
	\\
	\text{\small{(by~\eqref{eq:choice-of-kepsilon})}}\quad	
    &\ge \frac{1-\varepsilon'}{n}\sum_{i=k_0}^{k}\bigg(\big(N_i (n_i+\ell_i)+m_i\big)\cdot (h_i-\varepsilon_i)+N (n_{k+1}+\ell_{k+1}) (h_{k+1}-\varepsilon_{k+1})   \bigg)
    	\\
    \text{\small{(by~\eqref{eq:choice-of-kepsilon})}}\quad	
   	&\ge \frac{1-\e'}{n}(\h(f)-\varepsilon')
		\sum_{i=k_0}^{k}\bigg(\big(N_i (n_i+\ell_i)+m_i\big)+N (n_{k+1}+\ell_{k+1})   \bigg)\\
	&= (\h(f)-\varepsilon')(1-\e')\frac1n\big(t_k-t_{k_0-1} +N(n_{k+1}+\ell_{k+1})\big)	\\
\end{split}\]
By \eqref{eq:choice-of-j0}, we get
\[\begin{split}
\frac{1}{n}\log\big(\card \mcE_{k}(\card \cS_{k+1})^{N}\big)
	&\ge  (\h(f)-\varepsilon')(1-\e')\frac1n\big(n -t_{k_0-1} - (n_{k+1}+\ell_{k+1})\big)
		\\
	\text{\small (by \eqref{eqsomethingnewer})}\quad
	&\ge (\h(f)-\varepsilon')(1-\e')\frac1n\big(n -\varepsilon't_{k_0} - \varepsilon't_k\big)
	\\	
	\text{\small{(by \eqref{eq:choice-of-j0},  \eqref{eq:choice-of-eprime})}}\,
	&> (\h(f)-\varepsilon')(1-\e')(1 -2\varepsilon')	
	>\h(f)-\theta.
\end{split}\]
This proves the claim.
\end{proof}

By Lemma~\ref{lemcl.upper-mass-bound}, for every $\ell>k$ and $x\in M$, one gets $\mu_{\ell}(B_n(x,\delta/2))<e^{-n(\h(f)-\theta)}$. This implies 
\[
	\liminf_{\ell\to\infty}\mu_\ell(B_n(x,\delta/2))\leq e^{-n(\h(f)-\theta)}.
\]	 
As $B_n(x,\delta/2)$ is open and $\mu_\ell$ converges to $\mu$, by Portmanteau's theorem one has 
\[\mu(B_n(x,\delta/2))\leq \liminf_{\ell\to\infty}\mu_\ell(B_n(x,\delta/2))\leq e^{-n(\h(f)-\theta)},\] 
proving the second part of Proposition~\ref{prol.convergence-of-probability-measure}.
\qed

\subsection{Proof of Theorem \ref{THETHEObisbis} item \eqref{THETHEObisbis-1}}\label{proofitem1}

Let $\theta\in(0,\h(f))$.	
Let $\delta$ be sufficiently small as in Proposition \ref{prolempropos} and consider the asociated sequences $(t_k)_k$, $(T_k)_k$, and $(\mcE_k)_k$. Let $L_\delta=L\big(\delta,(t_k)_k,(T_k)_k,(\mcE_k)_k\big)$ as in \eqref{defLL} (in the following arguments, the only relevant parameter is $\delta$). Let $\mu_\delta$ be the probability measure and $n_0\in\bN$ provided by Proposition \ref{prol.convergence-of-probability-measure} such that $\mu_\delta(L_\delta)=1$ and that for every $n>n_0$ and $x\in M$ satisfying $B_n(x,\delta/2)\cap L_\delta\ne\emptyset$, it holds
\[
	\mu_\delta(B_n(x,\delta/2))\le e^{-n(\h(f)-\theta)}.
\]	 
Hence, by Theorem \ref{thm.entropy-distribution-principle} we get
\[
	 h_{\rm top}(f,L_\delta,\delta/2)
	\ge \h(f)-\theta.
\]
Lemmas \ref{lemmono} and \ref{lem:entropy} together imply that
\[
	h_{\rm top}(f,L_\delta)
	= \lim_{\varepsilon\to0}  h_{\rm top}(f,L_\delta,\varepsilon)
	\ge   h_{\rm top}(f,L_\delta,\delta/2)
	\ge \h(f)-\theta.
\]
As $\theta\in(0,\h(f))$ was arbitrary, this implies item \eqref{THETHEObisbis-1} in Theorem \ref{THETHEObisbis}.
\qed

\subsection{Comments on Scholium \ref{schol2}}\label{secSchol}

Let us briefly explain how this scholium is derived.
In Equation~\eqref{defmuk}, instead of using the Lebesgue measure on the curves $\gamma_k\in\mcE_k$, one can also define a probability measure on $L$ in the following way. Fix $\xi\in\mcE_1$ and a point $x\in \xi$. Note that $\gamma_k\in\mcE_k$ intersects $\cW^\uu(x,2\delta)$ in at most one point. Then one can define a probability measure 
\[\widehat \mu_k\eqdef
 \frac{\card \mcE_1}{\card \mcE_k} \sum_{\gamma_k\in \mcE_k}\delta_{\gamma_k\cap \cW^\uu(x,2\delta)}. \] 
 Analogously as in the proof of Proposition \ref{prol.convergence-of-probability-measure}, one gets that $\widehat\mu_k$ converges to a probability measure $\widehat\mu$ such that 
 \begin{itemize}
 	\item $\widehat\mu(L\cap \cW^\uu(x,2\delta))=1$;
 	\item  
 	for every $\theta>0$ there exists $n_0\in\bN$ such that for every $n> n_0$ and  every $z\in M$ satisfying $B_n(z,\delta/2)\cap \big(L\cap  \cW^\uu(x,2\delta)\big)\neq\emptyset$, one has 
 	\begin{equation*}
 		\widehat \mu(B_n(z,\delta/2))\leq e^{-n(\h(f)-\theta)}.
 	\end{equation*}
 \end{itemize}
 Arguing as in Section~\ref{proofitem1} and using Theorem \ref{thm.entropy-distribution-principle}, one can conclude. 
\qed

\section{End of the proof of Theorem \ref{THETHEObisbis}: Forward exponents}\label{secproofTHEO}

Throughout this section, we assume the hypotheses of Theorem \ref{THETHEObisbis}. 
Consider $L=L\big(\delta,(t_k)_k,(T_k)_k,(\mcE_k)_k\big)$ as in \eqref{defLL}, where $\delta$, $(t_k)_k$, $(T_k)_k$, and $(\mcE_k)_k$ are as in Proposition \ref{prolempropos}.
The following proposition is item \eqref{THETHEObisbis-2} in Theorem \ref{THETHEObisbis}.

\begin{proposition}\label{prolem:center-Lyapunov-exponent-converge-to-zero}
For every $x\in L$, $\lim_{n\to+\infty}\frac{1}{n}\log\|D^{\c}f^n(x)\|=0$ and this convergence is uniform on $L.$
\end{proposition}

\proof 
 Recall the quantifiers in Section \ref{ssecquanti}. 
 Recall $C_{\rm max}$ in \eqref{defCmax} and note that $|\chi_i|\leq C_{\rm max}/2$ for every  $i\in\mathbb{N}$.
 
 Fix $\e>0$.
By item (ii) in Section \ref{ssecquanti},  $\chi_i$ and $\varepsilon_i$ are chosen monotonically tending to zero. Hence there exists $k_0\in\mathbb{N}$ such that
\begin{equation}\label{eq:bound-LE-varepsilon}
	|\chi_i|+6\varepsilon_i<\e, 
	\quad\textrm{~ for every $i\geq k_0$.} 
\end{equation}  

Recall that \eqref{onemore} and \eqref{extra-cond-ts} imply that for all $i\in\bN$
\begin{equation}\label{useful}
	m_i C_{\rm max}<n_i\varepsilon_i<T_i\varepsilon_i<t_i\varepsilon_i
	\quad\text{ and }\quad
	t_iC_{\rm max}<t_{i+1}\varepsilon_{i+1}.
\end{equation}

Fix $n>t_{k_0+1}$. As in \eqref{eq:choice-of-j0}, we consider the associated index $k\geq k_0+1$ and $N\in\{0,\ldots,N_{k+1}\}$ such that
\begin{equation}\label{eqfixingn}\begin{split}
	&t_k\le n<t_{k+1},\\
	&t_k+N(n_{k+1}+\ell_{k+1})\le n<t_k+(N+1)(n_{k+1}+\ell_{k+1}).
\end{split}\end{equation}
From now on, $k$ and $N$ are fixed with the above properties.

Observe that
\begin{equation}\label{eqfixingnbis}\begin{split}
	\log\|D^{\c}f^{n}(x)\|
	= \log\|D^{\c}f^{t_k}(x)\| +\log\|D^{\c}f^{n-t_k}(f^{t_k}(x))\|.
\end{split}\end{equation}

For the next two lemmas, recall the distortion control-time $\ell_k^\flat=(\delta,\varepsilon_k)$ in Lemma~\ref{lemsomecount}. 

\begin{lemma}\label{lemc.estimate-center-LE-to-end}
	For every $x\in L$,  one has 
\[
	\big|\log\|D^{\c}f^{t_k}(x)\|\big|
	< t_k(|\chi_{k_0}|+5\varepsilon_{k_0}).
\]	
\end{lemma}
\begin{proof}
Fix $x\in L$. Recalling the definition of $t_k$ in \eqref{deftk}, 
\[
	t_k= t_{k_0-1}+\sum_{j=k_0}^k(T_j+m_j),
\]	 
let us first write
\begin{equation}\label{anumber}\begin{split}
	\log\|D^{\c}f^{t_k}(x)\|
	&=\log\|D^{\c}f^{t_{k_0-1}}(x)\|\\
	&+\sum_{j=k_0}^k\Big(\log\|D^{\c}f^{T_j}(f^{t_{j-1}}(x))\|
		+\log\|D^{\c}f^{m_j}(f^{t_j-m_j}(x))\|\Big).
\end{split}\end{equation}
As a first step, for every $k_0\le j\le k+1$, we choose points $y_j\in\Lambda_j$ that have controlled finite-time Lyapunov exponents approximating the one of $x$. 
By the definition of $L$ in \eqref{defLL}, for every $k_0\le j\le k+1$, there exists a (unique) center curve $\gamma_{j}\in \mcE_{j}$ such that 
\[ 
x\in B^{\cu}_{t_{j-1}+T_{j}}\big(\gamma_{j},\delta/2^{j-1}\big),
\] 
and thus there exists a point $y_{j}\in  \gamma_j$ such that 
\begin{equation}~\label{eq:follow-in-unstable}
	f^i(x)\in \cW^\uu\big( f^i(y_j),\delta/2^{j-1}\big)
	\subset \cW^\uu\big(f^{i}(\gamma_j),\delta/2^{j-1}\big)
	\quad\textrm{~~for $t_{j-1}\leq i\leq t_{j-1}+T_j-1$}.
\end{equation}
On the other hand, by item (3) in Proposition~\ref{prolempropos}, there exists a curve $\xi_{j}\in \DD_j$ such that $f^{t_{j-1}}(\gamma_j)\subset\cW^\ss(\xi_j,\delta/2^{j-1})$ and in particular,
\[
	d(f^i(\gamma_j),f^i(\xi_j))
	\leq \delta/2^{j-1} 
	\quad\textrm{~~for $t_{j-1}\leq i\leq t_{j-1}+T_j-1$}.
\] 
By the definition of $\DD_j$ (see Lemma \ref{lemclaDk}) and item (iv) in Section \ref{ssecquanti} for $\Lambda_j$ and the point $f^{t_{j-1}}(y_j)\in\cW^\ss(\xi_j,\delta/2^{j-1})$, one has
\begin{equation}\label{thiseq}
	\big|\log\|D^{\c}f^{T_j}(f^{t_{j-1}}(y_j))\|-T_j\chi_{j}\big|
	\leq \log K_{j}+T_j\varepsilon_{j}.
\end{equation}
By \eqref{eq:condnk-1}, we have $\ell_j^\flat\leq n_j< T_j$ and hence, together with \eqref{eq:follow-in-unstable}, we can apply Lemma~\ref{lemsomecount} to the points $f^{t_{j-1}}(x)$ and $f^{t_{j-1}}(y_j)$, and get 
\begin{equation}\label{eqA}
	\begin{split}
	&	\big|\log\|D^{\c}f^{T_j}(f^{t_{j-1}}(x))\|-T_j\chi_{j}\big|
	\\
		&\leq\big|\log\|D^{\c}f^{T_j}(f^{t_{j-1}}(x))\|-\log\|D^{\c}f^{T_j}(f^{t_{j-1}}(y_j))\| \big|\\
		&\hspace{3mm}+	\big|\log\|D^{\c}f^{T_j}(f^{t_{j-1}}(y_j))\|-T_j\chi_j\big| \\
		\text{\small{(by Lemma~\ref{lemsomecount} and \eqref{thiseq})}}\quad
		&\leq  T_j\varepsilon_j+\big(T_j\varepsilon_j+\log K_j\big)\\
		&=2T_j\varepsilon_j+\log K_j.
\end{split}\end{equation}
In particular, recalling that $T_j=N_j(n_j+\ell_j)$, we get
\begin{equation}\label{eqB}\begin{split}
	\big|\frac{1}{T_j}\log\|D^{\c}f^{T_j}(f^{t_{j-1}}(x))\|-\chi_{j}\big|
	&\leq 2\varepsilon_j+\frac{\log K_j}{T_j}
	= 2\varepsilon_j+\frac{\log K_j}{N_j(n_j+\ell_j)}\\
	\text{\small{(by \eqref{eqexp1bisbis})}}\quad
	&\le 3\varepsilon_j.
\end{split}\end{equation}
Hence, recalling \eqref{anumber} and putting all steps together, one gets
\[\begin{split}
	\big|\log\|D^{\c}f^{t_k}(x)\|\big|
	&=\big|\log\|D^{\c}f^{t_{k_0-1}}(x)\|\\
	&+\sum_{j=k_0}^k\big(\log\|D^{\c}f^{T_j}(f^{t_{j-1}}(x))\|
		+\log\|D^{\c}f^{m_j}(f^{t_j-m_j}(x))\|\big)\big|\\
	\text{\small{(by \eqref{defCmax} and  \eqref{eqB})}}\quad
	&\leq  t_{k_0-1} C_{\rm max} +\sum_{j=k_0}^k\big(T_j(|\chi_j|+3\varepsilon_j)+m_jC_{\rm max}\big)\\
	\text{\small{(by \eqref{useful})}}\quad
	&\leq t_{k_0} \varepsilon_{k_0} +\sum_{j=k_0}^kT_j(|\chi_j|+4\varepsilon_j)\\
	&\leq  t_k(|\chi_{k_0}|+5\varepsilon_{k_0}),
\end{split}\] 
where the last inequality follows from $t_{k_0}+\sum_{j=k_0}^kT_j<t_{k}$, and  $\varepsilon_j\le \varepsilon_{k_0}$ and $|\chi_j|\le |\chi_{k_0}|$  for every $j\geq k_0$ (see the property (ii) in Section \ref{ssecquanti}). This proves the lemma.
\end{proof}

With Lemma \ref{lemc.estimate-center-LE-to-end} at hand, we are now ready to estimate $\log\|D^\c f^n(x)\|$ in \eqref{eqfixingnbis}. There are three cases to consider, according to the position of $n$ in \eqref{eqfixingn}. \smallskip

\noindent\textbf{Case (1) $t_{k} \leq n< t_{k}+\ell_{k+1}^\flat$.}
The fact that $n>t_{k_0+1}$ together with our choices in \eqref{eq:condnk-5} and property (ii) in Section \ref{ssecquanti} imply 
\[\begin{split}
	\frac{1}{n}\big|\log\|D^{\c}f^{n-t_{k}}(f^{t_{k}}(x))\|\big|
	\leq \frac{\ell_{k_0+1}^\flat}{n} C_{\rm max}
	\leq \frac{\ell_{k_0+1}^\flat}{t_{k_0+1}} C_{\rm max}<\varepsilon_{k_0}.
\end{split}\] 
Hence, with Lemma~\ref{lemc.estimate-center-LE-to-end} we get
\[\begin{split}
	\frac{1}{n}\big|\log\|D^{\c}f^{n}(x)\|\big|
	&\le \frac{1}{n}\Big(\log\|D^{\c}f^{t_k}(x)\| +\log\|D^{\c}f^{n-t_k}(f^{t_k}(x))\|\Big)\\
	&\le \frac1n \Big(
		t_{k}(|\chi_{k_0}|+5\varepsilon_{k_0})+\log\|D^{\c}f^{n-t_{k}}(f^{t_{k}}(x))\|\Big)
	<|\chi_{k_0}|+6\varepsilon_{k_0}	
	<\e,
\end{split}\]
where the last inequality follows from~\eqref{eq:bound-LE-varepsilon}.\medskip

\noindent\textbf{Case (2) $t_{k}+\ell_{k+1}^\flat\leq n< t_{k}+T_{k+1}$.}
Recall again that $\varepsilon_j\searrow0$ monotonically and $t_k\geq N_k(n_k+\ell_k)$ due to \eqref{deftk}.
Note that $\ell^\flat_{k+1}\le n-t_{k}$. Hence, arguing as in  \eqref{eqA} for $j=k$, we get
\[\begin{split}
	\big\lvert\log\|D^{\c}f^{n-t_{k}}(f^{t_{k}}(x))\|-(n-t_{k})\chi_{k+1}\big\rvert
	&\le 2(n-t_{k})\varepsilon_{k+1}+\log K_{k+1}\\
	\text{\small{(by \eqref{eqexp1bis})}}\quad
	&\le2(n-t_{k})\varepsilon_{k+1} +n_{k}\varepsilon_{k}
	\\
	&<2(n-t_{k})\varepsilon_{k_0} +t_{k}\varepsilon_{k_0}.
\end{split}\]
Hence, as $|\chi_k|\searrow0$ monotonically, one has
\[
	\big|\log\|D^{\c}f^{n-t_{k}}(f^{t_{k}}(x))\|\big|
	\le (n-t_{k})(|\chi_{k+1}|+2\varepsilon_{k_0})+t_{k}\varepsilon_{k_0}
	<(n-t_{k})(|\chi_{k_0}|+2\varepsilon_{k_0})+t_{k}\varepsilon_{k_0}.
\]
Combining with Lemma~\ref{lemc.estimate-center-LE-to-end} and \eqref{eq:bound-LE-varepsilon}, one gets   
\[
	\frac{1}{n}\log\|D^{\c}f^{n}(x)\|
	<\frac1n\big(t_k(|\chi_{k_0}|+5\varepsilon_{k_0})+(n-t_{k})(|\chi_{k_0}|+2\varepsilon_{k_0})+t_{k}\varepsilon_{k_0}\big)
	<|\chi_{k_0}|+6\varepsilon_{k_0}<\e.
\] 

\noindent\textbf{Case (3) $t_{k}+T_{k+1}\leq n< t_{k+1}$.} Note that in this case  Equation \eqref{deftk} implies
\begin{equation}\label{eqblue}
	n-t_{k}-T_{k+1}
	< m_{k+1}.
\end{equation}
First observe that
\[\begin{split}
	&\frac{1}{n}\log\|D^{\c}f^{n-t_{k}}(f^{t_{k}}(x))\|\\
	&= \frac{1}{n}\log\|D^{\c}f^{T_{k+1}}(f^{t_{k}}(x))\|
		+\frac{1}{n}\log\|D^{\c}f^{n-t_{k}-T_{k+1}}(f^{t_{k}+T_{k+1}}(x))\|.
\end{split}\] 
Arguing as in  \eqref{eqB} for $j=k+1$ and using \eqref{eqblue}, we continue to estimate  
\[\begin{split}
	\frac{1}{n}\log\|D^{\c}f^{n-t_{k}}(f^{t_{k}}(x))\|
	&\le  \frac{T_{k+1}}{n}(\chi_{k+1}+3\varepsilon_{k+1})
		+\frac{m_{k+1}}{n}C_{\rm max}\\
	\text{\small{(as we are in Case (3))}}\quad
	&\leq 		\frac{n-t_{k}}{n}(|\chi_{k+1}|+3\varepsilon_{k+1})+\frac{m_{k+1}}{t_{k}}C_{\rm max}\\
	\text{\small{(by \eqref{useful}, $|\chi_{j+1}|<|\chi_{j}|$, and $\varepsilon_{j+1}<\varepsilon_{j}$)}}\quad
	&<\frac{n-t_{k}}{n}(|\chi_{k_0}|+3\varepsilon_{k_0})+\varepsilon_{k_0}.
\end{split}\] 
By Lemma~\ref{lemc.estimate-center-LE-to-end}, one gets   
\[\begin{split}
	\frac{1}{n}\log\|D^{\c}f^{n}(x)\|
	&<\frac{1}{n}\big(t_{k}(|\chi_{k_0}|+5\varepsilon_{k_0})+\log\|D^{\c}f^{n-t_{k}}(f^{t_{k}}(x))\|\big)\\
	&\leq \frac{1}{n}\big(t_{k}(|\chi_{k_0}|+5\varepsilon_{k_0})+(n-t_{k})(|\chi_{k_0}|+3\varepsilon_{k_0})\big)+\varepsilon_{k_0}\\
	\text{\small (by \eqref{eq:bound-LE-varepsilon})}\quad
	&\leq	|\chi_{k_0}|+6\varepsilon_{k_0}<\e.
\end{split}\] 

To summarize, for every $n> t_{k_0+1}$,  one has $\frac{1}{n}\log\|D^{\c}f^{n}(x)\|<\e$ which implies that the point $x$ has zero center Lyapunov exponent and the convergence is uniform. This finishes the proof of the proposition.
\endproof

\section{Fractal sets in $\cL(0)$: Proof of Theorem~\ref{theorem2}}\label{s.final-fractal}

By Theorem \ref{THETHEObisbis}, we have an estimate of the \emph{forward} Lyapunov exponent and the topological entropy of the fractal set $L=L\big(\delta,(t_k)_k,(T_k)_k,(\mcE_k)_k\big)=\bigcup\gamma_\infty$ in \eqref{defLL}. In Section \ref{secBacCon}, we construct a set $\widehat L\subset \cW^{\ss}_{\rm loc}(L)\cap \cL(0)$ which has the same entropy as $L$. The proof of Theorem~\ref{theorem2} is completed in Section \ref{finnnnnallll}.

\subsection{The set $\widehat L$: Backward concatenations}\label{secBacCon}

\begin{proposition}\label{finalprop}
Let  $L=L\big(\delta,(t_k)_k,(T_k)_k,(\mcE_k)_k\big)$ be the set provided by Theorem \ref{THETHEObisbis}. 
For every curve $\gamma_\infty$ of $L$, there is a center curve $\widehat\gamma_\infty$ such that
\begin{itemize}[ leftmargin=0.7cm ]
\item $\widehat\gamma_\infty\subset \cW^{\ss}(\gamma_\infty,2\delta)$ and is  $(\c, \cW^\ss(\gamma_\infty, 2\delta))$-complete,
\item $\lim_{n\to\pm\infty}\frac1n\log\|D^\c f^n(x)\|=0$ and this convergence is uniform on $\widehat L\eqdef \bigcup\widehat\gamma_\infty$.	
\end{itemize}
In particular, $\widehat L\subset\cL(0)$. Moreover, $h_{\rm top}(f,\widehat L)=h_{\rm top}(f,L)$.
\end{proposition}

\begin{proof}
Recall the choice of the ``pivotal'' point $q_0$ and the numbers  $r^\c,r^\s,r^\u_1,r^\u_2$ in Remark \ref{pivotal}. 
By hypothesis (B), $f$ has an unstable blender-horseshoe $\blender^+$. 
Given $r>0$, let $N(r)\in\mathbb{N}$ and $R(N(r))>0$ be its associated constants as in Lemma~\ref{lem:insideblenders} (with respect to $\blender^+$). 

Recall that, by item \eqref{THETHEObisbis-3} in Theorem~\ref{THETHEObisbis}, each center curve   $\gamma_\infty$ in $L$ is $(\c,\cW^\uu(\xi,2\delta))$-complete  
  for some curve $\xi\in \DD_1$. Now we are ready to construct for each $\gamma_\infty\subset L$ an auxiliary curve $\gamma_\infty'$. 

As $\gamma_\infty$ is $(\c,\cW^\uu(\xi,2\delta))$-complete and $2\delta<r^u$ due to Equation \eqref{fixdelta}, by Lemma \ref{remfirstLambda},  there is a center curve $\gamma_\infty' \subset \cW^{\u}(q,r^\u_1)$ which is $(\c,\cW^\ss(\gamma_\infty,r^\s))$-complete. 
 The point $q_0$ is now used to start a cascade of concatenations of center curves, each of which has   center exponents uniformly converging to zero (for $f^{-1}$). 

As $q_0$ is a hyperbolic periodic point with exponent $\chi^\c(q_0)>0$ and with unstable manifold $\cW^{\u}(q_0,r^\u_1)$ of inner radius $r^\u_1$, fixing any constant $\varepsilon_1\in(0,\chi^\c(q_0))$ there is $C_1\in\bN$ such that for every $x\in\gamma_\infty'$ and every $n\in\bN$ one has 
\[ 
	C_1^{-1} \cdot e^{-n(\chi^\c(q_0)+\varepsilon_1)}
	\leq \|D^{\c}f^{-n}(x)\|
	\leq C_1 \cdot e^{-n(\chi^\c(q_0)-\varepsilon_1)} .
\]

Analogously as in Section~\ref{ssecquanti}, there exists a sequence $(\Gamma_k)_k$ of basic sets of expanding type such that $\chi^\c(\nu_k)\to0$ as $k\to\infty$ uniformly for every sequence of invariant measures $\nu_k\in\cM(f|_{\Gamma_k})$. Moreover, we get a sequence $r_k\searrow0$, a sequences of positive numbers $(C_k)_k$, and sequences $\chi_k\searrow0$, and $\varepsilon_k\searrow0$.  
In each set $\Lambda_k$, we choose one center curve $\xi_k$ centered at some point in $\Lambda_k$ of length $r_k<\frac12 R(N(\delta/2^k))$ such that for every $y\in\xi_k$ and $n\in\bN$, one has 
\[ 
	C_k^{-1}  e^{-n(\chi_k+\varepsilon_k)}\leq \|D^{\c}f^{-n}(y)\|\leq C_k  e^{-n(\chi_k-\varepsilon_k)} .
\]

Consider now the sequence of center curves $\gamma_\infty',\xi_2,\xi_3,\ldots$. By Scholium \ref{schol} (applied to $f^{-1}$), we get a sequence of center curves $(\gamma_k)_k$ such that each $\gamma_k$ is $(\c,\cW^\ss(\gamma_\infty',\delta))$-complete. Now we let
\[
	\widehat\gamma_\infty
	\eqdef \lim_{k\to\infty}\gamma_k.
\]
By construction and  arguing as in Section \ref{secproofTHEO}, we get that for every $x\in\widehat\gamma_\infty$ we have
\begin{itemize}
\item  $\widehat\gamma_\infty\subset \cW^{\ss}(\gamma_\infty',2\delta)$; 
\item  $\lim_{n\to-\infty}\frac1n\log\,\lVert D^\c f^n(y)\rVert=0, $ for every $y\in\widehat\gamma_\infty$.
\end{itemize}
Note that, \emph{a priori}, the speed of the latter limit may depend on the curve $\gamma_\infty$ chosen. However, arguing as in Section \ref{secproofTHEO}, we get that it is uniform. Also, by construction, $\widehat L\eqdef \bigcup\widehat\gamma_\infty$ is a collection of pairwise disjoint center curves.

Observe finally that, by construction, $\widehat L \subset \cW^\ss(L,2\delta)$, and conversely one has $L\subset \cW^\ss(\widehat L,2\delta)$. By Theorem~\ref{theop.entropy-in-cu}, one has that $h_{\rm top}(f,\widehat L)=h_{\rm top}(f,L)$. This proves the proposition.
\end{proof} 

\subsection{End of the proof of Theorem \ref{theorem2}}\label{finnnnnallll}

Let $L_\delta=L(\delta,(t_k)_k,(T_k)_k,(\mcE_k)_k\big)$ as in Theorem \ref{THETHEObisbis} (here, the only relevant parameter is $\delta$).
Let $\widehat L_\delta\subset\cL(0)$ be the associated set of non-degenerate center curves provided by Proposition \ref{finalprop}, then  $h_{\rm top}(f,\widehat L_\delta)=h_{\rm top}(f,L_\delta)$. Hence, $h_{\rm top}(f,\widehat L_\delta)\ge\h(f)$.
 As $\widehat L_\delta\subset\cL(0)$, the monotonicity of entropy (Remark \ref{rem:monent}) implies
\[
	h_{\rm top}(f,\cL(0))
	\ge h_{\rm top}(f,\widehat L_\delta)
	\ge \h(f).
\]
By Theorem \ref{theorem1}, we proved already $h_{\rm top}(f,\cL(0))\le\h(f)$. Hence, we get equality.
Finally,  $\h(f)\ge\cH(0)>0$ was already shown in Remark \ref{remLFTransforFigures}.
This proves the theorem.
\qed

\appendix
\section{Topological entropy}\label{App:B}

\subsection{Definition of topological entropy and a distribution principle}\label{App:Bbowen}

Let $(X,d)$ be a compact metric space, $f\colon X\to X$ be a continuous map, and $Y\subset X$. We recall the definition of topological entropy of $f$ on $Y$. For that, we rely on \cite{Pes:97}. 
We also introduce its capacitive versions $\underline{Ch}_{\rm top}$ and $\overline{Ch}_{\rm top}$.

Recall the definition of a set of $(n,\e)$-separated points in \eqref{defseparated} and the definition of a Bowen ball in \eqref{defBowBal}.
Denote by $r_n(Y,\varepsilon)$ the smallest cardinality of a family of sets $\{B_n(x_i,\varepsilon)\}_i$ whose union covers $Y$ and denote by $s_n(Y,\e)$ the maximum cardinality of  a $(n,\varepsilon)$-separated set in $Y$. Recall that $r_n(Y,\varepsilon)\le s_n(Y,\varepsilon)$. Given $h\ge0$ let 
\[\begin{split}
	 m_{\varepsilon,h}(Y)
	&\eqdef \lim_{N\to\infty}  \Big( \inf\Big\{\sum_ie^{-hn_i}\colon
	Y\subset\bigcup_iB_{n_i}(x_i,\varepsilon),n_i\ge N\text{ for every }i\Big\} \Big),
	\\
	\underline{Cm}_{\varepsilon,h}(Y)
	&\eqdef \liminf_{n\to\infty}~ 
	e^{-nh}\cdot	r_n(Y,\varepsilon),\quad
	\overline{Cm}_{\varepsilon,h}(Y)
	\eqdef \limsup_{n\to\infty}
	e^{-nh}\cdot s_n(Y,\varepsilon).
\end{split}\]
Fixing $\varepsilon$, this 
 function jumps from $\infty$ to $0$ at a unique critical parameter. Let
\begin{equation}\label{eqjin}\begin{split}
 h_{\rm top}(f,Y,\varepsilon)
	&\eqdef \inf\big\{h\colon  m_{\varepsilon,h}(Y)=0\big\},\\
	\overline{Ch}_{\rm top}(f,Y,\varepsilon)
	&\eqdef \inf\big\{h\colon \underline{Cm}_{\varepsilon,h}(Y)=0\big\},\quad
	\underline{Ch}_{\rm top}(f,Y,\varepsilon)
	\eqdef \inf\big\{h\colon \overline{Cm}_{\varepsilon,h}(Y)=0\big\}.
\end{split}\end{equation}

\begin{lemma}\label{lemmono}
	For every $Y\subset X$, $h>0$, and $0<\varepsilon'<\varepsilon$, it holds $m_{\varepsilon,h}(Y)\le m_{\varepsilon',h}(Y)$. Hence, in particular, $ h_{\rm top}(f,Y,\varepsilon)\le   h_{\rm top}(f,Y,\varepsilon')$.
\end{lemma}

\begin{proof}
	Note that if $Y\subset \bigcup_iB_{n_i}(x_i,\varepsilon')$, then also $Y\subset\bigcup_iB_{n_i}(x_i,\varepsilon)$. Thus,
\[
	\inf\Big\{\sum_ie^{-hn_i}\colon
	Y\subset\bigcup_iB_{n_i}(x_i,\varepsilon),n_i\ge N\Big\}
	\le \inf\Big\{\sum_ie^{-hn_i}\colon
	Y\subset\bigcup_iB_{n_i}(x_i,\varepsilon'),n_i\ge N\Big\}.
\]	
Taking the limit $N\to\infty$, we get $m_{\varepsilon,h}(Y)\le m_{\varepsilon',h}(Y)$. This proves the lemma.
\end{proof}

Denote by $\widehat h_{\rm top}(f,Y)$ the \emph{topological entropy} of $f$ on $Y$ as defined in \cite{Bow:73}.

\begin{lemma}[{\cite[Chapter 4 Section 11]{Pes:97}}]\label{lem:entropy}
	For every $Y\subset X$, one has that
	\begin{itemize}
		\item[(1)] the following limit  exists:
		\[\begin{split}
			 h_{\rm top}(f,Y)
			&\eqdef\lim_{\varepsilon\to0}  h_{\rm top}(f,Y,\varepsilon).
			,\\
			\overline{Ch}_{\rm top}(f,Y)
			&\eqdef\lim_{\varepsilon\to0} \overline{Ch}_{\rm top}(f,Y,\varepsilon),\quad
			\underline{Ch}_{\rm top}(f,Y)
			\eqdef\lim_{\varepsilon\to0} \underline{Ch}_{\rm top}(f,Y,\varepsilon).
		\end{split}\]
		\item[(2)]  
		$\overline{Ch}_{\rm top}(f,Y)\ge\underline{Ch}_{\rm top}(f,Y)\ge
		\widehat  h_{\rm top}(f,Y)
		=  h_{\rm top}(f,Y)$.
	\end{itemize}
\end{lemma}

\begin{remark}\label{rem:monent}
We recall some further properties which are relevant in our context:
\begin{itemize}
\item[(E1)] (monotonicity) if $Z\subset Y\subset X$, then $ h_{\rm top}(f,Z)\le  h_{\rm top}(f,Y)$, 
\item[(E2)] (countable stability) $ h_{\rm top}(f,\bigcup_iY_i)=\sup_i  h_{\rm top}(f,Y_i)$.
\end{itemize}
\end{remark}

The following can be considered as an analogy with the classical \emph{Mass Distribution Principle} to estimate Hausdorff dimension (see, for example, \cite[Chapter 2]{Pes:97}). For completeness, we add its short proof.

\begin{theorem}[{Entropy Distribution Principle,  \cite[Theorem 3.6]{TakVer:03}}]\label{thm.entropy-distribution-principle}
	Let $f$ be a continuous map on a metric space $X$ and $Y\subset X$. Assume that there exist constants $h\geq 0$, $\delta>0$, and a Borel  probability measure $\mu_\delta$ such that 
\begin{itemize}
\item $\mu_\delta(Y)>0$;
\item there is $n_0\in\bN$ such that $\mu_\delta(B_n(x,\delta))\leq e^{-nh}$ for every $n\ge n_0$ and every Bowen ball $B_n(x,\delta)$ intersecting the set $Y$.
\end{itemize}
Then $ h_{\rm top}(f,Y,\delta)\geq h.$
\end{theorem}

\begin{proof}
By definition,  it suffices to show that $\widehat m_{\delta,h}(Y)>0$. Consider any countable family of Bowen balls $\{ B_{n_i}(x_i,\delta)\}_i$, $n_i\ge n_0$, which covers $Y$. Up to removing some Bowen balls, one can assume that each Bowen ball $B_{n_i}(x_i,\delta)$ intersects $Y$. 
Then 
\[
	0<\mu_\delta(Y)
	\leq \mu_\delta\big(\bigcup_i B_{n_i}(x_i,\delta)\big)
	\leq \sum_ie^{-n_i h}.
\]
Thus, 
\[
	  m_{\delta, h}(Y)
	=\lim_{N\to\infty} \inf\Big\{\sum_ie^{-hn_i}\colon Y\subset\bigcup_iB_{n_i}(x_i,\delta),n_i\ge N\text{ for every }i\Big\}\geq  \mu_\delta(Y)>0.
\]
This implies $ h_{\rm top}(f,Y,\delta)\geq h.$
\end{proof}

\subsection{``Stable shearing'' preserves topological entropy}

By definition, the concept of entropy only relies on the \emph{forward} iterates of a set, and, indeed, one can show that for a partially hyperbolic diffeomorphism, if one ``shears'' a set along its strong stable leaves, the entropy is unchanged. Let us make this more precise. Given a subset  $Y\subset M$ and $\ell>0$, in analogy to our notation in \eqref{defWgamma}, let 
\[
	\cW^\ss(Y,\ell)\eqdef \bigcup_{Y\in Y}\cW^\ss(Y,\ell).
\] 

\begin{theorem}\label{theop.entropy-in-cu}
Let $f$ be a $C^1$-partially hyperbolic diffeomorphism. Assume that $Y,Y^\prime\subset M$ are two subsets and $\ell>0$ such that $Y^\prime\subset \cW^\ss(Y,\ell)$ and $Y\subset \cW^\ss(Y^\prime,\ell)$. Then
\[
	h_{\rm top}(f,Y)=h_{\rm top}(f,Y^\prime).
\]
\end{theorem}

\proof
As $Y^\prime\subset \cW^\ss(Y,\ell)$ and $Y\subset \cW^\ss(Y^\prime,\ell)$, by monotonicity of entropy, it suffices to show that $h_{\rm top}(f,Y)\geq h_{\rm top}(f,\cW^\ss(Y,\ell))$.
	
By the uniform contraction of $f$ along the strong stable bundle, for every $\e>0,$ there is $k_\e\in\bN$  such that 
\begin{equation}\label{eqcontractionss}
	f^{k}(\cW^\ss(x,\ell))\subset \cW^\ss(f^{k}(x),{\e/2})
	\quad\text{ for every $x\in M$ and $k\geq k_\e$.}
\end{equation}	
	
Given any $\e>0$ and $N\in\bN$, fix now a countable collection of Bowen balls $\{B_{n_i}(x_i,\e)\}$ whose union covers $Y$ and which  satisfy $n_i\geq N+k_\e$ for each $i$.  

\begin{claim}\label{c.cover-under-iterates}
		$f^{k_\e}(\cW^\ss(Y,\ell))\subset \bigcup_{i} B_{n_i-k_\e}\big(f^{k_\e}(x_i),3\e/2\big)$. 
\end{claim}
	
\proof
For every $x\in\cW^\ss(Y,\ell)$, there exists $y\in Y$ such that $x\in\cW^\ss(y,\ell)$. Since there exists $x_i$ such that $y\in B_{n_i}(x_i,\e)$, thus  for every $n_i-k_\e>j\geq 0$, one has that \[d(f^{j+k_\e}(x),f^{j+k_\e}(x_i))\leq d(f^{j+k_\e}(x),f^{j+k_\e}(y))+d(f^{j+k_\e}(y),f^{j+k_\e}(x_i))<3\e/2. \] This proves the claim.
\endproof

Fixing some quantifiers, let
\begin{equation}\label{deflargestDf}
	C\eqdef\sup_{x\in M}\big\{\|Df(x)\|,\|Df^{-1}(x)\|\big\}.
\end{equation}
The following fact is straightforward to prove.

\begin{claim}\label{clfact}
There is a constant $c>0$ such that for every $x\in M$, there are $c\cdot C^{k_\e\cdot\dim(M)}$-many points which are $\e\cdot C^{-k_\e}$-dense in $B(x,3\e/2)$.  
\end{claim}
 
\begin{claim}\label{someclaimabove}
There exists a collection of sets $\{Y_i\}$ such that 
\[
	\cW^\ss(Y,\ell)\subset \bigcup_{i} \bigcup_{y\in Y_i} B_{n_i}(f^{-k_\e}(y),3\e)
	\quad\text{ and }\quad
	\card(Y_i)\le c\cdot C^{k_\e\cdot\dim(M)}.
\]	 
\end{claim}

\proof
Note that for each $x_i$, one has 
\[
	B_{n_i-k_\e}\big(f^{k_\e}(x_i),3\e/2\big)\subset B\big(f^{k_\e}(x_i),3\e/2\big).
\]	 
Hence, by Claim \ref{c.cover-under-iterates},
\[
	\cW^\ss(Y,\ell)
	\subset\bigcup_if^{-k_\e}\big(B_{n_i-k_\e}(f^{k_\e}(x_i),3\e/2)\big)
	\subset\bigcup_if^{-k_\e}\big(B(f^{k_\e}(x_i),3\e/2)\big).
\]
By Claim \ref{clfact}, for each $i$, we can choose  a finite set of points $Y_i\subset B(f^{k_\e}(x_i),3\e/2)$ which is $\e\cdot C^{-k_\e}$-dense in $B(f^{k_\e}(x_i),3\e/2)$ and has cardinality at most $c\cdot C^{k_\e\cdot\dim(M)}$. Without loss of generality, we can assume that $Y_i\subset B_{n_i-k_\e}(f^{k_\e}(x_i),3\e/2)$.

Now fix $x\in \cW^\ss(Y,\ell)$,  then there exists $x_i$ such that 
\begin{equation}\label{choisess}
	f^{k_\e}(x)\in  B_{n_i-k_\e}(f^{k_\e}(x_i),3\e/2).
\end{equation}
Then take $y\in Y_i$ such that $d(f^{k_\e}(x),y)<\e\cdot C^{-k_\e}.$ Now, we have the following estimates:
By Equation \eqref{deflargestDf},  for every  $0\leq j\leq k_\e$, one has
\[
	d(f^{-j}(y),f^{k_\e-j}(x))
	\le C^jd(y,f^{k_\e}(x))
	\le C^j\cdot \e\cdot C^{-k_\e}
	\le \e.
\] 
Moreover, as $y\in Y_i\subset  B_{n_i-k_\e}(f^{k_\e}(x_i),3\e/2)$ and $f^{k_\e}(x)\in  B_{n_i-k_\e}(f^{k_\e}(x_i),3\e/2)$, for $0\leq j<n_i-k_\e$, one has 
\[\begin{split}
	d\big(f^j(y),f^{j+k_\e}(x)\big)
	\leq d\big(f^j(y),f^{j+k_\e}(x_i)\big)+d\big(f^{j+k_\e}(x_i),f^{j+k_\e}(x)\big)
	\le 3\e/2+3\e/2
	= 3\e.
\end{split}\]
To summarize, one has
\[
	d(f^{j-k_\e}(y),f^j(x))
	\le3\e,\textrm{ for $0\leq j<n_i$},
\]
proving the claim.	
\endproof 
	
For every $h>h_{\rm top}(f,Y)$, 
by definition, there exists $\varepsilon_0>0$ such that for every $\e<\varepsilon_0$, one has $ m_{\e,h}(Y)=0$. By definition, we have
\[
 m_{3\e,h}(\cW^\ss(Y,\ell))
=\lim_{N\to\infty} \inf\Big\{\sum_je^{-hn_j}\colon
		\cW^\ss(Y,\ell)\subset\bigcup_jB_{n_j}(y_j,3\varepsilon),n_j\ge N+k_\e\text{ for every }j\Big\}.	
\]
Given now any countable cover $\{B_{n_i}(x_i,\e)\}$ of $Y$, by Claim \ref{someclaimabove}, we get a particular cover of $\cW^\ss(Y,\ell)$ and thus we can estimate as follows:
\[\begin{split}
	& m_{3\e,h}(\cW^\ss(Y,\ell))\\
	&\leq \lim_{N\to\infty} \inf\Big\{\sum_i\sum_{y\in Y_i}e^{-hn_i}\colon
		 Y\subset\bigcup_iB_{n_i}(x_i,\varepsilon),n_i\ge N+k_\e\text{ for every }i\Big\}\\
	&= \lim_{N\to\infty} \inf\Big\{\sum_i\card(Y_i)e^{-hn_i}\colon
		 Y\subset\bigcup_iB_{n_i}(x_i,\varepsilon),n_i\ge N+k_\e\text{ for every }i\Big\}\\
	&\leq c\cdot C^{k_\e\cdot\dim(M)}\cdot
		 \lim_{N\to\infty} \inf\Big\{ \sum_i e^{-hn_i}\colon
		 Y\subset\bigcup_iB_{n_i}(x_i,\varepsilon),n_i\ge N+k_\e\text{ for every }i\Big\}\\
	&=c\cdot C^{k_\e\cdot\dim(M)}\cdot  m_{\e,h}(Y)
		 =0.
\end{split}\]
	
As $\e\in(0,\varepsilon_0)$ is arbitrary, we get $ h_{\rm top}(f,\cW^\ss(Y,\ell))\leq h$. As $h>h_{\rm top}(f,Y)$ is arbitrary, this implies that $h_{\rm top}(f,\cW^\ss(Y,\ell))\leq h_{\rm top}(f,Y).$
\endproof


\begin{thebibliography}{RHRHTU12}

\bibitem[AV10]{AviVia:10}
A.~Avila and M.~Viana.
\newblock Extremal {Lyapunov} exponents: an invariance principle and
  applications.
\newblock {\em Invent. Math.}, 181(1):115--178, 2010.

\bibitem[BBD16]{BocBonDia:16}
J.~Bochi, Ch. Bonatti, and L.~J. D\'{\i}az.
\newblock Robust criterion for the existence of nonhyperbolic ergodic measures.
\newblock {\em Comm. Math. Phys.}, 344(3):751--795, 2016.

\bibitem[BBI04]{BriBurIva:04}
M.~Brin, D.~Burago, and S.~Ivanov.
\newblock On partially hyperbolic diffeomorphisms of 3-manifolds with
  commutative fundamental group.
\newblock In {\em Modern dynamical systems and applications. Dedicated to
  Anatole Katok on his 60th birthday}, pages 307--312. Cambridge: Cambridge
  University Press, 2004.

\bibitem[BC04]{BonCro:04}
Ch. Bonatti and S.~Crovisier.
\newblock R\'{e}currence et g\'{e}n\'{e}ricit\'{e}.
\newblock {\em Invent. Math.}, 158(1):33--104, 2004.

\bibitem[BCDW16]{BonCroDiaWil:16}
Ch. Bonatti, S.~Crovisier, L.~J. D{\'{\i}}az, and A.~Wilkinson.
\newblock What is {{\(\ldots a\)}} blender?
\newblock {\em Notices Am. Math. Soc.}, 63(10):1175--1178, 2016.

\bibitem[BCS22]{BuzCroSar:22}
J.~Buzzi, S.~Crovisier, and O.~Sarig.
\newblock Measures of maximal entropy for surface diffeomorphisms.
\newblock {\em Ann. Math. (2)}, 195(2):421--508, 2022.

\bibitem[BD96]{BonDia:96}
Ch. Bonatti and L.~J. D\'{\i}az.
\newblock Persistent nonhyperbolic transitive diffeomorphisms.
\newblock {\em Ann. of Math. (2)}, 143(2):357--396, 1996.

\bibitem[BD08]{BonDia:08}
Ch. Bonatti and L.~J. D\'{\i}az.
\newblock Robust heterodimensional cycles and {$C^1$}-generic dynamics.
\newblock {\em J. Inst. Math. Jussieu}, 7(3):469--525, 2008.

\bibitem[BD12]{BonDia:12}
Ch. Bonatti and L.~J. D\'{\i}az.
\newblock Abundance of {$C^1$}-robust homoclinic tangencies.
\newblock {\em Trans. Amer. Math. Soc.}, 364(10):5111--5148, 2012.

\bibitem[BDU02]{BonDiaUre:02}
Ch. Bonatti, L.~J. D\'{\i}az, and R.~Ures.
\newblock Minimality of strong stable and unstable foliations for partially
  hyperbolic diffeomorphisms.
\newblock {\em J. Inst. Math. Jussieu}, 1(4):513--541, 2002.

\bibitem[BDV05]{BonDiaVia:05}
Ch. Bonatti, L.~J. D{\'{\i}}az, and M.~Viana.
\newblock {\em Dynamics beyond uniform hyperbolicity. {A} global geometric and
  probabilistic perspective}, volume 102 of {\em Encycl. Math. Sci.}
\newblock Berlin: Springer, 2005.

\bibitem[BFSV12]{BuzFisSamVas:12}
J.~Buzzi, T.~Fisher, M.~Sambarino, and C.~V\'{a}squez.
\newblock Maximal entropy measures for certain partially hyperbolic, derived
  from {A}nosov systems.
\newblock {\em Ergodic Theory Dynam. Systems}, 32(1):63--79, 2012.

\bibitem[BFT22]{BuzFisTah:22}
J.~Buzzi, T.~Fisher, and A.~Tahzibi.
\newblock A dichotomy for measures of maximal entropy near time-one maps of
  transitive {Anosov} flows.
\newblock {\em Ann. Sci. {\'E}c. Norm. Sup{\'e}r. (4)}, 55(4):969--1002, 2022.

\bibitem[BGHP20]{BonGogHamPot:20}
Ch. Bonatti, A.~Gogolev, A.~Hammerlindl, and R.~Potrie.
\newblock Anomalous partially hyperbolic diffeomorphisms {III}: {A}bundance and
  incoherence.
\newblock {\em Geom. Topol.}, 24(4):1751--1790, 2020.

\bibitem[B72]{Bow:72}
R.~Bowen.
\newblock Entropy-expansive maps.
\newblock {\em Trans. Amer. Math. Soc.}, 164:323--331, 1972.

\bibitem[B73]{Bow:73}
R.~Bowen.
\newblock Topological entropy for noncompact sets.
\newblock {\em Trans. Amer. Math. Soc.}, 184:125--136, 1973.

\bibitem[B08]{Bow:08}
R.~Bowen.
\newblock {\em Equilibrium states and the ergodic theory of {A}nosov
  diffeomorphisms}, volume 470 of {\em Lecture Notes in Mathematics}.
\newblock Springer-Verlag, Berlin, revised edition, 2008.
\newblock With a preface by David Ruelle, Edited by Jean-Ren\'{e} Chazottes.

\bibitem[BP74]{BriPes:74}
M.~I. Brin and Ja.~B. Pesin.
\newblock Partially hyperbolic dynamical systems.
\newblock {\em Izv. Akad. Nauk SSSR Ser. Mat.}, 38:170--212, 1974.

\bibitem[BZ19]{BonZhan:19}
Ch. Bonatti and J.~Zhang.
\newblock Periodic measures and partially hyperbolic homoclinic classes.
\newblock {\em Trans. Amer. Math. Soc.}, 372(2):755--802, 2019.

\bibitem[CP]{CroPol:22}
S.~Crovisier and M.~Poletti.
\newblock Invariance principle and non-compact center foliations.
\newblock arXiv:2210.14989.

\bibitem[C11]{Cro:11}
S.~Crovisier.
\newblock Partial hyperbolicity far from homoclinic bifurcations.
\newblock {\em Adv. Math.}, 226(1):673--726, 2011.

\bibitem[CY05]{CowYou:05}
W.~Cowieson and L.-S. Young.
\newblock {SRB} measures as zero-noise limits.
\newblock {\em Ergodic Theory Dyn. Syst.}, 25(4):1115--1138, 2005.

\bibitem[DFPV12]{DiFiPaVi:12}
L.~J. D\'{i}az, T.~Fisher, M.~J. Pacifico, and J.~L. Vieitez.
\newblock Entropy-expansiveness for partially hyperbolic diffeomorphisms.
\newblock {\em Discrete Contin. Dyn. Syst.}, 32(12):4195--4207, 2012.

\bibitem[DGR17]{DiaGelRam:17}
L.~J. D\'{\i}az, K.~Gelfert, and M.~Rams.
\newblock Nonhyperbolic step skew-products: ergodic approximation.
\newblock {\em Ann. Inst. H. Poincar\'{e} Anal. Non Lin\'{e}aire},
  34(6):1561--1598, 2017.

\bibitem[DGR19]{DiaGelRam:19}
L.~J. D\'{\i}az, K.~Gelfert, and M.~Rams.
\newblock Entropy spectrum of {L}yapunov exponents for nonhyperbolic step
  skew-products and elliptic cocycles.
\newblock {\em Comm. Math. Phys.}, 367(2):351--416, 2019.

\bibitem[DGS20]{DiaGelSan:20}
L.~J. D\'{\i}az, K.~Gelfert, and B.~Santiago.
\newblock Weak{$*$} and entropy approximation of nonhyperbolic measures: a
  geometrical approach.
\newblock {\em Math. Proc. Cambridge Philos. Soc.}, 169(3):507--545, 2020.

\bibitem[G16]{Gel:16}
K.~Gelfert.
\newblock Horseshoes for diffeomorphisms preserving hyperbolic measures.
\newblock {\em Math. Z.}, 283(3-4):685--701, 2016.

\bibitem[GIKN05]{GorIlyKleNal:05}
A.~S. Gorodetski, Yu.~S. Ilyashenko, V.~A. Kleptsyn, and M.~B. Nalsky.
\newblock Nonremovable zero {Lyapunov} exponent.
\newblock {\em Funct. Anal. Appl.}, 39(1):21--30, 2005.

\bibitem[G07]{Gou:07}
N.~Gourmelon.
\newblock Adapted metrics for dominated splittings.
\newblock {\em Ergodic Theory Dynam. Systems}, 27(6):1839--1849, 2007.

\bibitem[GP17]{GorPes:17}
A.~Gorodetski and Ya. Pesin.
\newblock Path connectedness and entropy density of the space of hyperbolic
  ergodic measures.
\newblock In {\em Modern theory of dynamical systems. A tribute to Dmitry
  Victorovich Anosov}, pages 111--121. Providence, RI: American Mathematical
  Society (AMS), 2017.

\bibitem[H97]{Hay:97}
S.~Hayashi.
\newblock Connecting invariant manifolds and the solution of the {$C^1$}
  stability and {$\Omega$}-stability conjectures for flows.
\newblock {\em Ann. of Math. (2)}, 145(1):81--137, 1997.

\bibitem[HPS77]{HirPugShu:77}
M.~W. Hirsch, C.~C. Pugh, and M.~Shub.
\newblock {\em Invariant manifolds}.
\newblock Lecture Notes in Mathematics, Vol. 583. Springer-Verlag, Berlin-New
  York, 1977.

\bibitem[J19]{Jen:19}
O.~Jenkinson.
\newblock Ergodic optimization in dynamical systems.
\newblock {\em Ergodic Theory Dyn. Syst.}, 39(10):2593--2618, 2019.

\bibitem[KH95]{KatHas:95}
A.~Katok and B.~Hasselblatt.
\newblock {\em Introduction to the modern theory of dynamical systems},
  volume~54 of {\em Encyclopedia of Mathematics and its Applications}.
\newblock Cambridge University Press, Cambridge, 1995.
\newblock With a supplementary chapter by Katok and Leonardo Mendoza.

\bibitem[L84]{Led:84}
F.~Ledrappier.
\newblock Propri{\'e}t{\'e}s ergodiques des mesures de {Sina{\"{\i}}}.
\newblock {\em Publ. Math., Inst. Hautes {\'E}tud. Sci.}, 59:163--188, 1984.

\bibitem[LVY13]{LiViYa:13}
G.~Liao, M.~Viana, and J.~Yang.
\newblock The entropy conjecture for diffeomorphisms away from tangencies.
\newblock {\em J. Eur. Math. Soc. (JEMS)}, 15(6):2043--2060, 2013.

\bibitem[LY85a]{LedYou1:85}
F.~Ledrappier and L.-S. Young.
\newblock The metric entropy of diffeomorphisms. {I}: {Characterization} of
  measures satisfying {Pesin}'s entropy formula.
\newblock {\em Ann. Math. (2)}, 122:509--539, 1985.

\bibitem[LY85b]{LedYou2:85}
F.~Ledrappier and L.-S. Young.
\newblock The metric entropy of diffeomorphisms. {II}: {Relations} between
  entropy, exponents and dimension.
\newblock {\em Ann. Math. (2)}, 122:540--574, 1985.

\bibitem[M78]{Man:78}
R.~Ma\~{n}\'{e}.
\newblock Contributions to the stability conjecture.
\newblock {\em Topology}, 17(4):383--396, 1978.

\bibitem[O68]{Ose:68}
V.~I. Oseledec.
\newblock A multiplicative ergodic theorem. {C}haracteristic {L}japunov,
  exponents of dynamical systems.
\newblock {\em Trudy Moskov. Mat. Ob\v{s}\v{c}.}, 19:179--210, 1968.

\bibitem[O18]{BenOva:16}
S.~B. Ovadia.
\newblock Symbolic dynamics for non-uniformly hyperbolic diffeomorphisms of
  compact smooth manifolds.
\newblock {\em J. Mod. Dyn.}, 13:43--113, 2018.

\bibitem[P97]{Pes:97}
Ya.~B. Pesin.
\newblock {\em Dimension theory in dynamical systems}.
\newblock Chicago Lectures in Mathematics. University of Chicago Press,
  Chicago, IL, 1997.
\newblock Contemporary views and applications.

\bibitem[PS97]{PugShu:97}
Ch. Pugh and M.~Shub.
\newblock Stably ergodic dynamical systems and partial hyperbolicity.
\newblock {\em J. Complexity}, 13(1):125--179, 1997.

\bibitem[PS06]{PujSam:06}
E.~Pujals and M.~Sambarino.
\newblock A sufficient condition for robustly minimal foliations.
\newblock {\em Ergodic Theory Dynam. Systems}, 26(1):281--289, 2006.

\bibitem[PS07]{PfiSul:07}
C.-E. Pfister and W.~G. Sullivan.
\newblock On the topological entropy of saturated sets.
\newblock {\em Ergodic Theory Dynam. Systems}, 27(3):929--956, 2007.

\bibitem[PW97]{PesWei:97}
Ya.~B. Pesin and H.~Weiss.
\newblock The multifractal analysis of {G}ibbs measures: motivation,
  mathematical foundation, and examples.
\newblock {\em Chaos}, 7(1):89--106, 1997.

\bibitem[RRTU12]{RodRodTahUre:12}
F.~Rodriguez~Hertz, M.~A. Rodriguez~Hertz, A.~Tahzibi, and R.~Ures.
\newblock Maximizing measures for partially hyperbolic systems with compact
  center leaves.
\newblock {\em Ergodic Theory Dynam. Systems}, 32(2):825--839, 2012.

\bibitem[RRU07]{RodRodUre:07}
F.~Rodriguez~Hertz, M.~A. Rodriguez~Hertz, and R.~Ures.
\newblock Some results on the integrability of the center bundle for partially
  hyperbolic diffeomorphisms.
\newblock In {\em Partially hyperbolic dynamics, laminations, and
  {T}eichm\"{u}ller flow}, volume~51 of {\em Fields Inst. Commun.}, pages
  103--109. Amer. Math. Soc., Providence, RI, 2007.

\bibitem[RUY22]{RodUreYan:22}
J.~Rodriguez~Hertz, R.~Ures, and J.~Yang.
\newblock Robust minimality of strong foliations for da diffeomorphisms:
  {{\(cu\)}}-volume expansion and new examples.
\newblock {\em Trans. Am. Math. Soc.}, 375(6):4333--4367, 2022.

\bibitem[R04]{Rue:04}
D.~Ruelle.
\newblock {\em Thermodynamic formalism. {The} mathematical structures of
  equilibrium statistical mechanics.}
\newblock Camb. Math. Libr. Cambridge: Cambridge University Press., 2nd edition
  edition, 2004.

\bibitem[S71]{Shu:71}
M.~Shub.
\newblock Topological transitive diffeomorphisms on $\mathbb{T}^4$.
\newblock In {\em Proc. Sympos. Diff. Eq. Dyn. Syst., Warwick, 1968--1969},
  Lecture Notes in Mathematics, Vol. 206, pages 39--40. Springer-Verlag,
  Berlin-Heidelberg-New York, 1971.

\bibitem[S74]{Sig:74}
K.~Sigmund.
\newblock On dynamical systems with the specification property.
\newblock {\em Trans. Amer. Math. Soc.}, 190:285--299, 1974.

\bibitem[T21]{Tah:21}
A. Tahzibi.
\newblock Unstable entropy in smooth ergodic theory.
\newblock {\em Nonlinearity}, 34(8):r75--r118, 2021.

\bibitem[TV03]{TakVer:03}
F.~Takens and E.~Verbitskiy.
\newblock On the variational principle for the topological entropy of certain
  non-compact sets.
\newblock {\em Ergodic Theory Dynam. Systems}, 23(1):317--348, 2003.

\bibitem[TY19]{TahYan:19}
A.~Tahzibi and J.~Yang.
\newblock Invariance principle and rigidity of high entropy measures.
\newblock {\em Trans. Amer. Math. Soc.}, 371(2):1231--1251, 2019.

\bibitem[U12]{Ure:12}
R.~Ures.
\newblock Intrinsic ergodicity of partially hyperbolic diffeomorphisms with a
  hyperbolic linear part.
\newblock {\em Proc. Am. Math. Soc.}, 140(6):1973--1985, 2012.

\bibitem[UVY21]{UreViaYan:21}
R.~Ures, M.~Viana, and J.~Yang.
\newblock Maximal entropy measures of diffeomorphisms of circle fiber bundles.
\newblock {\em J. Lond. Math. Soc., II. Ser.}, 103(3):1016--1034, 2021.

\bibitem[W66]{Wijsman:66}
R.~A. Wijsman.
\newblock Convergence of sequences of convex sets, cones and functions. {II}.
\newblock {\em Trans. Am. Math. Soc.}, 123:32--45, 1966.

\bibitem[WX00]{WX:00}
L.~Wen and Z.~Xia.
\newblock {{\(C^1\)}} connecting lemmas.
\newblock {\em Trans. Am. Math. Soc.}, 352(11):5213--5230, 2000.

\bibitem[Y21]{Yan:21}
J.~Yang.
\newblock Entropy along expanding foliations.
\newblock {\em Adv. Math.}, 389:39, 2021.
\newblock Id/No 107893.

\bibitem[YZ20]{YanZha:20}
D.~Yang and J.~Zhang.
\newblock Non-hyperbolic ergodic measures and horseshoes in partially
  hyperbolic homoclinic classes.
\newblock {\em J. Inst. Math. Jussieu}, 19(5):1765--1792, 2020.

\end{thebibliography}
\end{document}